\documentclass[12pt,a4paper,reqno]{article}
\usepackage[T1]{fontenc}
\usepackage[utf8]{inputenc}
\usepackage{lmodern,amsmath,amssymb,amsthm,mathtools}
\let\standardintop\intop
\let\standardint\int
\let\standardiint\iint
\usepackage{esint}
\let\intop\standardintop
\let\int\standardint
\let\iint\standardiint
\usepackage[a4paper,left=25mm,top=28.4mm,textwidth=160mm,textheight=225mm,
headheight=0pt,headsep=0pt,footskip=12mm]{geometry}
\usepackage{microtype,enumitem,xcolor,indentfirst,needspace}
\usepackage[colorlinks=true,linkcolor=red,anchorcolor=green,
citecolor=blue,urlcolor=blue,bookmarksopen=true,
bookmarksnumbered=true,bookmarksdepth=2]{hyperref}
\hypersetup{pdftitle={The fractional Hénon-Lane-Emden conjecture and some critical nonexistence phenomena},pdfauthor={Kui Li, Haoyang Lu, Zhitao Zhang},bookmarksnumbered=true}
\allowdisplaybreaks[2]
\setlist{itemsep=3pt,topsep=5pt}
\numberwithin{equation}{section}
\newtheorem{theorem}{Theorem}[section]
\newtheorem{proposition}[theorem]{Proposition}
\newtheorem{lemma}[theorem]{Lemma}
\newtheorem{corollary}[theorem]{Corollary}
\theoremstyle{definition}
\newtheorem{definition}[theorem]{Definition}
\theoremstyle{remark}
\newtheorem{remark}[theorem]{Remark}
\newcommand{\R}{\mathbb R}
\newcommand{\Sph}{\mathbb S}

\newcommand{\dist}{\operatorname{dist}}

\newcommand{\avdir}[1]{\fint_{\Sph^{N-1}}#1\,de}

\theoremstyle{plain}
\newtheorem{mainthm}[theorem]{Theorem}
\title{\large\textbf{The fractional H\'enon-Lane-Emden conjecture\\
and some critical nonexistence phenomena}}
\author{Kui Li\thanks{E-mail: \href{mailto:likui@zzu.edu.cn}{\texttt{likui@zzu.edu.cn}}.}~~Haoyang Lu\thanks{E-mail: \href{mailto:luhaoyang25@mails.ucas.ac.cn}{\texttt{luhaoyang25@mails.ucas.ac.cn}}.}~~Zhitao Zhang\thanks{E-mail: \href{mailto:zzt@math.ac.cn}{\texttt{zzt@math.ac.cn}}.}}
\date{}
\begin{document}
\maketitle
\begin{abstract}
We refine the proof method for the H\'enon-Lane-Emden system and extend
it to fractional orders $0<s<1$, proving subcritical nonexistence of
positive classical  solutions. And we find
examples with $0<s<1/2$ and fixed nonzero weights for which existence
and nonexistence occur at different points of the same critical
hyperbola.
\par\smallskip\noindent
\textbf{Keywords:} Fractional H\'enon-Lane-Emden system,
positive entire solutions, critical hyperbola, Liouville theorem,
Kelvin symmetry.
\par\smallskip\noindent
\textbf{AMS Subject Classification (2020):} 35R11, 35B53, 35J91, 45G05.
\end{abstract}

{\hypersetup{linkcolor=black}\tableofcontents}
\clearpage

\section{Introduction}\label{sec:intro}

\subsection{From the H\'enon-Lane-Emden system to fractional orders}

The H\'enon-Lane-Emden system is
\[
\begin{cases}
-\Delta u=|x|^a v^p,\\
-\Delta v=|x|^b u^q,
\end{cases}
\qquad x\in\R^N.
\]
Here $a,b$ are the spatial weights and $p,q$ are the nonlinear powers.
For $N\ge3$, the weighted critical hyperbola is
\[
\frac{N+a}{p+1}+\frac{N+b}{q+1}=N-2.
\]
The problem is to determine when positive entire solutions exist.
For $a=b=0$ this is the Lane-Emden system; for $a=b$, $p=q$,
and $u=v$, it reduces to the scalar H\'enon equation.

The scalar H\'enon equation
$-\Delta u=|x|^a u^p$ was introduced in the study of spherical stellar
systems \cite{henon}. For $a=0$, Gidas-Spruck
\cite{gidas-spruck} proved nonexistence for
$N\ge3$ and $1<p<(N+2)/(N-2)$.
Chen-Li \cite{chen-li} later gave a proof using the Kelvin
transform and moving planes. For $a>-2$, the corresponding critical
exponent is $(N+2+2a)/(N-2)$ when $N\ge3$. Phan-Souplet
\cite{phan-souplet} formulated the weighted scalar conjecture and
proved it for bounded positive solutions in dimension three.
Li-Zhang \cite{li-zhang-monotonicity} obtained a necessary and
sufficient existence condition for positive solutions bounded near
the origin, as part of their study of weighted divergence-form
equations. Their argument uses a monotonicity theorem on the cylinder
$\R\times\Sph^{N-1}$.

For the Lane-Emden system in $N\ge3$, the critical exponent is
replaced by the Sobolev hyperbola
$1/(p+1)+1/(q+1)=(N-2)/N$. Mitidieri \cite{mitidieri}
settled nonexistence in the radial subcritical case and obtained
early nonexistence results for general solutions. Serrin-Zou
\cite{serrin-zou} established the three-dimensional conjecture
under a polynomial growth assumption. The doubling and rescaling
method of Pol\'a\v{c}ik-Quittner-Souplet \cite{pqs} removed
that assumption. Souplet \cite{souplet-four} subsequently proved
the conjecture in dimension four and obtained a further
nonexistence region in higher dimensions, using Rellich-Pohozaev
identities, comparison of the components, spherical Sobolev
inequalities, and feedback estimates.
For $N\ge3$, Lu-Zhang \cite{lu-zhang-disk} established nonexistence
for $p,q>0$, $pq>1$, and $(p,q)$ in a disk tangent to the critical
hyperbola, with the critical point $p=q=(N+2)/(N-2)$ excluded.
Their result requires no growth or decay assumptions.

The spatial weights introduce an additional difficulty because
translations no longer preserve the equations. For the local
H\'enon-Lane-Emden system, Bidaut-V\'eron-Giacomini
\cite[Theorem 1.4]{bidaut-giacomini} characterized positive radial
solutions by the weighted critical hyperbola when $N\ge3$,
$a,b>-2$, $p,q>0$, and $pq>1$. Fazly-Ghoussoub
\cite{fazly-ghoussoub} proved the subcritical nonexistence result
in dimension three for bounded positive solutions with
$a,b\ge0$, $p,q\ge1$, and $pq>1$. Li-Zhang
\cite{li-zhang-local} proved the three-dimensional conjecture
for real weights without boundedness or decay assumptions.
Hang Li \cite{hang-li} obtained further subcritical results for
$a,b>-2$ and $pq>1$ in dimensions four and five, under the additional
restrictions $p,q<4/3$ and $p,q<10/9$, respectively.

Two related extensions are relevant to the treatment of the
weighted origin. Li-Zhang \cite{li-zhang-higher} studied
higher-order Hardy-H\'enon systems by improving regularity at the
origin and establishing an equivalent integral formulation.
Cheng-Li-Zhang \cite{cheng-li-zhang} considered systems with
the weighted diffusion operator
$-\operatorname{div}(|x|^\vartheta\nabla\cdot)$ and proved
Liouville theorems for positive solutions in dimension two and
bounded positive solutions in dimension three, under the
corresponding weighted subcritical condition and admissible
weight assumptions. These results illustrate the roles of
regularity at the origin, scaling, and weighted energy identities
before passing to fractional operators.

Replacing $-\Delta$ by its fractional power gives the system studied here:
\begin{equation}\label{eq:system}
\begin{cases}
(-\Delta)^s u=|x|^a v^p,\\
(-\Delta)^s v=|x|^b u^q,
\end{cases}
\qquad x\in\R^N,\qquad 0<s<1.
\end{equation}
The operator and the solution class are fixed in the next subsection.
For fractional equations and systems, Dou-Zhou \cite{dou-zhou}
established Liouville theorems by moving planes in the integral
formulation. Ma-Li-Zhang \cite{ma-li-zhang} used moving planes
directly for the fractional system with nonnegative weights. In
the notation of \eqref{eq:system}, their nonexistence range includes
\[
a,b\ge0,\qquad
1<p<\frac{N+2s+a}{N-2s},\qquad
1<q<\frac{N+2s+b}{N-2s}.
\]
Li-Yu-Zhang \cite{li-yu-zhang} treated singular weights and
possibly different fractional orders. We use their definition of
classical solutions, their distributional extension across the
origin, and their Riesz representation. Their nonexistence result
\cite[Theorem 1.7]{li-yu-zhang} has the following equal-order form.

For $N\ge2$, $0<s<1$, and $a,b>-2s$, their result excludes
positive classical entire solutions when
\[
0<p\le\frac{N+2s+2a}{N-2s},\qquad
0<q\le\frac{N+2s+2b}{N-2s},
\]
with at least one strict inequality.

The two individual bounds describe a part of the weighted
subcritical region. Theorem \ref{thm:main} treats the full region.

For the scalar fractional H\'enon equation with positive weight,
Barrios-Quaas \cite{barrios-quaas} obtained nonexistence throughout
the scalar subcritical range. They also constructed radial solutions
at the critical exponent. In the range $0<s<1/2$, their construction
requires $0<a<2s(N-1)/(1-2s)$, as stated in
\cite[Theorem 1.3(a)]{barrios-quaas}. For $0<s<1$ and $-2s<a<0$,
\cite[Theorem 1.3(b)]{barrios-quaas} gives a bounded positive radial
weak solution at the critical exponent. For $0<s<1/2$, the construction in Section \ref{sec:construction}
provides classical solutions on the critical diagonal throughout
$-2s<a<2s(N-1)/(1-2s)$.
The introduction of Alarc\'on-Barrios-Quaas
\cite{alarcon-barrios-quaas} records the corresponding large-weight
critical problem as unresolved in that work. The nonexistence result below addresses this range without a radiality assumption. The proof uses the cylinder multiplier formula of
Ao-Chan-DelaTorre-Fontelos-Gonz\'alez-Wei \cite{ao-ode}.

For fractional Lane-Emden systems, Cheng-Gui-Hu \cite{cgh} constructed
positive solutions in the strictly supercritical region for
$p,q>1$. Our existence argument treats $a,b>-2s$, $p,q>0$, and $pq>1$,
including critical equality.
A single radial compactness argument treats all $0<s<1$.
Below one half, the extra inequality excludes the one-dimensional limits
that arise when the radius tends to infinity.
Villavert \cite{villavert} studied related integral systems with coefficient
functions bounded above and below by positive constants. Such
results must be distinguished from existence for the fixed
coefficients in \eqref{eq:system}.

The localized kernel identity and the comparison of level-set intervals
in OpenAI \cite{openai} and Mei-Wei \cite{meiwei} provide the starting point
for the subcritical argument. For the weights, the endpoint identity
$|x|^bu(x)^q=t$ introduces the factor $|x|^{-b/q}$; the short-interval
restriction in Section \ref{sec:pressure} keeps the endpoint and interior
radii comparable. For $2s\le1$, the two differentiated kernel terms must
be added before taking the limit at $x=y$. The difference of the cutoff
vectors supplies the factor $|x-y|$ that restores local integrability.
The estimates near the origin are established in Section \ref{sec:inputs}.

For the strictly supercritical scalar problem, the results of
Hasegawa-Ikoma-Kawakami \cite[Theorem 1.1(ii) and Remark 1.2]{hik-exist}
also give solutions outside the additional inequality in
Theorem \ref{thm:below-half-exist}. That condition is therefore not a
general necessity in the strictly supercritical region.

\subsection{The fractional problem and the main results}

For system \eqref{eq:system}, we consider the equal-order
case $s_1=s_2=s$ in the notation of Li-Yu-Zhang
\cite{li-yu-zhang}.
For the main system results we assume
\begin{equation}\label{eq:parameters-main}
N\ge2,\qquad 0<s<1,\qquad a,b>-2s,\qquad p,q>0,\qquad pq>1.
\end{equation}
In particular, $N>2s$.
We seek an existence criterion for positive entire solutions
continuous at the origin.

We normalize the fractional Laplacian by the Fourier symbol
$|\xi|^{2s}$. For a sufficiently regular function $w$,
\[
(-\Delta)^sw(x)=C(N,s)\lim_{\varepsilon\downarrow0}
\int_{|x-y|>\varepsilon}\frac{w(x)-w(y)}{|x-y|^{N+2s}}\,dy.
\]
Here $C(N,s)>0$ is the fixed normalization constant. This limit is
the principal value, also denoted by $\operatorname{p.v.}$.

To make the integral at large distances finite, we use the class
\[
\mathcal L_{2s}=\left\{w:\R^N\to\R:
\int_{\R^N}\frac{|w(x)|}{1+|x|^{N+2s}}\,dx<\infty\right\}.
\]
We use the definition of positive classical solutions in
Li-Yu-Zhang \cite[Definition 1.3 and (1.6)]{li-yu-zhang}, specialized to
equal fractional orders.

\begin{definition}\label{def:class}
A pair $(u,v)$, defined on all of $\R^N$, is a positive classical
entire solution of \eqref{eq:system} if
\[
u,v\in\mathcal L_{2s}\cap C(\R^N)
\cap C^{1,1}_{\mathrm{loc}}(\R^N\setminus\{0\}),
\]
both components are positive on $\R^N\setminus\{0\}$, and
\eqref{eq:system} holds pointwise there in the principal-value sense.
\end{definition}

The word \emph{entire} includes the origin in the domain of the
functions. Under \eqref{eq:parameters-main}, Lemma \ref{lem:origin}
shows that the equations also hold distributionally on all of
$\R^N$. For negative weights, no finite pointwise identity at the
origin is required. Continuity there excludes singular power solutions.

Definition \ref{def:class} requires neither radial symmetry nor global
boundedness. The integral representation gives $u(0),v(0)>0$;
see Lemma \ref{lem:system-potentials} and \cite[Remark 3.2]{li-yu-zhang}.
Positivity at zero is therefore a consequence of the equations.

We also use the Riesz potential
\begin{equation}\label{eq:riesz-definition}
I_{2s}f(x)=c_{N,s}\int_{\R^N}\frac{f(y)}{|x-y|^{N-2s}}\,dy,
\qquad
c_{N,s}=\frac{\Gamma((N-2s)/2)}{4^s\pi^{N/2}\Gamma(s)},
\end{equation}
where $\Gamma$ is the gamma function. For a nonnegative measurable
function $f$, the potential is defined wherever the integral is finite.
The corresponding integral system is
\begin{equation}\label{eq:integral-system-intro}
\begin{cases}
u=I_{2s}(|x|^av^p),\\
v=I_{2s}(|x|^bu^q),
\end{cases}
\qquad x\in\R^N.
\end{equation}
These identities are required at every point, including the origin,
and both potentials must be finite.
Proposition \ref{thm:representation} proves the equivalence of the
differential and integral formulations in the solution class used here.

The weighted Sobolev critical hyperbola is
\begin{equation}\label{eq:critical-hyperbola-intro}
\frac{N+a}{p+1}+\frac{N+b}{q+1}=N-2s.
\end{equation}
The regions where the left-hand side is greater than or smaller than
$N-2s$ are called subcritical and supercritical, respectively.
We ask whether the critical and supercritical conditions are sufficient
for existence. The answer is affirmative for $1/2\le s<1$, but changes
for $0<s<1/2$.

The results describe three regimes. All nonexistence statements use
Definition \ref{def:class}, with no additional radiality, boundedness,
stability, or finite-energy hypothesis.

\begin{mainthm}\label{thm:main}
Under \eqref{eq:parameters-main}, no positive classical entire solution
exists if
\begin{equation}\label{eq:hyperbola}
\frac{N+a}{p+1}+\frac{N+b}{q+1}>N-2s.
\end{equation}
The same conclusion holds for positive continuous solutions of
the Riesz system \eqref{eq:integral-system-intro}.
\end{mainthm}

\begin{mainthm}\label{thm:small-exist}\label{cor:sharp}
Suppose \eqref{eq:parameters-main} and $1/2\le s<1$.
A positive classical entire solution exists if and only if
\begin{equation}\label{eq:supercritical}
\frac{N+a}{p+1}+\frac{N+b}{q+1}\le N-2s.
\end{equation}
\end{mainthm}

For the small-order results write
$p_S(a)=(N+2s+2a)/(N-2s)$ and $a_*(N,s)=2s(N-1)/(1-2s)$.
Part~(i) below gives a sufficient existence region. Part~(ii) gives an
exact criterion on the critical diagonal, and part~(iii) locates existence
and nonexistence on one fixed critical curve.

\begin{mainthm}\label{thm:small-orders}
Suppose \eqref{eq:parameters-main} and $0<s<1/2$.
\begin{enumerate}[label=(\roman*),ref=\thetheorem(\roman*),leftmargin=*]
\item\label{thm:below-half-exist}
If
\begin{equation}\label{eq:sufficient-window-intro}
\frac{N+a}{p+1}+\frac{N+b}{q+1}\le N-2s,
\qquad \frac1{p+1}+\frac1{q+1}>1-2s,
\end{equation}
there is a positive classical entire solution.

\item\label{thm:critical-criterion}
On the critical diagonal $a=b$ and $p=q=p_S(a)$, a positive
classical entire solution exists if and only if
\begin{equation}\label{eq:critical-range}
-2s<a<a_*(N,s).
\end{equation}
For $a\ge a_*(N,s)$, including equality, no positive classical
entire pair exists.

\item\label{thm:fixed-critical-curve}
Let
\begin{equation}\label{eq:fixed-weights-main}
\begin{gathered}
N=2,\qquad 0<s<\frac14,\qquad -2s<b<0,\\
a=\frac{4s-b}{1-4s},\qquad p_0=\frac{1+4s}{1-4s}.
\end{gathered}
\end{equation}
Keep these $N,s,a,b$ fixed. On the critical curve
\begin{equation}\label{eq:fixed-critical-curve}
\frac{2+a}{p+1}+\frac{2+b}{q+1}=2-2s,
\end{equation}
there is no positive classical entire solution at $(p,q)=(p_0,1)$.
Every point with $p>p_0$ admits a positive classical entire solution. Both weights are nonzero.
\end{enumerate}
\end{mainthm}

\begin{remark}\label{rem:half-order-distinction}\label{rem:same-critical-curve-example}
The critical hyperbola characterizes existence at and above one half.
Below one half, the diagonal criterion and the fixed-curve result
show the additional dependence on the weights and powers.
The family in Theorem \ref{thm:fixed-critical-curve} has $a>0$, $b<0$,
and does not classify the part $p<p_0$.
\end{remark}

Section \ref{sec:inputs} introduces the common notation and proves the
potential representation and preliminary estimates. Section
\ref{sec:nonexistence} proves Theorem \ref{thm:main}. Section
\ref{sec:construction} defines the ball operators and constructs entire
solutions, proving Theorem \ref{thm:small-exist} and the existence part of
Theorem \ref{thm:small-orders}. Section \ref{sec:critical-criterion} proves
the nonexistence statements on the critical diagonal and on a fixed
critical curve.

\section{Notation and preliminary results}
\label{sec:inputs}\label{sec:representation}

\subsection{Notation and kernel estimates}

We write $B_r=B_r(0)$, $t_+=\max\{t,0\}$, and $t_-=\max\{-t,0\}$.
For a measurable set $A$ with $0<|A|<\infty$, the average integral is
$\fint_A f=|A|^{-1}\int_A f$. The symbols $C^\sigma$, for noninteger
$\sigma>0$, denote the usual H\"older spaces.
An integral written without a domain is over $\R^N$ with respect to Lebesgue measure.
The notation $f\asymp g$ means $cg\le f\le Cg$ with positive constants
independent of the variable in the stated range.
In uniform estimates, $C$ denotes a positive constant depending only
on the fixed parameters $N,s,a,b,p,q$ and the chosen cutoff functions.
Its value may change from line to line. Dependence on a solution,
a radius, or an additional parameter is stated when it is allowed.

For an entire pair, finite weighted energy means
\[
\int_{\R^N}\bigl(|x|^bu^{q+1}+|x|^av^{p+1}\bigr)\,dx<\infty.
\]
We do not assume this condition. The same integral on every bounded ball
is finite because $u,v$ are locally bounded and $a,b>-2s>-N$.

Put
\begin{equation}\label{eq:kernel}
K_s(x)=c_{N,s}|x|^{2s-N}\qquad(x\ne0),
\end{equation}
where $c_{N,s}$ is defined in \eqref{eq:riesz-definition}, so that
$I_{2s}f=K_s*f$.

For $N\ge1$, $0<s<\min\{1,N/2\}$, $D>-2s$, and $r>0$, the kernel satisfies
\begin{equation}\label{eq:weighted-kernel-local}
\sup_{x\in B_r}\int_{B_r}K_s(x-y)|y|^D\,dy
\le C_{N,s,D}r^{D+2s}.
\end{equation}
Indeed, at $x=0$ the integral is
$c_{N,s}|\Sph^{N-1}|r^{D+2s}/(D+2s)$.
For $0<|x|=\rho<r$, substitute $y=\rho z$ in the part $|y|<2\rho$:
\[
\int_{B_r\cap B_{2\rho}}K_s(x-y)|y|^D\,dy
\le C\rho^{D+2s}
\int_{B_2}\frac{|z|^D}{|x/\rho-z|^{N-2s}}\,dz
\le C\rho^{D+2s}.
\]
The two singular points $0$ and $x/\rho$ have distance one;
$D>-N$ and $2s>0$ give integrability at each, uniformly under rotations.
On the remaining region, $|x-y|\ge|y|/2$, so
\[
\int_{B_r\setminus B_{2\rho}}K_s(x-y)|y|^D\,dy
\le C\int_{\min\{2\rho,r\}}^r t^{D+2s-1}\,dt
\le Cr^{D+2s}.
\]
Adding these bounds proves \eqref{eq:weighted-kernel-local}.

\Needspace{6\baselineskip}
Write the two right-hand sides as
\[
F(x)=|x|^b u(x)^q,\qquad G(x)=|x|^a v(x)^p\quad(x\ne0),
\qquad F(0)=G(0)=0.
\]
Thus $(-\Delta)^s u=G$ and $(-\Delta)^s v=F$. We pair $F$ with
$u$ and $G$ with $v$, so $Fu=|x|^bu^{q+1}$ and $Gv=|x|^av^{p+1}$
almost everywhere.
The value at a single point does not affect an integral or a
distributional equation. Integrability near zero follows from polar
coordinates: for $D=a$ or $b$,
\[
\begin{aligned}
\int_{B_r}|y|^D\,dy
&=|\Sph^{N-1}|\int_0^r t^{N+D-1}\,dt
=\frac{|\Sph^{N-1}|}{N+D}r^{N+D},\\
\int_{B_r}K_s(y)|y|^D\,dy
&=c_{N,s}|\Sph^{N-1}|\int_0^r t^{2s+D-1}\,dt
=\frac{c_{N,s}|\Sph^{N-1}|}{2s+D}r^{2s+D}\longrightarrow0.
\end{aligned}
\]
Both denominators are positive since $D>-2s>-N$. Multiplication by
the locally bounded factors $u^q,v^p$ preserves these bounds.
No continuity of $F,G$ at zero is required.

The potential formulation is
\[
\begin{cases}
u=K_s*G,\\
v=K_s*F,
\end{cases}
\qquad x\in\R^N.
\]

\subsection{From the system to the exact potentials}

We use the equality case of the following cited lemma.
Its weight condition $a,b>-N$ is satisfied, since
\eqref{eq:parameters-main} gives $N>2s$ and $a,b>-2s$,
hence $a,b>-2s>-N$.

\begin{lemma}[{Li-Yu-Zhang \cite[Lemma 2.1]{li-yu-zhang}}]
\label{lem:origin}
Every positive classical entire solution satisfies
\[
\begin{cases}
\displaystyle\int_{\R^N}u(-\Delta)^s\phi\,dx=\int_{\R^N}G\phi\,dx,\\[0.4em]
\displaystyle\int_{\R^N}v(-\Delta)^s\phi\,dx=\int_{\R^N}F\phi\,dx,
\end{cases}
\qquad\phi\in C_c^\infty(\R^N).
\]
Thus \eqref{eq:system} holds in the sense of distributions on all
of $\R^N$, with no additional term supported at the origin.
\end{lemma}

\Needspace{7\baselineskip}
The following potential formula uses the Riesz normalization
in \eqref{eq:riesz-definition}. Its extension to the origin is justified below.

\begin{lemma}[{Li-Yu-Zhang \cite[Lemma 3.1]{li-yu-zhang}}]
\label{lem:system-potentials}
Every positive classical entire solution satisfies
\begin{equation}\label{eq:exact-potential}
\begin{cases}
u=I_{2s}(|x|^a v^p),\\
v=I_{2s}(|x|^b u^q),
\end{cases}
\qquad x\in\R^N.
\end{equation}
Both potentials are finite and continuous on $\R^N$.
\end{lemma}

The cited result states \eqref{eq:exact-potential} for $x\ne0$.
To include the origin, use local boundedness and $a,b>-2s$.
Substitute $D=a,b$ in \eqref{eq:weighted-kernel-local} and use the
local bounds for $v^p,u^q$ to obtain
\[
\sup_{|x|\le r/2}\int_{B_r}K_s(x-y)(F(y)+G(y))\,dy
\le C_{u,v}(r^{a+2s}+r^{b+2s})\longrightarrow0.
\]
For the source outside $B_r$, fix $0<|x_*|<r/2$. If $|x|<r/2$
and $|y|\ge r$, then
\[
K_s(x-y)\le 3^{N-2s}K_s(x_*-y).
\]
The right-hand kernel times either source is integrable, by the
potential identity at $x_*$. Dominated convergence therefore
gives continuity of this outer contribution at zero. Combining it
with the preceding bound extends both identities to $x=0$.
The positive kernel and the nonzero sources yield $u(0)>0$
and $v(0)>0$.

\subsection{Nonlinear integrals and scaling}

The integral estimates in this and the following two subsections
also apply in dimension one. Their parameter assumptions are
\begin{equation}\label{eq:integral-parameters}
N\ge1,\qquad 0<s<\min\{1,N/2\},\qquad
a,b>-2s,\qquad p,q>0,\qquad pq>1.
\end{equation}
We assume a positive continuous solution of
\eqref{eq:exact-potential} on all of $\R^N$, with both potentials
finite at every point.
For $N\ge2$ this includes the classical system solutions above.
The case $N=1$, necessarily with $s<1/2$, will be used for the
limiting functions in the existence construction.

For $R>0$, define
\begin{equation}\label{eq:dilation}
u_R(x)=R^\alpha u(Rx),\qquad v_R(x)=R^\beta v(Rx),
\end{equation}
where
\begin{equation}\label{eq:scales}
\alpha=\frac{2s(p+1)+a+pb}{pq-1},\qquad
\beta=\frac{2s(q+1)+b+qa}{pq-1}.
\end{equation}
Since $a+2s,b+2s>0$,
\[
\alpha=\frac{(a+2s)+p(b+2s)}{pq-1}>0,\qquad
\beta=\frac{(b+2s)+q(a+2s)}{pq-1}>0,
\]
and $p\beta-\alpha=2s+a$, $q\alpha-\beta=2s+b$.
Substitution $y=R^{-1}z$ in the Riesz integrals gives
\[
\begin{aligned}
I_{2s}(|\,\cdot\,|^a v_R^p)(x)
&=R^{p\beta-a-2s}I_{2s}(|\,\cdot\,|^a v^p)(Rx)=u_R(x),\\
I_{2s}(|\,\cdot\,|^b u_R^q)(x)
&=R^{q\alpha-b-2s}I_{2s}(|\,\cdot\,|^b u^q)(Rx)=v_R(x).
\end{aligned}
\]
Thus $(u_R,v_R)$ solves the same integral system. Moreover,
$F_R(x)=R^{q\alpha-b}F(Rx)$ and
$G_R(x)=R^{p\beta-a}G(Rx)$, hence
\begin{equation}\label{eq:dilated-sources}
F_R(x)=R^{\beta+2s}F(Rx),\qquad
G_R(x)=R^{\alpha+2s}G(Rx).
\end{equation}

For the estimates below, we write that
\begin{equation}\label{eq:gaps-main}
\gamma=\frac{N+a}{p+1}+\frac{N+b}{q+1}-(N-2s),
\qquad S=\frac1{p+1}+\frac1{q+1}.
\end{equation}
The relation between the two above is
\begin{equation}\label{eq:gap}
\gamma=(1-S)(\alpha+\beta+2s-N),
\qquad
1-S=\frac{pq-1}{(p+1)(q+1)}>0.
\end{equation}

\begin{proposition}\label{prop:source}
Under \eqref{eq:integral-parameters}, every positive continuous
solution of \eqref{eq:exact-potential} satisfies, for every $R>0$,
\begin{equation}\label{eq:source-bounds}
\int_{B_R}F\le CR^{N-2s-\beta},\qquad
\int_{B_R}G\le CR^{N-2s-\alpha}.
\end{equation}
Here $C=C(N,s,a,b,p,q)$ is independent of the solution and $R$.

\end{proposition}
For $N\ge2$ and $p,q>1$, these integral estimates for classical solutions are
\cite[Lemma 2.3]{li-yu-zhang}, specialized to equal orders.
That lemma assumes both powers exceed one. The following short
proof also covers positive continuous integral solutions when one
power is at most one.

\begin{proof}
Since $|x-y|\le2R$ for $x,y\in B_R$ and $2s-N<0$,
\eqref{eq:kernel}, $F,G\ge0$, and \eqref{eq:exact-potential} give
\[
\begin{gathered}
K_s(x-y)\ge cR^{2s-N}\qquad(x,y\in B_R,\ x\ne y),\\[0.4em]
\begin{cases}
\displaystyle u(x)\ge\int_{B_R}K_s(x-y)G(y)\,dy
\ge cR^{2s}\fint_{B_R}G\,dy,\\[0.4em]
\displaystyle v(x)\ge\int_{B_R}K_s(x-y)F(y)\,dy
\ge cR^{2s}\fint_{B_R}F\,dy,
\end{cases}
\qquad x\in B_R.
\end{gathered}
\]

Due to $q,p>0$, multiply by the respective weights,
and average over $B_R$, we can get:
\[
\begin{gathered}
\fint_{B_R}|x|^a\,dx=\frac{N}{N+a}R^a,\qquad
\fint_{B_R}|x|^b\,dx=\frac{N}{N+b}R^b,\\[0.5em]
\begin{cases}
\displaystyle\fint_{B_R}F\,dx
=\fint_{B_R}|x|^bu^q\,dx
\ge cR^{b+2sq}\left(\fint_{B_R}G\,dx\right)^q,\\[0.4em]
\displaystyle\fint_{B_R}G\,dx
=\fint_{B_R}|x|^av^p\,dx
\ge cR^{a+2sp}\left(\fint_{B_R}F\,dx\right)^p.
\end{cases}
\end{gathered}
\]
Insert the second average bound into the first and then the first
into the second:
\[
\begin{aligned}
\fint_{B_R}F
&\ge cR^{b+2sq+q(a+2sp)}\left(\fint_{B_R}F\right)^{pq},\\
\fint_{B_R}G
&\ge cR^{a+2sp+p(b+2sq)}\left(\fint_{B_R}G\right)^{pq}.
\end{aligned}
\]
Divide by the positive averages and use \eqref{eq:scales}:
\[
\begin{aligned}
\left(\fint_{B_R}F\,dx\right)^{pq-1}
&\le CR^{-b-2sq-q(a+2sp)}
=CR^{-(pq-1)(2s+\beta)},\\
\left(\fint_{B_R}G\,dx\right)^{pq-1}
&\le CR^{-a-2sp-p(b+2sq)}
=CR^{-(pq-1)(2s+\alpha)}.
\end{aligned}
\]
Since $pq>1$ and $|B_R|=|B_1|R^N$, these are
\eqref{eq:source-bounds}.

For $R\ge1$, \eqref{eq:source-bounds} gives
\[
\begin{aligned}
\alpha>N-2s
&\ \Longrightarrow\ 0<\int_{B_1}G\le\int_{B_R}G
\le CR^{N-2s-\alpha}\longrightarrow0,\\
\beta>N-2s
&\ \Longrightarrow\ 0<\int_{B_1}F\le\int_{B_R}F
\le CR^{N-2s-\beta}\longrightarrow0,
\end{aligned}
\qquad R\to\infty.
\]
Both alternatives are impossible, so $\alpha,\beta\le N-2s$.
\end{proof}

\subsection{Ball averages and weighted integrals}

Proposition \ref{prop:source} bounds the integrals of $F,G$ over
balls. We now use \eqref{eq:exact-potential} to estimate the ball
averages of $u,v$ and their weighted integrals at large distances.

\begin{proposition}\label{prop:inputs}
Assume \eqref{eq:integral-parameters}, and let $(u,v)$ be a positive
continuous solution of \eqref{eq:exact-potential}, with $\alpha,\beta$
as in \eqref{eq:scales}. There is $C=C(N,s,a,b,p,q)>0$,
independent of $(u,v)$, such that for every $R>0$,
\begin{equation}\label{eq:mean-decay}
\fint_{B_R}u(x)\,dx\le CR^{-\alpha},\qquad
\fint_{B_R}v(x)\,dx\le CR^{-\beta}.
\end{equation}
The same constant can be chosen so that
\begin{equation}\label{eq:inputs}
\begin{aligned}
\int_{B_4}(F+G+u+v)\,dx&\le C,\\
\int_{|x-y|\ge\ell}K_s(x-y)(F(y)+G(y))\,dy
&\le C\ell^{2s-N}
\end{aligned}
\end{equation}
for every $x\in B_2$ and $0<\ell\le1$.
For every $\tau>0$, there is $C_\tau=C(N,s,a,b,p,q,\tau)>0$
such that for all $R>0$,
\begin{equation}\label{eq:general-tail}
\int_{|x|>R}\frac{u(x)}{|x|^{N+2\tau}}\,dx
\le C_\tau R^{-\alpha-2\tau},\qquad
\int_{|x|>R}\frac{v(x)}{|x|^{N+2\tau}}\,dx
\le C_\tau R^{-\beta-2\tau}.
\end{equation}
In particular,
\[
\int_{\R^N}\frac{u(x)+v(x)}{1+|x|^{N+2\tau}}\,dx<\infty
\qquad(\tau>0).
\]
\end{proposition}
\begin{proof}
Fix $R>0$. For $x\in B_R$,
\[
\begin{aligned}
|y|<2R&\ \Longrightarrow\ x-y\in B_{3R},\\
|y|\ge2R&\ \Longrightarrow\ |x-y|\ge |y|-R\ge |y|/2.
\end{aligned}
\]
By \eqref{eq:kernel}, $\int_{B_{3R}}K_s(z)\,dz\le CR^{2s}$.
Together with $|B_R|=|B_1|R^N$, the preceding comparisons give
\[
\fint_{B_R}K_s(x-y)\,dx\le
\begin{cases}
CR^{2s-N},&|y|<2R,\\
C|y|^{2s-N},&|y|\ge2R.
\end{cases}
\]
Split $|y|\ge2R$ into the rings $2^jR\le|y|<2^{j+1}R$, $j\ge1$.
Tonelli's theorem and \eqref{eq:source-bounds} give
\[
\begin{aligned}
\fint_{B_R}u(x)\,dx
&=\int_{\R^N}G(y)\left(\fint_{B_R}K_s(x-y)\,dx\right)dy\\
&\le CR^{2s-N}\int_{B_{2R}}G\,dy
+C\sum_{j\ge1}(2^jR)^{2s-N}\int_{B_{2^{j+1}R}}G\,dy\\
&\le CR^{-\alpha}\left(1+\sum_{j\ge1}2^{-j\alpha}\right)
\le CR^{-\alpha},
\end{aligned}
\]
where $\alpha>0$. Replacing $u,G,\alpha$ by $v,F,\beta$ proves
\eqref{eq:mean-decay}.

Equations \eqref{eq:source-bounds} and \eqref{eq:mean-decay} with $R=4$
give the first bound in \eqref{eq:inputs}. For $x\in B_2$,
\[
\begin{aligned}
|x-y|\ge\ell&\ \Longrightarrow\ K_s(x-y)\le C\ell^{2s-N},\\
|y|\ge8&\ \Longrightarrow\ |x-y|\ge |y|/2.
\end{aligned}
\]
Splitting at $|y|=8$ and applying \eqref{eq:source-bounds} yields
\[
\begin{aligned}
&\int_{|x-y|\ge\ell}K_s(x-y)(F(y)+G(y))\,dy\\
&\quad\le C\ell^{2s-N}\int_{B_8}(F+G)\,dy
+C\sum_{j\ge3}2^{j(2s-N)}\int_{B_{2^{j+1}}}(F+G)\,dy\\
&\quad\le C\ell^{2s-N}
+C\sum_{j\ge3}\left(2^{-j\alpha}+2^{-j\beta}\right)
\le C\ell^{2s-N}.
\end{aligned}
\]
The last inequality uses $\alpha,\beta>0$ and
$\ell^{2s-N}\ge1$ for $0<\ell\le1$.

For $\tau>0$, decompose $|x|>R$ into $2^jR<|x|\le2^{j+1}R$, $j\ge0$.
By \eqref{eq:mean-decay},
\[
\begin{aligned}
\int_{|x|>R}\frac{u(x)}{|x|^{N+2\tau}}\,dx
&\le\sum_{j\ge0}(2^jR)^{-N-2\tau}
|B_{2^{j+1}R}|\fint_{B_{2^{j+1}R}}u\,dx\\
&\le CR^{-\alpha-2\tau}\sum_{j\ge0}2^{-j(\alpha+2\tau)}
\le C_\tau R^{-\alpha-2\tau}.
\end{aligned}
\]
The series converges since $\alpha+2\tau>0$. Replacing $u,\alpha$
by $v,\beta$ proves \eqref{eq:general-tail}. The last assertion
follows by taking $R=1$ in \eqref{eq:mean-decay} and
\eqref{eq:general-tail} and splitting the integral at $|x|=1$.
\end{proof}

\subsection{Regularity and the converse implication}

The next lemma justifies differentiation of the Riesz identities and
integration by parts near the origin.

\begin{lemma}\label{lem:regularity}
Under \eqref{eq:integral-parameters}, let $(u,v)$ be a positive
continuous solution of \eqref{eq:exact-potential}. Both components are smooth on
$\R^N\setminus\{0\}$. For each
fixed solution there is $C_{u,v}$ such that
\begin{equation}\label{eq:origin-gradient}
|x|\bigl(|\nabla u(x)|+|\nabla v(x)|\bigr)\le C_{u,v}
\qquad(0<|x|<1).
\end{equation}
For $x\ne0$,
\begin{equation}\label{eq:gradient}
\nabla u(x)=\lim_{\varepsilon\downarrow0}
\int_{|x-y|>\varepsilon}\nabla K_s(x-y)G(y)\,dy,
\end{equation}
and the analogous formula holds for $v$. The integral converges
absolutely when $2s>1$.
\end{lemma}
\begin{proof}
Put $A=\{1/4<|z|<4\}$ and $A'=\{1/2<|z|<2\}$.
For $r>0$, let $U_r(z)=u(rz)$, $V_r(z)=v(rz)$, and
$h_r(z)=r^{2s}G(rz)$, so $U_r=I_{2s}h_r$ on $\R^N$.
Continuity and positivity at zero give, for some $r_0\in(0,1)$,
\[
c\le U_r,V_r\le C\qquad(z\in B_8,\ 0<r<r_0).
\]
All annular constants may depend on $(u,v)$.
Choose $0\le\chi\le1$ smooth, equal to one near $\overline A$,
and supported in $\{1/8<|z|<8\}$. Since $a+2s>0$, for $z\in A$,
\[
\begin{aligned}
\|\chi h_r\|_\infty+\|I_{2s}(\chi h_r)\|_\infty
&\le Cr^{a+2s}\le C,\\
|\nabla^k I_{2s}((1-\chi)h_r)(z)|
&\le C_k I_{2s}h_r(z)=C_kU_r(z)\le C_k
\qquad(k\ge0).
\end{aligned}
\]
Indeed, for $z\in A$ and $y\in\operatorname{supp}(1-\chi)$,
$|z-y|\ge c_0>0$ and
\[
|\nabla^kK_s(z-y)|\le C_k|z-y|^{-k}K_s(z-y)
\le C_kK_s(z-y).
\]
Integrating this inequality against $h_r\ge0$ gives the second bound.
The same argument applies to $V_r$ with $b$ in place of $a$.
Fix an irrational $\theta\in(0,\min\{2s,1\})$. On nested annuli
$A'\Subset A_{j+1}\Subset A_j\Subset A$, the potential estimates
of Silvestre \cite[Propositions 2.8--2.9]{silvestre-regularity} give
\[
\begin{gathered}
U_r,V_r\in L^\infty(A)
\ \Longrightarrow\ U_r,V_r\in C^\theta(A_1),\\
U_r,V_r\in C^{j\theta}(A_j)
\ \Longrightarrow\ |z|^aV_r^p,\ |z|^bU_r^q\in C^{j\theta}(A_j)
\ \Longrightarrow\ U_r,V_r\in C^{(j+1)\theta}(A_{j+1}).
\end{gathered}
\]
At each step, use a cutoff supported in $A_j$ and equal to one
near $\overline{A_{j+1}}$, and repeat the same decomposition.
For $j\theta=k+\vartheta$, with $k$ an integer and $0<\vartheta<1$,
let $f$ be either localized right-hand side. For $|\nu|\le k$,
\[
\partial^\nu I_{2s}f=I_{2s}(\partial^\nu f),\qquad
\|\partial^\nu I_{2s}f\|_\infty\le C\|\partial^\nu f\|_\infty.
\]
The derivative identity holds first in distributions; the compact
support makes the potentials on the right continuous and bounded.
Apply the cited estimates to these derivatives. Since $\theta<2s$,
a slightly smaller input H\"older exponent avoids an integer endpoint
while still giving the required output exponent $(j+1)\theta$.
The middle implication uses $c\le U_r,V_r\le C$; the factors
$r^{a+2s},r^{b+2s}$ and the smooth outer potentials are uniformly
bounded. Choose $j\theta>1$ to obtain
\[
\|U_r\|_{C^1(A')}+\|V_r\|_{C^1(A')}\le C_{u,v}.
\]
For any fixed $r>0$, positivity and continuity give upper and lower
bounds on the compact annuli, now allowed to depend on $r$.
Iteration to arbitrary $j$ therefore gives smoothness on
$\R^N\setminus\{0\}$.
Now $\nabla U_r(z)=r\nabla u(rz)$ and its $v$ counterpart prove
\eqref{eq:origin-gradient} for small $|x|$; smoothness bounds the
remaining compact annulus.

Fix $x\ne0$ and $0<\rho<|x|/2$. Smoothness of $G$ near $x$ gives
\[
\begin{aligned}
\int_{B_\rho(x)}|\nabla K_s(x-y)|\,|G(y)-G(x)|\,dy
&\le C_{x,u,v}\int_0^\rho t^{2s-1}\,dt<\infty,\\
\int_{|x-y|\ge\rho}|\nabla K_s(x-y)|G(y)\,dy
&\le C\rho^{-1}u(x)<\infty.
\end{aligned}
\]
Oddness gives $\int_{\varepsilon<|z|<\rho}\nabla K_s(z)\,dz=0$.
Differentiation with this subtraction therefore yields
\[
\nabla u(x)=\int_{B_\rho(x)}\nabla K_s(x-y)(G(y)-G(x))\,dy
+\int_{|x-y|\ge\rho}\nabla K_s(x-y)G(y)\,dy,
\]
which is \eqref{eq:gradient}. The proof for $v$ is identical.
Without subtraction the local radial integral is
$\int_0^\rho t^{2s-2}\,dt$, finite only when $2s>1$.
\end{proof}

\begin{proposition}
\label{thm:representation}
Under \eqref{eq:parameters-main}, every positive classical entire
solution satisfies \eqref{eq:exact-potential} at every point of
$\R^N$, with finite continuous potentials. Conversely, every positive
continuous solution of this integral system belongs to the class in
Definition \ref{def:class}. It solves \eqref{eq:system} distributionally
on $\R^N$ and pointwise on $\R^N\setminus\{0\}$, where it is smooth.
\end{proposition}
\begin{proof}
Lemma \ref{lem:system-potentials} proves the assertion for system
solutions. It remains to prove the converse.

Proposition \ref{prop:inputs}, with $\tau=s$, gives
$u,v\in\mathcal L_{2s}$, and Lemma \ref{lem:regularity} gives
smoothness away from zero. For $\phi\in C_c^\infty(\R^N)$, the singular-integral formula gives
$|(-\Delta)^s\phi(x)|\le C_\phi(1+|x|)^{-N-2s}$.
For $|y|\ge2$, split its potential into $|x|<|y|/2$,
$|x-y|<|y|/2$, and the remaining region. For $|y|<2$, use the
local integrability of $K_s$ and the displayed decay. More explicitly,
put $f=(-\Delta)^s\phi$ and $R=|y|\ge2$. The three regions give
\[
\begin{aligned}
I_{2s}|f|(y)
&\le CR^{2s-N}\|f\|_{L^1}
 +C_\phi R^{-N-2s}\int_{B_{R/2}}K_s(z)\,dz\\
&\quad+CR^{2s-N}\int_{|x|\ge R/2}|f(x)|\,dx\\
&\le C_\phi(R^{2s-N}+R^{-N}+R^{-N})
 \le C_\phi R^{2s-N}.
\end{aligned}
\]
Thus $I_{2s}|(-\Delta)^s\phi|(y)\le C_\phi(1+|y|)^{2s-N}$ for all $y$.
Proposition \ref{prop:source} implies the absolute Fubini bound
\[
\begin{aligned}
\int(F+G)I_{2s}|(-\Delta)^s\phi|
&\le C_\phi\left[\int_{B_2}(F+G)
+\sum_{j\ge1}2^{j(2s-N)}\int_{B_{2^{j+1}}}(F+G)\right]\\
&\le C_\phi\left[1+\sum_{j\ge1}(2^{-j\alpha}+2^{-j\beta})\right]<\infty.
\end{aligned}
\]
Use this absolute integrability to interchange the integrals.
The Fourier normalization gives $I_{2s}((-\Delta)^s\phi)=\phi$, hence
\[
\begin{aligned}
\int_{\R^N}u(-\Delta)^s\phi\,dx
&=\int_{\R^N}G(y)
 \left[\int_{\R^N}K_s(x-y)(-\Delta)^s\phi(x)\,dx\right]dy\\
&=\int_{\R^N}G\phi\,dy.
\end{aligned}
\]
The same substitution with $v=I_{2s}F$ gives
$\int v(-\Delta)^s\phi=\int F\phi$.
Together with smoothness off zero and $u,v\in\mathcal L_{2s}$,
these identities give the pointwise equations there and all
requirements of Definition \ref{def:class}.
\end{proof}

\section{Subcritical nonexistence}\label{sec:nonexistence}

We prove the fractional H\'enon-Lane-Emden nonexistence statement
throughout the strict weighted subcritical region
\[
\frac{N+a}{p+1}+\frac{N+b}{q+1}>N-2s.
\]
Throughout this section, $(u,v)$ is a positive continuous solution
of \eqref{eq:exact-potential}, both potentials are finite, and
\eqref{eq:integral-parameters} holds. In particular, $N=1$ is
allowed when $0<s<1/2$, and neither $p$ nor $q$ is required
separately to exceed one.

The proof first bounds
$\int_{B_1}(|x|^bu^{q+1}+|x|^av^{p+1})$ independently of the
solution. We obtain this bound from a localized Pohozaev identity
and a comparison of complete component intervals of the sets
$\{F>t\}$ and $\{G>t'\}$. We then apply the bound to
\eqref{eq:dilation}; the strict subcritical inequality makes the
rescaled energy grow by a positive power of the radius.
Local boundedness supplies the required integrability of the
nonlinear energy on bounded sets.
Unless stated otherwise, $C$ depends on $N,s,a,b,p,q$ and the fixed
cutoffs, and is independent of the solution and of the parameters
$\kappa,\delta$ introduced below.

\subsection{The cutoff and the local energies}\label{sec:source-energies}

With $F,G$ as in Section \ref{sec:inputs}, define the nonnegative
measure
\[
d\pi(x,y)=F(x)G(y)K_s(x-y)\,dx\,dy.
\]
All double integrals against $d\pi$ are over
$\R^N\times\R^N$. Substituting $u=K_s*G$ and $v=K_s*F$ gives
\begin{equation}\label{eq:marginals}
\begin{aligned}
\int_{\R^N}F(x)G(y)K_s(x-y)\,dy&=F(x)u(x),\\
\int_{\R^N}F(x)G(y)K_s(x-y)\,dx&=G(y)v(y).
\end{aligned}
\end{equation}
Because $a,b>-2s>-N$ and $u,v$ are locally bounded,
$Fu=|x|^bu^{q+1}$ and $Gv=|x|^av^{p+1}$ are locally integrable.

Choose
\begin{equation}\label{eq:cutoff-powers-final}
m=(N+1)^2,\qquad
\rho(x)=\min\{1,(2-|x|)_+\},\qquad
\eta(x)=\bigl([1-(|x|-1)_+^3]_+\bigr)^m,
\end{equation}
and set $X(x)=\eta(x)x$ and $w(x)=-x\cdot\nabla\eta(x)$.
For $1<r=|x|<2$,
\[
1-(r-1)^3=(2-r)\bigl(1+(r-1)+(r-1)^2\bigr),\qquad
w(x)=3mr(r-1)^2[1-(r-1)^3]^{m-1}.
\]
Thus $\eta\in C_c^2(\R^N)$, $\eta=1$ on $B_1$, $w\ge0$,
and $w=0$ outside $1<|x|<2$. Differentiating the displayed
formulas gives
\begin{equation}\label{eq:cutoff}
\begin{gathered}
\rho^m\le\eta\le C\rho^m,\qquad
|\nabla\eta|+w\le C\rho^{m-1},\qquad \rho w\le C\eta,\\
|\nabla^2\eta|+|\nabla w|+|\nabla^2X|\le C\rho^{m-2}.
\end{gathered}
\end{equation}
Here the constants depend only on $N$. In particular, $w$ and
its first derivatives vanish at $|x|=1$.

Define the finite integrals
\[
\begin{gathered}
E_F=\int\eta Fu,\qquad E_G=\int\eta Gv,\qquad E=E_F+E_G,\\
W_F=\int wFu,\qquad W_G=\int wGv.
\end{gathered}
\]
For $|e|=1$ and $x\ne y$, put
\[
D_e(x,y)=\frac{((x-y)\cdot e)^2}{|x-y|^2},\qquad
\omega_x=\frac{x}{|x|},\qquad D(x,y)=D_{\omega_x}(x,y)\quad(x\ne0).
\]
Set $D_e=0$ on the diagonal and $D=0$ also at $x=0$.
Since $0\le D_e,D\le1$, \eqref{eq:marginals} gives
\[
T=\iint w(x)D(x,y)\,d\pi,\qquad 0\le T\le W_F.
\]

\subsection{The localized Pohozaev estimate}\label{sec:virial}

For $x\ne y$, set
\[
\Phi(x,y)=\frac{(X(x)-X(y))\cdot(x-y)}{|x-y|^2},
\]
and set $\Phi=0$ on the diagonal. The cancellation
$X(x)-X(y)$ is essential when $2s\le1$: it replaces
$|\nabla K_s(x-y)|$ by a constant times $K_s(x-y)$.

\begin{lemma}\label{lem:virial}
Under \eqref{eq:integral-parameters}, the cutoff and integrals
defined above satisfy
\begin{equation}\label{eq:virial}
\frac{N+b}{q+1}E_F+\frac{N+a}{p+1}E_G
-\frac{W_F}{q+1}-\frac{W_G}{p+1}
=(N-2s)\iint\Phi\,d\pi.
\end{equation}
The integral on the right is absolutely convergent.
\end{lemma}
\begin{proof}
Since $X=0$ on $\R^N\setminus B_2$ and $X$ is Lipschitz,
\[
|\Phi(x,y)|\le\operatorname{Lip}(X)
\bigl(\mathbf1_{B_2}(x)+\mathbf1_{B_2}(y)\bigr).
\]
Integrating this inequality with \eqref{eq:marginals} gives
\[
\iint|\Phi|\,d\pi
\le\operatorname{Lip}(X)\int_{B_2}(Fu+Gv)<\infty.
\]

To differentiate the potentials, use the same cutoff in both
equations:
\[
J_\varepsilon^G(x)
=\int_{|x-y|>\varepsilon}\nabla K_s(x-y)G(y)\,dy,\qquad
J_\varepsilon^F(x)
=\int_{|x-y|>\varepsilon}\nabla K_s(x-y)F(y)\,dy.
\]
The identity
$\nabla K_s(z)=-(N-2s)zK_s(z)/|z|^2$ implies
\[
|J_\varepsilon^G(x)|\le(N-2s)\varepsilon^{-1}u(x),
\qquad
|J_\varepsilon^F(x)|\le(N-2s)\varepsilon^{-1}v(x).
\]
These bounds and the support of $X$ justify Fubini in
\[
\begin{aligned}
&-\int F(x)X(x)\cdot J_\varepsilon^G(x)\,dx
-\int G(y)X(y)\cdot J_\varepsilon^F(y)\,dy\\
&\quad=-\iint_{|x-y|>\varepsilon}
F(x)G(y)\bigl(X(x)-X(y)\bigr)\cdot\nabla K_s(x-y)\,dx\,dy\\
&\quad=(N-2s)\iint_{|x-y|>\varepsilon}\Phi\,d\pi.
\end{aligned}
\]
The second line uses
$\nabla K_s(y-x)=-\nabla K_s(x-y)$.
The bound on $\iint|\Phi|\,d\pi$ permits
$\varepsilon\downarrow0$ on the right.

We next obtain an integrable bound on the left, uniform in
$\varepsilon$. Let $r=|x|<1/2$.
Continuity and positivity on a fixed ball give
$c_{u,v}\le u,v\le C_{u,v}$ there. For $|x-y|<r/2$,
Lemma \ref{lem:regularity} and $r/2<|y|<3r/2$ imply
\[
\begin{aligned}
|\nabla G(y)|
&\le |a|\,|y|^{a-1}v(y)^p
+p|y|^av(y)^{p-1}|\nabla v(y)|
\le C_{u,v}r^{a-1},\\
|G(y)-G(x)|&\le C_{u,v}r^{a-1}|x-y|.
\end{aligned}
\]
The positive lower bound for $v$ also covers $0<p<1$.
Oddness of $\nabla K_s$ removes the constant $G(x)$ on
$\{\varepsilon<|x-y|<r/2\}$. On the complementary region
$|x-y|\ge r/2$, one has
$r|\nabla K_s(x-y)|\le CK_s(x-y)$. Consequently,
\[
\begin{aligned}
r|J_\varepsilon^G(x)|
&\le C_{u,v}r^a\int_0^{r/2}t^{2s-1}\,dt
+C\int_{\R^N}K_s(x-y)G(y)\,dy\\
&\le C_{u,v}r^{a+2s}+Cu(x)\le C_{u,v}.
\end{aligned}
\]
If $\varepsilon\ge r/2$, only the second term is needed.
The last inequality uses $a+2s>0$ and boundedness of $u$ near
zero. The identical computation with $(F,v,b,q)$ gives
$r|J_\varepsilon^F(x)|\le C_{u,v}$.
On each compact set away from zero, subtracting $G(x)$ and
$F(x)$ on a fixed small ball gives uniform bounds in the same way.
Since $|X(x)|\le|x|\mathbf1_{B_2}(x)$, we obtain
\[
|FX\cdot J_\varepsilon^G|\le C_{u,v}F\mathbf1_{B_2},
\qquad
|GX\cdot J_\varepsilon^F|\le C_{u,v}G\mathbf1_{B_2}.
\]
The right sides are integrable because $a,b>-N$.
Using \eqref{eq:gradient} and dominated convergence therefore yields
\[
-\int FX\cdot\nabla u-\int GX\cdot\nabla v
=(N-2s)\iint\Phi\,d\pi.
\]

It remains to evaluate these two integrals. On $|x|>\delta$,
\[
F\nabla u=\frac{|x|^b}{q+1}\nabla(u^{q+1}),\qquad
\operatorname{div}(|x|^bX)
=|x|^b\bigl((N+b)\eta-w\bigr).
\]
For $0<\delta<1$, the inner unit normal to
$\R^N\setminus\overline B_\delta$ is $-x/\delta$ and
$X=x$ on $\partial B_\delta$. Integration by parts gives
\[
\begin{aligned}
-\int_{|x|>\delta}FX\cdot\nabla u
={}&\frac1{q+1}\int_{|x|>\delta}
|x|^bu^{q+1}\bigl((N+b)\eta-w\bigr)\\
&+\frac{\delta^{b+1}}{q+1}
\int_{\partial B_\delta}u^{q+1}\,dS.
\end{aligned}
\]
Since $u$ is bounded near zero and $N+b>0$,
\[
0\le\delta^{b+1}\int_{\partial B_\delta}u^{q+1}\,dS
\le C_{u,v}\delta^{N+b}\longrightarrow0.
\]
For $(G,v,a,p)$ the boundary term is bounded by
$C_{u,v}\delta^{N+a}\to0$. Substituting both limits in the preceding
paired identity proves \eqref{eq:virial}. Constants depending on
the solution have been used only to justify the limiting operations.
\end{proof}

\subsubsection{A common cutoff estimate}\label{sec:localization}

The following estimate bounds the Taylor error in $\Phi$ and the
difference between a cutoff evaluated at $x$ and at $y$.

\begin{lemma}\label{lem:localization}
For every $0<\delta<1/4$,
\begin{equation}\label{eq:localization}
\iint|\Phi-\eta(x)+w(x)D(x,y)|\,d\pi
\le C\delta E+C\delta^{2s-N-1}.
\end{equation}
If $h\in C_c^1(\R^N)$ vanishes outside $B_2$ and, for some $H\ge0$,
\begin{equation}\label{eq:cutoff-comparison-hypothesis}
|h|\le H\rho^{m-N},\qquad
|\nabla h|\le H\rho^{m-N-1},
\end{equation}
then
\begin{equation}\label{eq:cutoff-comparison}
\iint|h(x)-h(y)|\,d\pi
\le CH\delta E+CH\delta^{2s-N}.
\end{equation}
The constant $C$ is independent of $h,H,\delta$ and $(u,v)$.
\end{lemma}
\begin{proof}
For $z\in B_2$, set $\ell(z)=\delta\rho(z)^{N+1}$.
Let $A_\delta$ consist of pairs for which
$|x-y|\le\ell(z)$ at at least one endpoint
$z\in\{x,y\}\cap B_2$.
Since $\rho$ is $1$-Lipschitz and
$\ell(z)\le\delta\rho(z)$, every $\xi\in[x,y]$ satisfies
\[
(1-\delta)\rho(z)\le\rho(\xi)\le(1+\delta)\rho(z).
\]
In particular, this segment lies in $B_2$.
For $x\ne0$,
\[
\nabla X(x)=\eta(x)\operatorname{Id}
-w(x)\omega_x\otimes\omega_x,
\]
with the second term zero on $B_1$. Taylor's formula and
\eqref{eq:cutoff} therefore give, on $A_\delta$,
\[
\begin{aligned}
|\Phi-\eta(x)+w(x)D(x,y)|
&\le C|x-y|\rho(z)^{m-2}
\le C\delta\rho(z)^{m+N-1}
\le C\delta\rho(z)^m,\\
|h(x)-h(y)|
&\le CH|x-y|\rho(z)^{m-N-1}
\le CH\delta\rho(z)^m.
\end{aligned}
\]
The first line uses $N\ge1$ and $0<\rho\le1$.
Substitute
$\rho(z)^m\le\eta(z)\le\eta(x)+\eta(y)$ and then
\eqref{eq:marginals}:
\[
\begin{aligned}
\iint_{A_\delta}|\Phi-\eta(x)+w(x)D|\,d\pi&\le C\delta E,\\
\iint_{A_\delta}|h(x)-h(y)|\,d\pi&\le CH\delta E.
\end{aligned}
\]

On $A_\delta^c$, each endpoint $z\in B_2$ satisfies
$|x-y|>\ell(z)$. From $|X(z)|\le2\eta(z)$ and
\eqref{eq:cutoff}, one obtains
\[
\begin{aligned}
|\Phi-\eta(x)+w(x)D|
&\le C\sum_{z\in\{x,y\}\cap B_2}
\left(\rho(z)^{m-1}+\frac{\rho(z)^m}{|x-y|}\right),\\
|h(x)-h(y)|
&\le H\sum_{z\in\{x,y\}\cap B_2}\rho(z)^{m-N}.
\end{aligned}
\]
Apply \eqref{eq:inputs} at this endpoint. Since
$0<\ell(z)<1$,
\[
\begin{aligned}
\int_{|z-y|>\ell(z)}K_s(z-y)(F(y)+G(y))\,dy
&\le C\ell(z)^{2s-N},\\
\int_{|z-y|>\ell(z)}
\frac{K_s(z-y)}{|z-y|}(F(y)+G(y))\,dy
&\le C\ell(z)^{2s-N-1}.
\end{aligned}
\]
The second estimate follows by inserting
$|z-y|^{-1}\le\ell(z)^{-1}$ in the first.
The powers at the remaining endpoint are exactly
\[
\begin{aligned}
\rho^{m-1}\ell^{2s-N}
&=\delta^{2s-N}\rho^{N+2s(N+1)},\\
\rho^m\ell^{2s-N-1}
&=\delta^{2s-N-1}\rho^{2s(N+1)},\\
\rho^{m-N}\ell^{2s-N}
&=\delta^{2s-N}\rho^{1+2s(N+1)}.
\end{aligned}
\]
Each exponent of $\rho$ is positive. Thus $\rho^\theta\le1$
in all three lines, and the remaining endpoint integral is
bounded by $\int_{B_2}(F+G)\le C$, again by \eqref{eq:inputs}.
It follows that
\[
\begin{aligned}
\iint_{A_\delta^c}|\Phi-\eta(x)+w(x)D|\,d\pi
&\le C\delta^{2s-N}+C\delta^{2s-N-1},\\
\iint_{A_\delta^c}|h(x)-h(y)|\,d\pi
&\le CH\delta^{2s-N}.
\end{aligned}
\]
Adding the two regions and using $\delta<1$ proves both claims.
\end{proof}

The function $h=\eta$ satisfies
\eqref{eq:cutoff-comparison-hypothesis} with $H=C$:
indeed, $\rho^m\le\rho^{m-N}$ and
$\rho^{m-1}\le\rho^{m-N-1}$ for $0<\rho\le1$.
Equations \eqref{eq:marginals} and \eqref{eq:cutoff-comparison}
therefore imply
\begin{equation}\label{eq:energy-symmetry}
\begin{aligned}
|E_F-E_G|
&=\left|\iint(\eta(x)-\eta(y))\,d\pi\right|\\
&\le\iint|\eta(x)-\eta(y)|\,d\pi
\le C\delta E+C\delta^{2s-N}.
\end{aligned}
\end{equation}

\begin{proposition}\label{prop:local-balance}
For $0<\delta<1/4$,
\begin{equation}\label{eq:local-balance}
\frac\gamma2E\le
\frac{W_F}{q+1}+\frac{W_G}{p+1}-(N-2s)T
+C\delta E+C\delta^{2s-N-1},
\end{equation}
where $\gamma$ is defined in \eqref{eq:gaps-main}.
\end{proposition}
\begin{proof}
Define
$Q=\iint(\Phi-\eta(x)+w(x)D)\,d\pi$.
By \eqref{eq:marginals},
$\iint\Phi\,d\pi=E_F-T+Q$.
Substituting this identity in \eqref{eq:virial} gives
\[
\begin{aligned}
&\left(\frac{N+b}{q+1}-(N-2s)\right)E_F
+\frac{N+a}{p+1}E_G\\
&\qquad=\frac{W_F}{q+1}+\frac{W_G}{p+1}
-(N-2s)T+(N-2s)Q.
\end{aligned}
\]
Use $E_F=(E+(E_F-E_G))/2$ and
$E_G=(E-(E_F-E_G))/2$ on the left. The coefficient of $E$
is $\gamma/2$, so
\[
\begin{aligned}
\frac\gamma2E
={}&\frac{W_F}{q+1}+\frac{W_G}{p+1}
-(N-2s)T+(N-2s)Q\\
&-\frac12\left(\frac{N+b}{q+1}-\frac{N+a}{p+1}
-(N-2s)\right)(E_F-E_G).
\end{aligned}
\]
Bound $|Q|$ by \eqref{eq:localization} and
$|E_F-E_G|$ by \eqref{eq:energy-symmetry}.
Since $\delta^{2s-N}\le\delta^{2s-N-1}$, this is
\eqref{eq:local-balance}.
\end{proof}

\subsection{Estimate for the cutoff terms}\label{sec:pressure}

We now estimate the three cutoff terms on the right of
\eqref{eq:local-balance}. Integration in the source level gives
\begin{equation}\label{eq:ideal-layer}
\begin{aligned}
\int_0^{F(x)}
\left[u(x)-(t/|x|^b)^{1/q}\right]dt
&=F(x)u(x)-\frac{q}{q+1}|x|^{-b/q}F(x)^{1+1/q}\\
&=\frac{F(x)u(x)}{q+1}\qquad(x\ne0).
\end{aligned}
\end{equation}
Here $F=|x|^bu^q$, and the calculation requires only $q>0$.
The counterpart for $G,v,a,p$ requires only $p>0$.
At a finite endpoint of a component interval of $\{F>t\}$
away from zero, $F=t$, so the corresponding endpoint value of
$u$ is $(t/|x|^b)^{1/q}$. We compare these values with the
interior value through the Riesz kernel.

\subsubsection{The comparison on two parallel intervals}\label{sec:interval}

Fix $e\in\Sph^{N-1}$ and write
\[
x=z+\xi e,\qquad y=z'+\xi'e,\qquad
z,z'\in e^\perp,\qquad \lambda=|z-z'|,\qquad h=\xi-\xi'.
\]
For $N\ge2$ we first take $z\ne z'$, hence $\lambda>0$.
For $N=1$, $e^\perp=\{0\}$ has zero-dimensional measure one,
$\lambda=0$, and $0<N-2s=1-2s<1$.
The restricted kernel and its directional factor are
\[
\mathcal K(\xi,\xi')
=c_{N,s}(\lambda^2+h^2)^{-(N-2s)/2},\qquad
\mathcal D(\xi,\xi')=\frac{h^2}{\lambda^2+h^2}.
\]
When $\lambda=0$, set $\mathcal D=1$ off the diagonal; values
on the diagonal do not affect the integrals.
For a bounded interval $I=(\xi_-,\xi_+)$, put
\[
\mathcal K_I(\xi')
=\frac{\mathcal K(\xi_-,\xi')+\mathcal K(\xi_+,\xi')}{2},
\]
and define $\mathcal K_J(\xi)$ by averaging at the endpoints of
a bounded interval $J$ in the second variable.

\begin{lemma}\label{lem:interval}
Let $I,J\subset\R$ be nonempty open intervals and $c_I,c_J\ge0$.
Assume that $I$ is bounded when $c_I>0$, and $J$ is bounded
when $c_J>0$. If $\lambda>0$, then
\begin{equation}\label{eq:interval}
\begin{aligned}
\iint_{I\times J}(c_I+c_J)\mathcal K
&\le\iint_{I\times J}(c_I\mathcal K_I+c_J\mathcal K_J)\\
&\quad+\iint_{I\times J}
\bigl((N-2s)\min(c_I,c_J)\mathcal D+|c_I-c_J|\bigr)\mathcal K.
\end{aligned}
\end{equation}
The same inequality holds for $\lambda=0$ if $0<N-2s<1$.
All integrals use $d\xi\,d\xi'$, and each nonnegative term may
equal $+\infty$. An endpoint term with zero coefficient is
defined to be zero without evaluating the endpoint.
\end{lemma}
\begin{proof}
If $\min(c_I,c_J)=0$, then
$|c_I-c_J|\mathcal K=(c_I+c_J)\mathcal K$.
The remaining terms on the right are nonnegative, proving the
inequality, including for unbounded intervals.

Suppose $c_I,c_J>0$ and $\lambda>0$.
Both intervals are bounded and $\mathcal K$ is smooth on their
closures. Write $m_I,m_J$ for their midpoints and
$d_0=m_I-m_J$. The endpoint identity
\[
[(\xi-m_I)\mathcal K]_{\xi_-}^{\xi_+}
=\frac{|I|}{2}\bigl(\mathcal K(\xi_-,\xi')+
\mathcal K(\xi_+,\xi')\bigr)=|I|\mathcal K_I(\xi')
\]
gives
\[
\int_I(\mathcal K-\mathcal K_I)\,d\xi
=-\int_I(\xi-m_I)\partial_\xi\mathcal K\,d\xi.
\]
The analogous formula for $J$ is
\[
\int_J(\mathcal K-\mathcal K_J)\,d\xi'
=-\int_J(\xi'-m_J)\partial_{\xi'}\mathcal K\,d\xi'.
\]
Insert
\[
\partial_\xi\mathcal K=-\partial_{\xi'}\mathcal K
=-(N-2s)\frac{h}{\lambda^2+h^2}\mathcal K
\]
in these two identities, integrate in the other variable, and add:
\begin{equation}\label{eq:midpoint}
\iint_{I\times J}(2\mathcal K-\mathcal K_I-\mathcal K_J)
=(N-2s)\iint_{I\times J}
\frac{h^2-d_0h}{\lambda^2+h^2}\mathcal K.
\end{equation}

To determine the sign of the $d_0h$ term, use the overlap length
\[
\rho_0(h)=|I\cap(J+h)|
=\max\left\{0,\min\left\{|I|,|J|,
\frac{|I|+|J|}{2}-|h-d_0|\right\}\right\}.
\]
For every nonnegative measurable $f$,
$\iint_{I\times J}f(\xi-\xi')\,d\xi\,d\xi'
=\int_\R f(h)\rho_0(h)\,dh$.
The formula for $\rho_0$ shows it is nonincreasing as
$|h-d_0|$ increases. If $h>0$ and $d_0\ge0$, then
$|h-d_0|\le|-h-d_0|$, so $\rho_0(h)\ge\rho_0(-h)$.
If $d_0\le0$, both inequalities reverse. In either case,
\[
d_0\bigl(\rho_0(h)-\rho_0(-h)\bigr)\ge0\qquad(h>0).
\]
Pairing $h$ and $-h$ therefore gives
\[
\begin{aligned}
d_0\iint_{I\times J}\frac{h}{\lambda^2+h^2}\mathcal K
&=c_{N,s}\int_0^\infty
\frac{d_0h\bigl(\rho_0(h)-\rho_0(-h)\bigr)}
{(\lambda^2+h^2)^{(N-2s)/2+1}}\,dh
\ge0.
\end{aligned}
\]
Deleting this nonnegative quantity from the right of
\eqref{eq:midpoint} yields
\[
2\iint_{I\times J}\mathcal K
\le\iint_{I\times J}
\bigl(\mathcal K_I+\mathcal K_J+(N-2s)\mathcal D\mathcal K\bigr).
\]
Multiply by $\min(c_I,c_J)$ and insert
\[
c_I+c_J=2\min(c_I,c_J)+|c_I-c_J|,
\quad
\min(c_I,c_J)\mathcal K_I\le c_I\mathcal K_I,
\quad
\min(c_I,c_J)\mathcal K_J\le c_J\mathcal K_J.
\]
This proves \eqref{eq:interval} for $\lambda>0$.

For $\lambda=0$, apply the proved inequality with
$\lambda=\varepsilon>0$. As $\varepsilon\downarrow0$,
the kernels
\[
c_{N,s}(\varepsilon^2+h^2)^{-(N-2s)/2},\qquad
c_{N,s}h^2(\varepsilon^2+h^2)^{-(N-2s)/2-1}
\]
and their endpoint values increase to the required kernels
almost everywhere. Monotone convergence in the nonnegative
inequality proves the assertion. When both coefficients are
positive, the intervals are bounded; the condition $N-2s<1$
also gives
$\int_{-r}^r|h|^{-(N-2s)}\,dh<\infty$ at every diagonal or
endpoint singularity.
\end{proof}

\subsubsection{Averaging near the radial direction}

For a function of $e\in\Sph^{N-1}$, write
\[
\avdir{H(e)}=\frac1{|\Sph^{N-1}|}\int_{\Sph^{N-1}}H(e)\,de.
\]
For $N=1$, this means $(H(-1)+H(1))/2$.
Fix a smooth function $0\le\psi\le1$, equal to one near zero
and zero on $[1,\infty)$. For $0<\sigma\le1$ and $|\omega|=1$,
set
\[
Z(\sigma)=\avdir{\psi(|e-\omega|/\sigma)}.
\]
Rotation invariance removes the dependence on $\omega$.
For $N\ge2$, spherical coordinates with
$|e-\omega|=2\sin(\theta/2)=\sigma t$ give
\[
Z(\sigma)=\frac{|\Sph^{N-2}|}{|\Sph^{N-1}|}\sigma^{N-1}
\int_0^1\psi(t)t^{N-2}(1-\sigma^2t^2/4)^{(N-3)/2}\,dt.
\]
Since $3/4\le1-\sigma^2t^2/4\le1$ and $\psi=1$ near zero,
this formula and its derivative imply
\[
c\sigma^{N-1}\le Z(\sigma)\le C\sigma^{N-1},
\qquad |Z'(\sigma)|\le C\sigma^{N-2}.
\]
For $N=1$, exactly one of $e=-1,1$ equals $\omega$, so
$Z(\sigma)=1/2$ and the same bounds hold.

Fix $0<\kappa<1$ and define
\[
h_e(x)=
\begin{cases}
\displaystyle\frac{w(x)}{Z(\kappa\rho(x))}
\psi\!\left(\frac{|e-\omega_x|}{\kappa\rho(x)}\right),
&1<|x|<2,\\
0,&\text{otherwise}.
\end{cases}
\]
For $N=1$, this is
$h_e(x)=2w(x)\mathbf1_{\{e=\operatorname{sign}x\}}$.
The definition of $Z$ gives $\avdir{h_e(x)}=w(x)$.
On the annulus, $|\nabla\omega_x|\le1$; where a derivative of
$\psi$ is nonzero, $|e-\omega_x|\le\kappa\rho(x)$. Hence
\[
\left|\nabla\frac{|e-\omega_x|}{\kappa\rho(x)}\right|
\le\frac{C}{\kappa\rho(x)},\qquad
\left|\nabla\frac1{Z(\kappa\rho(x))}\right|
\le C\kappa^{1-N}\rho(x)^{-N}.
\]
Together with \eqref{eq:cutoff}, these give
\begin{equation}\label{eq:cap}
\avdir{h_e}=w,\qquad
0\le h_e\le C\kappa^{1-N}\rho^{m-N},\qquad
|\nabla h_e|\le C\kappa^{-N}\rho^{m-N-1}.
\end{equation}
The constant value of $\psi$ near zero removes the apparent
nondifferentiability at $e=\omega_x$. At the two annular
boundaries, $w=O((|x|-1)^2)$ and $m-N>1$ show that $h_e$
and its first derivatives match zero. Thus $h_e\in C_c^1(\R^N)$.

Use \eqref{eq:cap} in Lemma \ref{lem:localization}, with
$H=C\kappa^{-N}$:
\begin{equation}\label{eq:cap-mismatch}
\iint|h_e(x)-h_e(y)|\,d\pi
\le C\kappa^{-N}\delta E+C\kappa^{-N}\delta^{2s-N}.
\end{equation}
For a unit vector $\nu=(x-y)/|x-y|$,
\[
|D_e-D|=|(\nu\cdot e)^2-(\nu\cdot\omega_x)^2|
\le2|e-\omega_x|.
\]
On the support of $h_e(x)$ this is at most
$2\kappa\rho(x)$. Multiplication by $h_e(x)$, averaging in
$e$, and \eqref{eq:marginals} now give
\begin{equation}\label{eq:cap-angle}
\begin{aligned}
\avdir{\iint h_e(x)D_e(x,y)\,d\pi}
&\le T+2\kappa\iint\rho(x)w(x)\,d\pi\\
&=T+2\kappa\int\rho wFu
\le T+C\kappa E.
\end{aligned}
\end{equation}
The last inequality is $\rho w\le C\eta$ from \eqref{eq:cutoff}.

\subsubsection{Restriction to short intervals}

For $0<t<F(x)$, let $I=I(e,x,t)$ be the component interval
containing $x$ of $\{F>t\}\cap(x+\R e)$.
To keep the coefficient constant on the whole interval, define
\[
c_I=
\begin{cases}
\inf_Ih_e,&I\text{ bounded and }|I|\le2\delta\inf_I\rho,\\
0,&\text{otherwise},
\end{cases}
\qquad c_F(e,x,t)=c_I.
\]
For a component $J$ of $\{G>t'\}$, define $c_G(e,y,t')=c_J$
by the same rule. In particular,
\[
0\le c_F(e,x,t)\le h_e(x),\qquad
0\le c_G(e,y,t')\le h_e(y).
\]

\begin{remark}\label{rem:layer-measurability}
These coefficients and all finite endpoints are jointly
measurable in direction, base point, and level.
Indeed, $\{(x,t):x\ne0,\ 0<t<F(x)\}$ is open.
The forward exit distance exceeds $r\ge0$ exactly when
$F(x+\xi e)>t$ for every $0\le\xi\le r$.
Compactness of this segment and openness of the layer make this
an open condition in $(e,x,t)$. The backward exit satisfies the
same property. Both exit distances, and hence the interval
length and its finite endpoints, are measurable.
The infima of the continuous functions $\rho,h_e$ over the
interval equal their infima over rational offsets strictly
between the two exits. These are infima of countably many
measurable functions. This proves measurability of $c_F$;
the proof for $c_G$ is identical.
\end{remark}

For fixed $e$, define
\[
\begin{gathered}
H_F(e)=\int h_eFu,\qquad H_G(e)=\int h_eGv,\\
L_F(e)=\int u(x)\int_0^{F(x)}
\bigl(h_e(x)-c_F(e,x,t)\bigr)\,dt\,dx,\\
L_G(e)=\int v(y)\int_0^{G(y)}
\bigl(h_e(y)-c_G(e,y,t')\bigr)\,dt'\,dy.
\end{gathered}
\]
The compact support and boundedness of $h_e$ give
\[
0\le L_F(e)\le H_F(e)<\infty,\qquad
0\le L_G(e)\le H_G(e)<\infty.
\]

\begin{lemma}\label{lem:layer-loss}
For $0<\kappa<1$ and $0<\delta<1/4$,
\begin{equation}\label{eq:layer-loss}
\avdir{\bigl(L_F(e)+L_G(e)\bigr)}
\le C\kappa^{-N}\delta E
+C\kappa^{(1-N)/2}\delta^{-N/2}\sqrt E.
\end{equation}
The constant is independent of $\kappa,\delta$ and the solution.
\end{lemma}
\begin{proof}
We estimate $L_F$.
Its integrand vanishes unless $1<|x|<2$, because
$0\le c_F\le h_e$. Fix such an $x$ and put
$\ell=\delta\rho(x)^{N+1}$.
If $|I(e,x,t)|\le\ell$, then
\[
\inf_I\rho\ge\rho(x)-\ell\ge(1-\delta)\rho(x),\qquad
|I|\le\delta\rho(x)\le2\delta\inf_I\rho.
\]
Thus this interval satisfies the retention rule and
$c_F=\inf_Ih_e$. Every point of $I$ has $\rho$ comparable to
$\rho(x)$, so the gradient bound in \eqref{eq:cap} gives
\begin{equation}\label{eq:short-layer-loss}
\begin{aligned}
0\le h_e(x)-c_F(e,x,t)
&\le C\kappa^{-N}|I|\rho(x)^{m-N-1}\\
&\le C\kappa^{-N}\delta\rho(x)^m.
\end{aligned}
\end{equation}
Integrating $u(x)$ times this bound over $0<t<F(x)$ yields
\[
\int u(x)\int_0^{F(x)}
(h_e-c_F)\mathbf1_{\{|I|\le\ell\}}\,dt\,dx
\le C\kappa^{-N}\delta\int\rho^mFu
\le C\kappa^{-N}\delta E_F.
\]

For the remaining intervals, $|I|>\ell$ implies that at least
one of the two segments of length $\ell/2$ from $x$ in the
directions $e,-e$ lies in $\{F>t\}$.
Let $\mathcal A(x,t)$ consist of directions whose forward
segment has this property. Polar coordinates centered at $x$
give
\[
\begin{aligned}
\int_{B_{\ell/2}(x)}F(y)\,dy
&=\int_{\Sph^{N-1}}\int_0^{\ell/2}
F(x+re)r^{N-1}\,dr\,de\\
&\ge\int_{\mathcal A(x,t)}\int_0^{\ell/2}tr^{N-1}\,dr\,de
=\frac{t\ell^N}{N2^N}|\mathcal A(x,t)|.
\end{aligned}
\]
Since
$\{e:|I(e,x,t)|>\ell\}\subset
\mathcal A(x,t)\cup(-\mathcal A(x,t))$,
and $B_{\ell/2}(x)\subset B_4$, \eqref{eq:inputs} gives
\begin{equation}\label{eq:long-layer-probability}
\frac{|\{e:|I(e,x,t)|>\ell\}|}{|\Sph^{N-1}|}
\le\frac{C}{t\ell^N}\int_{B_{\ell/2}(x)}F(y)\,dy
\le\frac{C}{t\ell^N}.
\end{equation}
For $N=1$, polar integration is the sum over the two rays,
so this calculation remains valid.

On these intervals $h_e-c_F\le h_e$.
Cauchy-Schwarz on the sphere, followed by
$\avdir{h_e^2}\le(\sup_eh_e)\avdir{h_e}=(\sup_eh_e)w$,
gives
\[
\begin{aligned}
\avdir{h_e(x)\mathbf1_{\{|I|>\ell\}}}
&\le\bigl[(\sup_eh_e(x))w(x)\bigr]^{1/2}
\left[\frac{|\{e:|I|>\ell\}|}{|\Sph^{N-1}|}\right]^{1/2}\\
&\le C\kappa^{(1-N)/2}
\rho(x)^{(2m-N-1)/2}\,
t^{-1/2}\delta^{-N/2}\rho(x)^{-N(N+1)/2}\\
&=C\kappa^{(1-N)/2}\delta^{-N/2}
\rho(x)^{m/2}t^{-1/2}.
\end{aligned}
\]
The second line uses \eqref{eq:cap}, \eqref{eq:cutoff},
and \eqref{eq:long-layer-probability}; the last line uses
$m=(N+1)^2$.
Integrate first in $t$, using
$\int_0^{F(x)}t^{-1/2}\,dt=2\sqrt{F(x)}$, and then in $x$:
\[
\begin{aligned}
&\int_{1<|x|<2}u(x)\int_0^{F(x)}
\avdir{h_e(x)\mathbf1_{\{|I|>\ell\}}}\,dt\,dx\\
&\quad\le C\kappa^{(1-N)/2}\delta^{-N/2}
\int_{1<|x|<2}\sqrt{\rho^mFu}\sqrt u\\
&\quad\le C\kappa^{(1-N)/2}\delta^{-N/2}
\left(\int\rho^mFu\right)^{1/2}
\left(\int_{B_2}u\right)^{1/2}\\
&\quad\le C\kappa^{(1-N)/2}\delta^{-N/2}\sqrt{E_F}.
\end{aligned}
\]
Here the last line uses $\rho^m\le\eta$ and
$\int_{B_2}u\le C$ from \eqref{eq:inputs}.
Combining the short and long interval estimates proves the
claimed bound for $\avdir{L_F}$.
Replace $(F,u)$ by $(G,v)$ throughout and add.
Finally,
$\sqrt{E_F}+\sqrt{E_G}\le\sqrt{2(E_F+E_G)}$
gives \eqref{eq:layer-loss}.
\end{proof}

\subsubsection{Endpoint values and integration over levels}

If $c_I>0$, then continuity of $h_e$ gives $h_e\ge c_I>0$
on $\overline I$. Thus both endpoints $x_-,x_+$ lie in
$1<|x|<2$. They are endpoints of a complete component of
$\{F>t\}$, and $F$ is continuous there, so
\[
F(x_-)=F(x_+)=t,\qquad
u(x_\pm)=(t/|x_\pm|^b)^{1/q}.
\]
Set $u_I=(u(x_-)+u(x_+))/2$.
For $x\in I$, we can have
\[
|x_\pm-x|\le |I|\le2\delta\inf_I\rho\le2\delta\rho(x).
\]
The map $(r,r')\mapsto(r/r')^{b/q}$ has bounded first derivatives
on $[1,2]^2$. Hence, with
$\zeta=(t/|x|^b)^{1/q}\le u(x)$,
\[
|u(x_\pm)-\zeta|
=\zeta\left|\left(\frac{|x|}{|x_\pm|}\right)^{b/q}-1\right|
\le C\zeta\,|x_\pm-x|.
\]
Averaging these two inequalities yields
\begin{equation}\label{eq:endpoint}
|u_I-(t/|x|^b)^{1/q}|
\le C\delta\rho(x)(t/|x|^b)^{1/q}
\le C\delta\rho(x)u(x).
\end{equation}
Define $v_J$ at retained $G$ intervals in the same manner;
the bound uses $(a,p)$ in place of $(b,q)$.
A product involving $u_I$ or $v_J$ with zero coefficient is
always defined as zero without evaluating the endpoints.

\begin{proposition}\label{prop:boundary-estimate}
For $0<\kappa<1$ and $0<\delta<1/4$,
\begin{equation}\label{eq:pressure-target}
\frac{W_F}{q+1}+\frac{W_G}{p+1}-(N-2s)T
\le C(\kappa+\kappa^{-N}\delta)E
+C_{\kappa,\delta}(\sqrt E+1).
\end{equation}
Here $C$ is independent of $\kappa,\delta$ and the solution,
and $C_{\kappa,\delta}$ is independent of the solution.
\end{proposition}
\begin{proof}
Fix $e$ and levels $t,t'>0$. Apply Lemma \ref{lem:interval}
to every pair of complete component intervals of
$\{F>t\}$ and $\{G>t'\}$ on parallel lines, with the
coefficients defined above. For $N\ge2$, integrate over
$z,z'\in e^\perp$ with $dx=dz\,d\xi$ and $dy=dz'\,d\xi'$;
the set $z=z'$ has product measure zero.
For $N=1$, use the $\lambda=0$ assertion of that lemma on
the single line. Each open subset of a line has countably
many component intervals. Remark \ref{rem:layer-measurability}
and Tonelli's theorem therefore permit summing over components
and integrating in $t,t'$ before any subtraction.

We record precisely how these integrations recover the
potentials. For a fixed $x$ and a fixed retained $F$ interval,
\[
\begin{aligned}
\int_0^\infty\int_{\{G>t'\}}K_s(x-y)\,dy\,dt'
&=\int K_s(x-y)G(y)\,dy=u(x),\\
\int_0^\infty\int_{\{G>t'\}}
\frac{K_s(x_--y)+K_s(x_+-y)}2\,dy\,dt'
&=\frac{u(x_-)+u(x_+)}2=u_I.
\end{aligned}
\]
Consequently, the $c_I\mathcal K$ and $c_I\mathcal K_I$
terms in \eqref{eq:interval} become respectively
\[
\int\int_0^{F(x)}c_Fu(x)\,dt\,dx,\qquad
\int\int_0^{F(x)}c_Fu_I\,dt\,dx.
\]
The $c_J$ terms give the corresponding expressions with
$G,v,v_J$.

All these integrated terms are finite. Indeed,
$c_F\le h_e$ and \eqref{eq:endpoint} give
\[
\begin{aligned}
\int\int_0^{F(x)}c_Fu(x)\,dt\,dx&\le H_F(e),\\
\int\int_0^{F(x)}c_Fu_I\,dt\,dx
&\le C\int\int_0^{F(x)}h_e(x)u(x)\,dt\,dx
=CH_F(e).
\end{aligned}
\]
The same bounds hold with $G$ in place of $F$.
The directional term is at most $(N-2s)H_F(e)$ because
$\min(c_F,c_G)\le h_e(x)$ and $D_e\le1$.
The mismatch term is at most $H_F(e)+H_G(e)$ because
$|c_F-c_G|\le h_e(x)+h_e(y)$.
We may therefore subtract the endpoint terms and define the
finite signed quantity
\[
J_e=\int\int_0^{F(x)}c_F(u-u_I)\,dt\,dx
+\int\int_0^{G(y)}c_G(v-v_J)\,dt'\,dy.
\]
For each pair of levels,
\[
|c_F-c_G|
\le |h_e(x)-h_e(y)|
+\bigl(h_e(x)-c_F\bigr)+\bigl(h_e(y)-c_G\bigr).
\]
The first term integrates to
$\iint|h_e(x)-h_e(y)|\,d\pi$.
For the second term, integration in $t',y$ gives $u(x)$,
so its integral is exactly $L_F(e)$; the third term similarly
gives $L_G(e)$. Thus the integrated form of
\eqref{eq:interval} is
\begin{equation}\label{eq:integrated-interval}
J_e\le (N-2s)\iint h_e(x)D_e(x,y)\,d\pi
+\iint|h_e(x)-h_e(y)|\,d\pi+L_F(e)+L_G(e).
\end{equation}

Next compare $J_e$ with the exact level formula
\eqref{eq:ideal-layer}. Since $0\le\zeta=(t/|x|^b)^{1/q}\le u(x)$,
\eqref{eq:endpoint} gives
\[
\begin{aligned}
h_e(x)(u-\zeta)-c_F(u-u_I)
&=(h_e(x)-c_F)(u-\zeta)+c_F(u_I-\zeta)\\
&\le (h_e(x)-c_F)u
+C\delta\rho(x)h_e(x)u.
\end{aligned}
\]
When $c_F=0$, the same inequality follows directly from
$0\le u-\zeta\le u$ and the zero-coefficient convention.
Integrate this inequality in $t,x$, apply \eqref{eq:ideal-layer},
and add its $(G,v)$ counterpart:
\[
\frac{H_F(e)}{q+1}+\frac{H_G(e)}{p+1}
\le J_e+L_F(e)+L_G(e)
+C\delta\int\rho h_e(Fu+Gv).
\]
Insert \eqref{eq:integrated-interval} in the right side,
then average in $e$. The three identities
\[
\begin{gathered}
\avdir{H_F}=W_F,\qquad \avdir{H_G}=W_G,\\
\avdir{\int\rho h_e(Fu+Gv)}
=\int\rho w(Fu+Gv)\le CE.
\end{gathered}
\]
use \eqref{eq:cap} and $\rho w\le C\eta$.
Applying \eqref{eq:cap-angle} to the directional term gives
\[
\begin{aligned}
\frac{W_F}{q+1}+\frac{W_G}{p+1}-(N-2s)T
&\le C(\kappa+\delta)E
+\avdir{\iint|h_e(x)-h_e(y)|\,d\pi}\\
&\quad+2\avdir{\bigl(L_F(e)+L_G(e)\bigr)}.
\end{aligned}
\]
Finally substitute \eqref{eq:cap-mismatch} and
\eqref{eq:layer-loss}:
\[
\begin{aligned}
\frac{W_F}{q+1}+\frac{W_G}{p+1}-(N-2s)T
\le{}&C(\kappa+\kappa^{-N}\delta)E\\
&+C\kappa^{(1-N)/2}\delta^{-N/2}\sqrt E
+C\kappa^{-N}\delta^{2s-N}.
\end{aligned}
\]
Here $\delta\le\kappa^{-N}\delta$ since $0<\kappa<1$.
For fixed $\kappa,\delta$, the last two coefficients depend
only on the equation parameters and these two choices.
This proves \eqref{eq:pressure-target}.
\end{proof}

\subsection{Absorption and subcritical scaling}\label{sec:finish}

\begin{theorem}\label{thm:riesz-liouville}
Let $N\ge1$, $0<s<\min\{1,N/2\}$, $a,b>-2s$,
$p,q>0$, and $pq>1$. If
\[
\frac{N+a}{p+1}+\frac{N+b}{q+1}>N-2s,
\]
then there is no positive continuous pair on $\R^N$ with finite
potentials satisfying
\[
\begin{cases}
u=I_{2s}(|x|^av^p),\\
v=I_{2s}(|x|^bu^q),
\end{cases}
\qquad x\in\R^N.
\]
\end{theorem}
\begin{proof}
The strict inequality is $\gamma>0$.
Insert \eqref{eq:pressure-target} into
\eqref{eq:local-balance} and use
$\delta\le\kappa^{-N}\delta$:
\[
\frac\gamma2E
\le C(\kappa+\kappa^{-N}\delta)E
+C_{\kappa,\delta}(\sqrt E+1).
\]
The term $C\delta^{2s-N-1}$ from
\eqref{eq:local-balance} is included in $C_{\kappa,\delta}$.
The coefficient $C$ of $E$ is independent of both choices.
Choose first $0<\kappa<1$ such that
$C\kappa\le\gamma/8$.
With this $\kappa$ fixed, choose $0<\delta<1/4$ such that
$C\kappa^{-N}\delta\le\gamma/8$.
Moving these two terms to the left gives
\begin{equation}\label{eq:absorb}
\frac\gamma4E\le C_{\kappa,\delta}(\sqrt E+1).
\end{equation}
If $E\ge1$, then $\sqrt E+1\le2\sqrt E$;
division by $\sqrt E$ gives
$\sqrt E\le8C_{\kappa,\delta}/\gamma$.
If $E<1$, no further estimate is needed.
Since $\eta=1$ on $B_1$,
\begin{equation}\label{eq:energy}
\int_{B_1}\bigl(|x|^bu^{q+1}+|x|^av^{p+1}\bigr)\,dx
\le E\le
\max\{1,(8C_{\kappa,\delta}/\gamma)^2\}=C.
\end{equation}
The choices of $\kappa,\delta$ depend only on
$N,s,a,b,p,q$, so this $C$ is uniform over all solutions.

For every $R>0$, \eqref{eq:dilation} gives another solution.
Apply \eqref{eq:energy} to it and change variables $y=Rx$:
\[
\begin{aligned}
\int_{B_1}|x|^bu_R(x)^{q+1}\,dx
&=R^{(q+1)\alpha-b-N}\int_{B_R}|y|^bu(y)^{q+1}\,dy,\\
\int_{B_1}|x|^av_R(x)^{p+1}\,dx
&=R^{(p+1)\beta-a-N}\int_{B_R}|y|^av(y)^{p+1}\,dy.
\end{aligned}
\]
The scale relations
$q\alpha-\beta=2s+b$ and
$p\beta-\alpha=2s+a$ show that both exponents equal
$\alpha+\beta+2s-N$.
By \eqref{eq:gap}, this common value is
$\gamma/(1-S)>0$. Thus
\[
\int_{B_R}\bigl(|x|^bu^{q+1}+|x|^av^{p+1}\bigr)\,dx
\le CR^{-\gamma/(1-S)}.
\]
For $R\ge1$, the positive components and positive weights on
$B_1\setminus\{0\}$ imply
\[
0<
\int_{B_1}\bigl(|x|^bu^{q+1}+|x|^av^{p+1}\bigr)\,dx
\le CR^{-\gamma/(1-S)}\longrightarrow0.
\]
This contradiction proves the theorem.
\end{proof}

\begin{proof}[Proof of Theorem \ref{thm:main}]
For $N\ge2$ and $0<s<1$, one has $2s<N$.
The assumptions of Theorem \ref{thm:main} therefore satisfy
all the hypotheses of Theorem \ref{thm:riesz-liouville}, which
excludes its positive continuous integral solutions.
If a positive classical entire solution existed,
Proposition \ref{thm:representation} would give such an integral
solution with finite potentials at every point.
The integral nonexistence statement excludes this possibility
as well.
\end{proof}

\section{Existence at and above one half}\label{sec:construction}

We construct a positive entire solution by letting the radius of a
Dirichlet ball tend to infinity. A positive linear term in the ball
equations will tend to zero, and a fixed central normalization will
prevent the limit from vanishing. The construction assumes \eqref{eq:supercritical} and
$1/(p+1)+1/(q+1)>1-2s$. The latter condition is automatic for $s\ge1/2$
and also gives the existence result used in Section
\ref{sec:critical-criterion}.

\subsection{Common notation and operators}\label{sec:common-notation}

Throughout Sections \ref{sec:construction} and
\ref{sec:critical-criterion}, assume \eqref{eq:parameters-main}.
Thus $N>2s$ and $a,b>-N$. The Riesz operator $I_{2s}$, its kernel
$K_s$, and its constant $c_{N,s}$ retain the normalization in
\eqref{eq:riesz-definition}. For real functions define
\begin{equation}\label{eq:common-energy-form}
\mathcal E_s(w,z)=\frac{C(N,s)}2
\iint_{\R^N\times\R^N}
\frac{(w(x)-w(y))(z(x)-z(y))}{|x-y|^{N+2s}}\,dx\,dy.
\end{equation}
The associated homogeneous space is
\begin{equation}\label{eq:common-homogeneous-space}
\dot H^s(\R^N)=
\overline{C_c^\infty(\R^N)}^{\|\cdot\|_{\dot H^s}},
\qquad \|w\|_{\dot H^s}^2=\mathcal E_s(w,w),
\qquad \dot H^s(\R^N)\hookrightarrow L^{2N/(N-2s)}(\R^N).
\end{equation}
The same form and homogeneous space will be used for the tests in
Section \ref{sec:critical-criterion}.

The fractional Sobolev space $H^s(\R^N)$ consists of functions $w$ such that
\[
\|w\|_{H^s(\R^N)}^2
=\int_{\R^N}|w|^2\,dx
+\iint_{\R^N\times\R^N}
\frac{|w(x)-w(y)|^2}{|x-y|^{N+2s}}\,dx\,dy<\infty.
\]
The zero exterior condition is imposed through the subspace
\[
H^s_0(B_R)=\{w\in H^s(\R^N):
w=0\text{ almost everywhere in }\R^N\setminus B_R\}.
\]
For $f\in C_c^\infty(B_R)$, the Green operator
$\mathcal G_{s,R}$ gives the unique weak solution $w\in H^s_0(B_R)$ of
\[
\begin{cases}
(-\Delta)^s w=f,&x\in B_R,\\
w=0,&x\in\R^N\setminus B_R.
\end{cases}
\]
The weak identity is
\[
\mathcal E_s(w,\phi)=\int_{B_R}f\phi\qquad(\phi\in H^s_0(B_R)).
\]
Its integral representation is
\[
(\mathcal G_{s,R}f)(x)
=\int_{B_R}G_{s,R}(x,y)f(y)\,dy,\qquad x\in B_R.
\]
For $R=1$, Dipierro-Grunau \cite[Theorem 1]{boggio} give the kernel
\begin{equation}\label{eq:boggio}
\begin{gathered}
G_s(x,y)=\kappa_{N,s}\int_0^{\omega(x)\omega(y)}
\frac{t^{s-1}}{(|x-y|^2+t)^{N/2}}\,dt,
\qquad x,y\in B_1,\quad x\ne y,\\
\omega(x)=1-|x|^2,\qquad
\kappa_{N,s}=\frac{\Gamma(N/2)}{4^s\pi^{N/2}\Gamma(s)^2}.
\end{gathered}
\end{equation}
For $x,y\in B_R$ with $x\ne y$, scaling gives
\[
G_{s,R}(x,y)=R^{2s-N}G_s(x/R,y/R),\qquad
0<G_{s,R}(x,y)\le K_s(x-y).
\]
Put
\[
\tau_R(x)=(R^2-|x|^2)_+,\qquad
\delta_R(x)=R-|x|\quad(x\in B_R),\qquad
\tau_R=(R+|x|)\delta_R.
\]
For fixed $x\ne y$, the upper limit in the scaled integral is
\[
\frac{\tau_R(x)\tau_R(y)}{R^2}
=R^2-|x|^2-|y|^2+\frac{|x|^2|y|^2}{R^2}\uparrow\infty
\quad\bigl(R>\max\{|x|,|y|\}\bigr).
\]
The  integral then gives
\begin{equation}\label{eq:green-exhaustion}
G_{s,R}(x,y)\uparrow K_s(x-y)\qquad(R\to\infty).
\end{equation}
For $\Omega=B_R(c)$, we write
$G_\Omega(x,y)=G_{s,R}(x-c,y-c)$ and
$\mathcal G_\Omega f=\int_\Omega G_\Omega(\cdot,y)f(y)\,dy$.
The same formulas define potentials for measurable sources whenever
the integrals are finite.

The boundary factor $\tau_R^s$
is built into the space
\[
X_R=\{w=\tau_R^s z:\ z\in C(\overline{B_R})\text{ radial}\},
\qquad \|w\|_{X_R}=\|z\|_{C(\overline{B_R})},
\]
where $w=0$ outside $B_R$. Thus $w\in X_R$ means that
$w/\tau_R^s$ extends continuously to the closed ball; its boundary
values always mean this extension. For $D>-2s$, define
\[
(\mathcal G^{(D)}_{s,R}f)(x)
=\int_{B_R}G_{s,R}(x,y)|y|^D f(y)\,dy,\qquad x\in B_R,
\]
again with zero values outside $B_R$.

The pointwise positive cone in $X_R$ and its interior are
\[
P_R=\{w\in X_R:w\ge0\},\qquad
\operatorname{int}P_R
=\{\tau_R^s z:\min_{\overline{B_R}}z>0\}.
\]
It is normal: $0\le w\le z$ implies
$\|w\|_{X_R}\le\|z\|_{X_R}$. Put $K_R=\mathcal G_{s,R}$ and,
for $u_+=\max\{u,0\}$, define
\begin{equation}\label{eq:composite}
\mathcal N_R(u)=\mathcal G^{(a)}_{s,R}
\left[\left(\mathcal G^{(b)}_{s,R}(u_+^q)\right)^p\right].
\end{equation}
Lemma \ref{lem:ball-maps} will give compactness and strong positivity
of $K_R$. The composition $\mathcal N_R$ eliminates the second
component; its properties are verified before Proposition
\ref{prop:ball-continuum}.

Our ball solutions satisfy
\begin{equation}\label{eq:ball-green-system}
\begin{cases}
u=\mu K_Ru+\mathcal G^{(a)}_{s,R}(v^p),\\
v=\mathcal G^{(b)}_{s,R}(u^q),
\end{cases}
\qquad x\in B_R,\quad \mu\ge0.
\end{equation}
Here $(u,v)\in P_R\times P_R$. Equivalently, they solve the zero exterior problem
\begin{equation}\label{eq:aux-system}
\begin{cases}
(-\Delta)^su=\mu u+|x|^av^p,\\
(-\Delta)^sv=|x|^bu^q,
\end{cases}
\quad\text{in }B_R,\qquad u=v=0\quad\text{on }\R^N\setminus B_R.
\end{equation}
The equations hold weakly in $B_R$ and pointwise off the origin.
We use $\mu=0$ for the Pohozaev argument and $\mu>0$ for the
construction. Ball functions are always extended by zero to the
whole complement, not only assigned a trace on $\partial B_R$.

\subsection{Positive solutions on balls}\label{sec:existence}

We first establish the Green mapping properties for every $D>-2s$.
We then exclude $\mu=0$ under \eqref{eq:supercritical} and construct
a connected set of solutions with $\mu>0$.

\begin{lemma}\label{lem:ball-maps}
Let $N\ge2$, $0<s<1$, $R>0$, and $D>-2s$.
For a bounded radial function
$f\in C(B_R\setminus\{0\})\cap L^\infty(B_R)$, set
$w=\mathcal G^{(D)}_{s,R}f$. Then:
\begin{enumerate}[label=\textup{(\roman*)}]
\item $w\in X_R\cap H^s_0(B_R)$, and
\[
\|w\|_{X_R}+\|w\|_{H^s(\R^N)}
\le C_{N,s,D,R}\|f\|_{L^\infty(B_R)}.
\]
\item If $(f_j)$ is a sequence of such functions with
$\sup_j\|f_j\|_{L^\infty(B_R)}<\infty$, then
$\mathcal G^{(D)}_{s,R}f_j$ has a subsequence converging in $X_R$.
\item If $f\ge0$ and $\int_{B_R}f\,dx>0$, then
\[
\min_{\overline{B_R}}\frac{w}{\tau_R^s}>0.
\]
\end{enumerate}
The constant in \textup{(i)} is independent of $f$.
\end{lemma}
\begin{proof}
Rotation invariance of \eqref{eq:boggio} shows that $w$ is radial.
Choose $r>2N/(N+2s)$ such that $rD>-N$. For $D<0$ this is possible
because
\[
\frac{2N}{N+2s}<\frac{N}{2s}<\frac{N}{-D};
\]
for $D\ge0$ there is no upper restriction on $r$. Consequently
\[
\||x|^Df\|_{L^r(B_R)}\le C_{N,D,r,R}\|f\|_\infty,
\qquad r'=\frac r{r-1}<\frac{2N}{N-2s}.
\]
First take a smooth radial $f$ supported away from zero and the
boundary. Test its weak equation by $w$. The fractional Poincar\'e
and Sobolev inequalities give
\[
\begin{aligned}
\|w\|_{H^s}^2
&\le C_R\mathcal E_s(w,w)
=C_R\int_{B_R}|x|^Dfw\\
&\le C_R\||x|^Df\|_r\|w\|_{r'}\\
&\le C_R\||x|^Df\|_r\,\mathcal E_s(w,w)^{1/2}
\le C_R\||x|^Df\|_r\|w\|_{H^s}.
\end{aligned}
\]
Dividing by $\|w\|_{H^s}$ when $w\ne0$ proves the bound; it is immediate for $w=0$.
For general $f$, radial cutoffs and mollification give $f_j$ with
$|f_j|\le\|f\|_\infty$, $f_j\to f$ a.e., and
$|x|^Df_j\to |x|^Df$ in $L^r$. 

The same estimate for differences
gives convergence of the weak solutions in $H^s_0(B_R)$.
By \eqref{eq:weighted-kernel-local} and
$G_{s,R}\le K_s$, their Green integrals converge pointwise to $w$,
this identifies the weak limit.

It remains to control the boundary quotient. Assume
$\|f\|_\infty\le1$. Formula \eqref{eq:weighted-kernel-local} gives
\[
\|w\|_{L^\infty(B_R)}
\le\sup_{x\in B_R}\int_{B_R}K_s(x-y)|y|^D\,dy\le C.
\]
For $0<\rho<R/4$ and $|x|\le\rho/2$, split the integral into
$B_\rho$ and its complement. The first part, at $x$ and at zero,
is at most $C\rho^{D+2s}$. On the second part, the kernel and
its $x$ derivatives are bounded for $|x|\le\rho/2$, with an
integrable bound in $y$. Hence
\[
|w(x)-w(0)|\le C\rho^{D+2s}+C_\rho|x|.
\]
Since $D+2s>0$, first choosing $\rho$ small and then $|x|$ small
gives a modulus at zero independent of $f$. On each compact subset
of $B_R\setminus\{0\}$, the source $|x|^Df$ is uniformly bounded;
interior estimates and the preceding global supremum bound give a
uniform H\"older modulus there.

Choose $\chi\in C_c^\infty(B_{R/2})$ radial, with $\chi=1$ on
$B_{R/4}$ and $0\le\chi\le1$, and write
\[
w=w_0+w_1,\qquad
w_0=\mathcal G_{s,R}(\chi|x|^Df),\quad
w_1=\mathcal G_{s,R}((1-\chi)|x|^Df).
\]
For $|x|\ge3R/4$ and $|y|\le R/2$,
\[
|x-y|\ge |x|-|y|\ge 3R/4-R/2=R/4.
\]
In the scaled version of \eqref{eq:boggio}, substitute
$t=\tau_R(x)\tau_R(y)\eta/R^2$. The resulting formula is
\[
\frac{G_{s,R}(x,y)}{\tau_R(x)^s}
=\frac{\kappa_{N,s}\tau_R(y)^s}{R^{2s}}
\int_0^1
\frac{\eta^{s-1}\,d\eta}
{(|x-y|^2+\tau_R(x)\tau_R(y)\eta/R^2)^{N/2}}.
\]
The displayed lower bound for $|x-y|$ permits one derivative in
$x$ up to $\partial B_R$. Since $D>-N$,
\[
\left\|\frac{w_0}{\tau_R^s}\right\|_
{C^1(\overline{B_R}\setminus B_{3R/4})}
\le C\int_{B_{R/2}}|y|^D\,dy\le C.
\]
The other source is bounded because $1-\chi=0$ near zero.
Ros-Oton-Serra \cite[Theorem 1.2]{ros-boundary}, applied to this
bounded source, and $\tau_R=(R+|x|)\delta_R$ imply
\[
\left\|\frac{w_1}{\tau_R^s}\right\|_{C^\eta(\overline{B_R})}
\le C\|(1-\chi)|x|^Df\|_\infty\le C
\]
for some $\eta>0$. Combining this bound with the origin and
interior moduli proves uniform boundedness and equicontinuity
of $w/\tau_R^s$. This proves the $X_R$ estimate, and
Arzel\`a-Ascoli proves compactness.

If $f\ge0$ and $\int f>0$, the positivity of the kernel gives $w>0$
inside $B_R$. At $\zeta\in\partial B_R$, the quotient formula
and Fatou's lemma give
\[
\left(\frac w{\tau_R^s}\right)(\zeta)
\ge \frac{\kappa_{N,s}}{sR^{2s}}
\int_{B_R}\frac{\tau_R(y)^s|y|^Df(y)}{|\zeta-y|^N}\,dy>0.
\]
The quotient is therefore positive everywhere on the compact
set $\overline{B_R}$, so its minimum is positive.
\end{proof}

The next argument uses the dilation of the Green kernel directly.
Only the bounds supplied by Lemma \ref{lem:ball-maps} are needed at
the origin and the boundary.

\begin{lemma}\label{lem:ball-pohozaev}
Assume \eqref{eq:parameters-main} and \eqref{eq:supercritical}.
For every $R>0$, there is no nonzero nonnegative pair
$(u,v)\in X_R\times X_R$ satisfying
\[
\begin{cases}
u=\mathcal G^{(a)}_{s,R}(v^p),\\
v=\mathcal G^{(b)}_{s,R}(u^q),
\end{cases}
\qquad x\in B_R.
\]
\end{lemma}
\begin{proof}
Green kernel scaling and the exponents in \eqref{eq:dilation}
reduce the proof to $R=1$.
If one component is zero, both are zero. Otherwise, positivity of
the kernel and Lemma \ref{lem:ball-maps} give
\[
u,v>0\quad\hbox{in }B_1,\qquad
u,v\le C(1-|x|)^s.
\]
The constants in this proof may depend on this fixed pair.
Put $F=|x|^bu^q$, $G=|x|^av^p$, and $\delta(x)=1-|x|$.
Kernel symmetry gives the finite positive number
\[
E_1:=\int_{B_1}Fu
=\iint_{B_1^2}F(x)G_s(x,y)G(y)\,dx\,dy
=\int_{B_1}Gv>0.
\]
Indeed, $u,v$ are bounded, $a,b>-N$, and $s(p+1),s(q+1)>0$.
These facts also imply
\begin{equation}\label{eq:ball-weighted-integrability}
\int_{B_1}\frac{Fu+Gv}{\delta}\,dx<\infty.
\end{equation}

Write $\omega(x)=1-|x|^2$ and $\zeta=\omega(x)\omega(y)$.
In \eqref{eq:boggio}, substitute the integration variable
$t=\lambda^2 r$ in $G_s(\lambda x,\lambda y)$.
Its new upper limit is
$\lambda^{-2}-|x|^2-|y|^2+\lambda^2|x|^2|y|^2$.
Differentiating at $\lambda=1$ therefore gives
\begin{equation}\label{eq:green-ball-dilation}
\begin{aligned}
(x\cdot\nabla_x+y\cdot\nabla_y)G_s(x,y)
&=(2s-N)G_s(x,y)-H_s(x,y),\\
H_s(x,y)
&=\frac{2\kappa_{N,s}(1-|x|^2|y|^2)\zeta^{s-1}}
{(|x-y|^2+\zeta)^{N/2}}>0
\qquad(x\ne y).
\end{aligned}
\end{equation}
The lower bound
\[
G_s(x,y)\ge
\frac{\kappa_{N,s}}s
\frac{\zeta^s}{(|x-y|^2+\zeta)^{N/2}}
\]
and $1-|x|^2|y|^2\le\omega(x)+\omega(y)$ yield
\[
H_s(x,y)\le
C\bigl(\delta(x)^{-1}+\delta(y)^{-1}\bigr)G_s(x,y).
\]
Multiplying by $F(x)G(y)$ and using
\eqref{eq:ball-weighted-integrability} proves
\begin{equation}\label{eq:green-ball-positive-remainder}
0<\iint_{B_1^2}F(x)G(y)H_s(x,y)\,dx\,dy<\infty.
\end{equation}

We justify the use of $x\cdot\nabla$ by compact cutoffs.
For a vector field $Y$, define
\[
\mathcal D_Y(x,y)
=Y(x)\cdot\nabla_xG_s(x,y)+Y(y)\cdot\nabla_yG_s(x,y).
\]
Put $A_s(x,y)=\kappa_{N,s}\zeta^{s-1}
(|x-y|^2+\zeta)^{-N/2}$.
Differentiating \eqref{eq:boggio} gives the paired formula
\[
\begin{aligned}
\mathcal D_Y
={}&-\bigl((N-2s)G_s+2\zeta A_s\bigr)
\frac{(Y(x)-Y(y))\cdot(x-y)}{|x-y|^2}\\
&-2\zeta A_s\left(
\frac{x\cdot Y(x)}{\omega(x)}+
\frac{y\cdot Y(y)}{\omega(y)}\right).
\end{aligned}
\]
The preceding lower bound gives $\zeta A_s\le sG_s$.
Using also $\omega\ge\delta$, we obtain
\begin{equation}\label{eq:green-vector-bound}
|\mathcal D_Y(x,y)|
\le CG_s(x,y)\left(
\frac{|Y(x)-Y(y)|}{|x-y|}
+\frac{|Y(x)|}{\delta(x)}+\frac{|Y(y)|}{\delta(y)}
\right).
\end{equation}
For $X\in C_c^1(B_1\setminus\{0\};\R^N)$, differentiation of the
two potentials and addition on the common set $|x-y|>\varepsilon$
give
\begin{equation}\label{eq:ball-paired-green}
\int_{B_1}FX\cdot\nabla u+\int_{B_1}GX\cdot\nabla v
=\iint_{B_1^2}F(x)G(y)\mathcal D_X(x,y)\,dx\,dy.
\end{equation}
Here $u,v$ are smooth on compact subsets of $B_1\setminus\{0\}$:
the local iteration in Lemma \ref{lem:regularity} applies because
$G_s-K_s$ is smooth there and the components are positive.
On such compact subsets the constant subtraction in the proof of
Lemma \ref{lem:virial} justifies the differentiated potentials.
On the right, \eqref{eq:green-vector-bound} bounds the paired
integrand by $C_XF(x)G(y)G_s(x,y)$, which is integrable.
Thus the common deletion radius may tend to zero, also when $2s\le1$.

Choose cutoffs $\chi_\varepsilon=0$ on $B_\varepsilon$ and
$\chi_\varepsilon=1$ outside $B_{2\varepsilon}$, and
$\psi_\tau=0$ where $\delta\le\tau$ and $\psi_\tau=1$ where
$\delta\ge2\tau$. Take their gradients bounded by $C/\varepsilon$
and $C/\tau$, respectively, and set
$X_{\varepsilon,\tau}(x)=x\chi_\varepsilon(x)\psi_\tau(x)$.
For $0<\varepsilon,\tau<1/8$, write
$X_{\varepsilon,\tau}-x=Y_0+Y_\partial$, where the two terms
are supported in $B_{2\varepsilon}$ and $\{\delta<2\tau\}$.
They satisfy
\[
\|\nabla Y_0\|_\infty\le C,\qquad
\|\nabla Y_\partial\|_\infty\le C/\tau,\qquad
|Y_0|\le C\varepsilon,\quad |Y_\partial|\le C.
\]
In \eqref{eq:green-vector-bound}, the difference $Y(x)-Y(y)$
vanishes unless at least one point belongs to the support of $Y$.
For the boundary term also use $\tau^{-1}\le2\delta^{-1}$ on
that support. Integrating the resulting bounds gives
\[
\begin{aligned}
&\iint_{B_1^2}F(x)G(y)
|\mathcal D_{X_{\varepsilon,\tau}}-\mathcal D_{\mathrm{id}}|\,dx\,dy\\
&\quad\le C\int_{B_{2\varepsilon}}(Fu+Gv)\,dx
+C\int_{\{\delta<2\tau\}}\frac{Fu+Gv}{\delta}\,dx\\
&\quad\le C\left(\varepsilon^{N+a}+\varepsilon^{N+b}
+\tau^{s(p+1)}+\tau^{s(q+1)}\right)\longrightarrow0.
\end{aligned}
\]
Here $\mathrm{id}(x)=x$. All four powers are positive.
Consequently the right side of
\eqref{eq:ball-paired-green} converges, by
\eqref{eq:green-ball-dilation}, to
$(2s-N)E_1-\iint F(x)G(y)H_s(x,y)\,dx\,dy$.

On its left, integrate by parts before taking the cutoff limit:
\[
\int FX_{\varepsilon,\tau}\cdot\nabla u
=-\frac1{q+1}\int u^{q+1}
\operatorname{div}(|x|^bX_{\varepsilon,\tau}),
\]
and use the analogous identity for $G,v,a,p$.
The derivatives of the two cutoffs contribute at most the same
four vanishing powers just displayed. Since
$\operatorname{div}(|x|^bx)=(N+b)|x|^b$, the limit is
$-[(N+b)/(q+1)+(N+a)/(p+1)]E_1$.
Equating both limits proves
\begin{equation}\label{eq:ball-pohozaev}
\left(\frac{N+b}{q+1}+\frac{N+a}{p+1}-(N-2s)\right)E_1
=\iint_{B_1^2}F(x)G(y)H_s(x,y)\,dx\,dy>0.
\end{equation}
The coefficient on the left is nonpositive under
\eqref{eq:supercritical}, a contradiction.
\end{proof}

Lemma \ref{lem:ball-pohozaev} excludes $\mu=0$ in the critical and
supercritical range. We now construct positive solutions with $\mu>0$.

\subsubsection{An unbounded connected set of positive solutions}

For $u\ge0$, substitute
$v=\mathcal G^{(b)}_{s,R}(u^q)$ from the second identity of
\eqref{eq:ball-green-system} into its first identity. This gives
$u=\mu K_Ru+\mathcal N_R(u)$, with \eqref{eq:composite}.
Since $\|u\|_\infty\le R^{2s}\|u\|_{X_R}$, two applications of
Lemma \ref{lem:ball-maps}\textup{(i)} give
\[
\|\mathcal G^{(b)}_{s,R}(u_+^q)\|_\infty
\le C_R\|u\|_{X_R}^q,
\]
and hence
\begin{equation}\label{eq:degree}
\|\mathcal N_R(u)\|_{X_R}
\le C_R\|\mathcal G^{(b)}_{s,R}(u_+^q)\|_\infty^p
\le C_R\|u\|_{X_R}^{pq}.
\end{equation}
Positive powers are continuous on uniformly bounded sets of
nonnegative functions in the uniform norm, so
Lemma \ref{lem:ball-maps}\textup{(i)} also makes $\mathcal N_R$
continuous. The displayed inner bound and part \textup{(ii)},
applied to the outer Green operator, make $\mathcal N_R$ compact.
Moreover, linearity of both Green operators gives
$\mathcal N_R(tu)=t^{pq}\mathcal N_R(u)$ for $t\ge0$.
All constants here depend only on $N,s,a,b,p,q,R$.

Let $\lambda_1(R)>0$ be the first Dirichlet eigenvalue of
$(-\Delta)^s$ on $B_R$. Strong positivity of $K_R$ implies that
its largest eigenvalue is the simple eigenvalue $1/\lambda_1(R)$,
and no other eigenvalue has a nonnegative eigenfunction. Its positive eigenfunction
is radial, by rotational invariance and simplicity.

For a real parameter $\lambda$, consider the single equation
\begin{equation}\label{eq:bifurcation}
u=e^\lambda K_Ru+\mathcal N_R(u),\qquad u\in X_R.
\end{equation}
A positive solution of \eqref{eq:bifurcation}, together with
$v=\mathcal G^{(b)}_{s,R}(u^q)$ and $\mu=e^\lambda$, gives
\eqref{eq:ball-green-system}.

\begin{proposition}\label{prop:ball-continuum}
Assume \eqref{eq:parameters-main} and \eqref{eq:supercritical}.
For each $R>0$, the positive solutions of \eqref{eq:bifurcation}
have a connected component
$\mathcal C_R\subset\R\times\operatorname{int}P_R$
whose closure contains $(\log\lambda_1(R),0)$.
On this component,
\[
0<\mu=e^\lambda<\lambda_1(R),\qquad
\sup_{(\lambda,u)\in\mathcal C_R}\|u\|_{X_R}=\infty.
\]
The closure and connectedness use the product topology.
\end{proposition}
\begin{proof}
Let $\lambda_0=\log\lambda_1(R)$ and $L_\lambda=I-e^\lambda K_R$.
We verify the hypotheses of Cano-Casanova-L\'opez-G\'omez-Molina-Meyer
\cite[Theorem 4.1]{cano}. The cone $P_R$ is normal with nonempty
interior. The family $e^\lambda K_R$ is compact and real analytic,
and $L_\lambda$ is invertible for
$e^\lambda\|K_R\|<1$, by its Neumann series.
The remainder $-\mathcal N_R$ is continuous and compact on
bounded sets. By \eqref{eq:degree} and $pq>1$,
\[
\frac{\|\mathcal N_R(u)\|_{X_R}}{\|u\|_{X_R}}
\le C_R\|u\|_{X_R}^{pq-1}\longrightarrow0
\quad(u\to0),
\]
uniformly in $\lambda$ on compact intervals. Finally, a nonzero
nonnegative solution satisfies
$u\ge e^\lambda K_Ru\in\operatorname{int}P_R$, by
Lemma \ref{lem:ball-maps}(iii). This verifies the strong positivity
condition on solutions required in that theorem.

Let $\phi_1>0$ be the first eigenfunction. Strong positivity of
$K_R$ gives its algebraically simple eigenvalue
$1/\lambda_1(R)$, and
\[
\ker L_{\lambda_0}=\operatorname{span}\{\phi_1\},\qquad
\partial_\lambda L_{\lambda_0}\phi_1=-\phi_1.
\]
For $w\in X_R$, kernel symmetry and
$K_R\phi_1=\phi_1/\lambda_1(R)$ give
\[
\int_{B_R}\phi_1L_{\lambda_0}w=0,\qquad
\int_{B_R}\phi_1\partial_\lambda L_{\lambda_0}\phi_1
=-\int_{B_R}\phi_1^2<0.
\]
Thus the crossing has algebraic multiplicity one. Moreover,
\[
\ker L_\lambda\cap(P_R\setminus\{0\})\ne\varnothing
\quad\Longleftrightarrow\quad \lambda=\lambda_0,
\]
because no other eigenvalue of a strongly positive compact
operator has a nonnegative eigenfunction. The cited theorem
therefore gives a connected component $\mathcal C_R$ that is
unbounded in $\R\times X_R$ and whose closure contains
$(\lambda_0,0)$.

Test the first weak equation by $\phi_1$. Symmetry of the energy
form and the eigenfunction equation yield
\[
(\lambda_1(R)-\mu)\int_{B_R}u\phi_1
=\int_{B_R}|x|^av^p\phi_1>0.
\]
Since $\int u\phi_1>0$, every point of the component satisfies
$0<\mu<\lambda_1(R)$.

Suppose that $\|u\|_{X_R}$ were bounded on this component.
Its unboundedness and $\lambda<\lambda_0$ would then give
$(\lambda_j,u_j)\in\mathcal C_R$ with $\lambda_j\to-\infty$.
In
\[
u_j=e^{\lambda_j}K_Ru_j+\mathcal N_R(u_j),
\]
the first term tends to zero and the second has a convergent
subsequence. Hence $u_j\to u_*=\mathcal N_R(u_*)$ in $X_R$.
Dividing the norm inequality by $\|u_j\|_{X_R}>0$ gives
\[
1\le e^{\lambda_j}\|K_R\|+
C_R\|u_j\|_{X_R}^{pq-1}.
\]
This excludes $u_*=0$. Setting
$v_*=\mathcal G^{(b)}_{s,R}(u_*^q)$ now gives a nonzero
nonnegative pair with $\mu=0$, contrary to Lemma
\ref{lem:ball-pohozaev}. Thus the $X_R$ norms are unbounded.
\end{proof}

On $\mathcal C_R$, $v=\mathcal G^{(b)}_{s,R}(u^q)$ is determined by
$u$, and the map $(\lambda,u)\mapsto u(0)+v(0)$ is continuous by
Lemma \ref{lem:ball-maps}. To select $u(0)+v(0)=1$, it suffices
to prove that the central sum is unbounded on $\mathcal C_R$.
We next show that a bound for this sum controls the supremum
norms on each fixed ball.

\subsection{Uniform bounds and compactness}\label{sec:green-profiles}

\subsubsection{Green potentials on expanding balls}

The following estimates hold under \eqref{eq:parameters-main}
for all $0<s<1$. The weighted critical inequality is needed for
the construction of positive ball solutions in Section \ref{sec:existence}, but not
for the compactness estimates below.

\begin{lemma}\label{lem:inner-central}
Let $N\ge2$, $0<s<1$, $R>0$, and let $h\ge0$ be radial in $B_R$.
If $z=\mathcal G_{s,R}h$ and $z(0)<\infty$, then
\begin{equation}\label{eq:green-mass}
\int_{B_R}z\,dx\le C_{N,s}R^Nz(0).
\end{equation}
The constant is independent of $R$ and $h$.
\end{lemma}
\begin{proof}
For $R=1$, the potential of the constant source is
\[
\int_{B_1}G_s(x,y)\,dx
=\frac{\Gamma(N/2)}{4^s\Gamma(1+s)\Gamma(N/2+s)}
(1-|y|^2)^s.
\]
This is the ball solution formula in
\cite[formula (1.4)]{ros-boundary}, together with kernel symmetry.
At $x=0$, the integration interval is
$0<t<1-|y|^2$, so $|y|^2+t\le1$. Hence
\[
G_s(0,y)\ge\kappa_{N,s}\int_0^{1-|y|^2}t^{s-1}\,dt
=\frac{\kappa_{N,s}}s(1-|y|^2)^s.
\]
Combining the two formulas and using symmetry gives
$\int_{B_1}G_s(x,y)\,dx\le C_{N,s}G_s(0,y)$.
After scaling,
\[
\int_{B_R}G_{s,R}(x,y)\,dx
\le C_{N,s}R^NG_{s,R}(0,y).
\]
Multiply by $h(y)\ge0$, integrate in $y$, and apply Tonelli's
theorem. The result is
\[
\int_{B_R}z(x)\,dx
\le C_{N,s}R^N\int_{B_R}G_{s,R}(0,y)h(y)\,dy
=C_{N,s}R^Nz(0).
\]
\end{proof}

\begin{lemma}\label{lem:green-limits}\label{lem:weighted-local-continuity}
Let $N\ge2$, $0<s<1$, and $a,b>-2s$.
Let $D_j$ be balls containing zero, with
$d_j=\dist(0,\partial D_j)\to\infty$, and set
\[
Z_j(x)=\int_{D_j}G_{D_j}(x,y)H_j(y)\,dy,\qquad H_j\ge0.
\]
Assume $\sup_jZ_j(0)<\infty$ and, for each $L>0$,
\[
H_j(y)\le C_L(1+|y|^a+|y|^b)\quad\hbox{a.e. on }B_{2L}
\]
for all sufficiently large $j$. For every
$0<\sigma<\min\{1,2s,2s+a,2s+b\}$,
$\|Z_j\|_{C^\sigma(\overline B_L)}$ is bounded for large $j$.
If, along a subsequence, $Z_j\to Z$ locally uniformly and
$H_j\to H$ locally uniformly away from zero, then
\begin{equation}\label{eq:limit-positive-constant}
Z=I_{2s}H+C_*,\qquad C_*\ge0,
\end{equation}
with finite potentials at every point. Bounded local sources are
included by taking $a=b=0$.
\end{lemma}
\begin{proof}
Differentiate the scaled formula \eqref{eq:boggio}. With
$\zeta=\tau_R(x)\tau_R(y)/R^2$, the upper-limit contribution is
bounded using
\[
\zeta^s(|x-y|^2+\zeta)^{-N/2}
\le s\int_0^\zeta t^{s-1}(|x-y|^2+t)^{-N/2}\,dt,\qquad
\frac{|\nabla_x\zeta|}{\zeta}\le\frac2{\delta_R(x)}.
\]
The derivative of the integrand costs at most $C/|x-y|$.
Translation therefore gives, on every ball $\Omega$,
\begin{equation}\label{eq:green-log-derivative}
|\nabla_xG_\Omega(x,y)|
\le C_{N,s}\left(|x-y|^{-1}+\dist(x,\partial\Omega)^{-1}\right)
G_\Omega(x,y).
\end{equation}

We first estimate the contribution of nearby sources.
For $\Omega\supset B_{4L}$, $D>-2s$, and bounded $f$, put
$P_\Omega(x)=\int_{B_{2L}}G_\Omega(x,y)|y|^Df(y)\,dy$.
We claim that
\begin{equation}\label{eq:weighted-local-continuity}
\|P_\Omega\|_{C^\sigma(\overline B_L)}
\le C_{N,s,D,L,\sigma}\|f\|_\infty,
\qquad 0<\sigma<\min\{1,2s,2s+D\},
\end{equation}
independently of $\Omega$. Scale $L$ to one and assume
$\|f\|_\infty\le1$. Direct integration gives
\[
\int_{B_r(x)\cap B_2}|y|^D\,dy\le Cr^{N+\min\{D,0\}}.
\]
For $D<0$, use $B_r(x)\subset B_{3r}$ when $|x|\le2r$;
otherwise $|y|>r$ on $B_r(x)$. For $D\ge0$ the weight is bounded.
Set $\nu=2s+\min\{D,0\}>0$. Dyadic integration with $G_\Omega\le K_s$
gives a uniform supremum bound. For $h=|x-x'|<1/4$, the parts
within distance $2h$ of either point contribute at most $Ch^\nu$.
On the remaining annuli apply \eqref{eq:green-log-derivative};
the boundary distance is at least three. Summation gives
\[
|P_\Omega(x)-P_\Omega(x')|
\le C\left(h^\nu+h+
h\sum_{\substack{k\ge0\\2^kh\le4}}(2^kh)^{\nu-1}\right)
\le C_\sigma h^\sigma.
\]
If $\nu=1$, use $h(1+|\log h|)\le C_\sigma h^\sigma$ for $\sigma<1$.
The supremum bound handles $h\ge1/4$, proving
\eqref{eq:weighted-local-continuity}.
Write each local source as
$H_j=(1+|y|^a+|y|^b)f_j$ with $0\le f_j\le C_L$, and apply this
estimate with $D=0,a,b$.

For the distant sources fix $K>0$, take $L>2K$, and let
$Z_{j,>L}$ denote their potential. On segments in $B_K$, with
$|y|>L$ and $d_j>2K$, integration of
\eqref{eq:green-log-derivative} gives
\[
\left|\log\frac{G_{D_j}(x,y)}{G_{D_j}(0,y)}\right|
\le C_K(L^{-1}+d_j^{-1}).
\]
Since the sources are nonnegative, $Z_{j,>L}(0)\le Z_j(0)$.
Exponentiating, and then integrating against $H_j$, yields
\begin{equation}\label{eq:far-flat}
|Z_{j,>L}(x)-Z_{j,>L}(0)|
\le C_K(L^{-1}+d_j^{-1})Z_j(0)\qquad(x\in B_K).
\end{equation}
The same computation between two points gives a uniform Lipschitz
bound for this distant part. Together with
\eqref{eq:weighted-local-continuity} and the central bound, this
proves the stated local bounds.

It remains to identify a convergent subsequence.
Write $D_j=B_{R_j}(c_j)$. For fixed $x\ne y$ its Green upper limit
satisfies
\[
\frac{(R_j^2-|x-c_j|^2)(R_j^2-|y-c_j|^2)}{R_j^2}
\ge(d_j-|x|)(d_j-|y|)\longrightarrow\infty.
\]
Thus $G_{D_j}(x,y)\to K_s(x-y)$. On a fixed $B_L$, convergence of
the sources away from zero and their common weighted bound give
\[
Z_{j,\le L}\longrightarrow I_{2s}(H\mathbf1_{B_L})
\quad\hbox{locally uniformly}.
\]
To check uniformity at zero, remove $B_\rho$: its contribution is at
most $C(\rho^{2s}+\rho^{2s+a}+\rho^{2s+b})$ by
\eqref{eq:weighted-kernel-local}, also for evaluation points outside
$B_{2\rho}$ by $|x-y|\ge|x|/2$.
On the remaining compact set the sources are uniformly bounded;
the integrable diagonal singularity and the preceding kernel
convergence give the claimed limit.

Nonnegativity and \eqref{eq:far-flat} now imply
\[
0\le I_{2s}(H\mathbf1_{B_L})\le Z,\qquad
\operatorname{osc}_{B_K}\bigl(Z-I_{2s}(H\mathbf1_{B_L})\bigr)
\le C_K/L.
\]
Let $L\to\infty$. Monotone convergence gives a finite potential
$I_{2s}H\le Z$, and the oscillation bound makes $Z-I_{2s}H$
constant. This proves \eqref{eq:limit-positive-constant}.
\end{proof}

We will also need a boundary estimate whose constant is
independent of the radius and center of a ball.

\begin{lemma}\label{lem:uniform-ball-boundary}
Let $0<s<1$, let $\Omega$ be a ball, and let $y_0\in\partial\Omega$.
Suppose $w\in H^s(\R^N)\cap C(\R^N)\cap L^\infty(\R^N)$,
$w=0$ outside $\Omega$, and
$|(-\Delta)^sw|\le M$ weakly in $\Omega\cap B_8(y_0)$, where $M\ge0$.
Then
\begin{equation}\label{eq:uniform-boundary-vanishing}
|w(x)|\le C_{N,s}(\|w\|_\infty+M)\dist(x,\partial\Omega)^s
\qquad(x\in\Omega\cap B_1(y_0)),
\end{equation}
and the zero extension satisfies
\begin{equation}\label{eq:uniform-boundary-holder}
\|w\|_{C^{s/2}(\overline{B_{1/2}(y_0)})}
\le C_{N,s}(\|w\|_\infty+M).
\end{equation}
The constants are independent of the radius and center of $\Omega$.
\end{lemma}
\begin{proof}
Ros-Oton-Serra \cite[Lemma 2.6]{ros-boundary} supply a bounded
radial barrier $\varphi\in H^s_{\rm loc}(\R^N)\cap C(\R^N)$ with
\[
\varphi=0\text{ on }B_1,\qquad
\varphi\ge1\text{ on }\R^N\setminus B_4,\qquad
(-\Delta)^s\varphi\ge1\text{ in }B_4\setminus\overline B_1,
\]
and $0\le\varphi(z)\le C(|z|-1)^s$ in that annulus.
Fix $x\in\Omega\cap B_1(y_0)$, take a nearest boundary point
$\zeta$, and put $c=\zeta+\nu(\zeta)$. Convexity of the ball
implies $B_1(c)\cap\Omega=\varnothing$. Also
\[
|\zeta-y_0|\le |x-\zeta|+|x-y_0|<2,\qquad
B_4(c)\subset B_7(y_0)\subset B_8(y_0).
\]
Let $A=\|w\|_\infty+M$. On
$\Omega\cap(B_4(c)\setminus\overline B_1(c))$,
\[
(-\Delta)^s(A\varphi(\cdot-c)\pm w)\ge A-M\ge0.
\]
On its complement the same functions are nonnegative: use
$w=0$ outside $\Omega$ and $\varphi\ge1$ outside $B_4(c)$.
Testing their negative parts, which are in $H^s$ and supported
in this bounded set, gives $|w|\le A\varphi(\cdot-c)$.
Since $x=\zeta-\delta_\Omega(x)\nu(\zeta)$,
\[
|x-c|=1+\delta_\Omega(x),\qquad
|w(x)|\le CA\delta_\Omega(x)^s.
\]
This proves \eqref{eq:uniform-boundary-vanishing} with a constant
depending only on $N,s$.

For $x,x'\in B_{1/2}(y_0)$ let $h=|x-x'|<1/16$ and
$\rho=\sqrt h$. If one point is outside $\Omega$, the other
has boundary distance at most $h$, so the last estimate gives
$|w(x)-w(x')|\le CAh^s$. If both lie inside and one boundary
distance is at most $2\rho$, the other is at most
$2\rho+h\le3\rho$, giving $CA\rho^s=CAh^{s/2}$.
In the remaining case $B_\rho(x)\subset\Omega$ and
$x'\in B_{\rho/2}(x)$. Rescale this ball to a unit ball and apply
\cite[Proposition 2.3]{ros-boundary} with H\"older exponent $s$:
\[
|w(x)-w(x')|
\le C\rho^{-s}(\|w\|_\infty+\rho^{2s}M)h^s
\le CAh^{s/2}.
\]
All balls used lie in $B_8(y_0)$. The bound
$|w(x)-w(x')|\le2\|w\|_\infty$ handles $h\ge1/16$ and
completes \eqref{eq:uniform-boundary-holder}.
\end{proof}

\subsubsection{One-dimensional limits}

A radial rescaling has a simple geometric limit. If
$x_j=r_je_1$ and $\ell_j/r_j\to0$, then, uniformly on each
fixed ball in the $z$ variable,
\begin{equation}\label{eq:radial-flattening}
\frac{|x_j+\ell_jz|-r_j}{\ell_j}
=z_1+O\!\left(\frac{\ell_j}{r_j}|z|^2\right).
\end{equation}
For a radial function $u_j$ and a fixed amplitude factor $A_j$, set
$U_j(z)=A_ju_j(x_j+\ell_jz)$ and
$a_j(z)=(|x_j+\ell_jz|-r_j)/\ell_j$. Then
\[
U_j(z)=U_j(a_j(z)e_1),\qquad a_j(z)\longrightarrow z_1
\]
locally uniformly. If $U_j$ is locally equicontinuous, every limit
satisfies $U(z)=U(z_1e_1)$. The same argument applies to $v_j$,
both in the interior and near the boundary.

\begin{lemma}
\label{lem:profile-exclusion}
Let $N\ge2$, $0<s<1$, $p,q>0$, $pq>1$, and assume
\begin{equation}\label{eq:radial-window}
\frac1{p+1}+\frac1{q+1}>1-2s.
\end{equation}
For $a_0,b_0>0$ and $C_U,C_V\ge0$, there is no bounded
continuous nonnegative nonzero pair, depending only on $x_1$, with
finite potentials satisfying
\[
\begin{cases}
U=a_0I_{2s}(V^p)+C_U,\\
V=b_0I_{2s}(U^q)+C_V,
\end{cases}
\qquad x\in\R^N.
\]
\end{lemma}
\begin{proof}
If one component were identically zero, its potential equation
would force the other to be zero. Since the pair is nonzero,
the positive kernel makes both components positive everywhere.
If $C_U>0$, then $U\ge C_U$ and
\[
I_{2s}(U^q)(0)
\ge cC_U^q\int_{|y|>1}|y|^{2s-N}\,dy
=c' C_U^q\int_1^\infty r^{2s-1}\,dr=\infty,
\]
contrary to finiteness of $V$. Thus $C_U=0$, and the other
equation gives $C_V=0$ by the same calculation.

Write $y=(\tau,y')\in\R\times\R^{N-1}$. For $t\ne0$,
polar coordinates in $\R^{N-1}$ give
\begin{equation}\label{eq:transverse-riesz}
\int_{\R^{N-1}}(t^2+|y'|^2)^{-(N-2s)/2}\,dy'
=\begin{cases}
c_{N,s}^*|t|^{2s-1},&0<s<1/2,\\
+\infty,&1/2\le s<1.
\end{cases}
\end{equation}
Indeed the radial integrand at infinity is comparable to
$r^{2s-2}$. For $s<1/2$, the substitution $y'=|t|z'$
and the beta integral yield
\[
c_{N,s}^*=\pi^{(N-1)/2}
\frac{\Gamma((1-2s)/2)}{\Gamma((N-2s)/2)},\qquad
c_{N,s}c_{N,s}^*=c_{1,s}.
\]
If $s\ge1/2$, integrate the positive one-dimensional source
over any interval in $\tau$ not containing the evaluation
coordinate. Tonelli's theorem and \eqref{eq:transverse-riesz}
make the potential infinite, a contradiction.

For $s<1/2$, Tonelli's theorem instead converts the two
$N$-dimensional potentials into one-dimensional Riesz
potentials. Put
\[
\widetilde U(t)=(a_0b_0^p)^{1/(pq-1)}U(te_1),\qquad
\widetilde V(t)=(b_0a_0^q)^{1/(pq-1)}V(te_1).
\]
The coefficient identities
$a_0(a_0b_0^p)^{1/(pq-1)}
=(b_0a_0^q)^{p/(pq-1)}$ and its counterpart give
\[
\widetilde U=I_{2s}(\widetilde V^p),\qquad
\widetilde V=I_{2s}(\widetilde U^q)\quad\text{on }\R.
\]
Theorem \ref{thm:riesz-liouville} applies with dimension one,
$a=b=0$, and $0<s<1/2$. Its strict subcritical inequality is
exactly \eqref{eq:radial-window}, so it excludes this pair.
\end{proof}

\begin{lemma}\label{lem:half-line}
Let $0<s<1$, $a_0,b_0>0$, and $p,q>0$.
There is no bounded continuous nonnegative nonzero pair on $\R$,
zero on $(-\infty,0]$, satisfying distributionally
\[
\begin{cases}
(-\Delta)^sU=a_0V^p,\\
(-\Delta)^sV=b_0U^q,
\end{cases}
\qquad t\in(0,\infty).
\]
\end{lemma}
\begin{proof}
The strong maximum principle makes both components positive
on $(0,\infty)$; if one were zero, the other equation would
force both to be zero. Interior regularity and iteration on
compact subintervals then give smoothness there. To obtain a
lower boundary bound, choose a nonzero nonnegative smooth source
supported in $(1,2)$ below each of $a_0V^p$ and $b_0U^q$.
Compare its zero exterior Green potential on $(0,3)$ with $U$
and $V$, respectively. The ball Green formula in dimension one,
after translation, gives a positive limit of this potential
divided by $t^s$ at zero. Therefore
\begin{equation}\label{eq:half-line-hopf}
U(t),V(t)\ge ct^s\qquad(0<t<t_*).
\end{equation}

For $\lambda>0$ and $t<\lambda$, set
\[
W_{1,\lambda}(t)=U(2\lambda-t)-U(t),\qquad
W_{2,\lambda}(t)=V(2\lambda-t)-V(t).
\]
They are nonnegative for $t\le0$ and antisymmetric across
$\lambda$. Subtract the original equations from their
reflections. The result on $(0,\lambda)$ is
\[
(-\Delta)^sW_{1,\lambda}=c_{1,\lambda}W_{2,\lambda},\qquad
(-\Delta)^sW_{2,\lambda}=c_{2,\lambda}W_{1,\lambda},
\]
where, with positive arguments,
\[
\begin{aligned}
c_{1,\lambda}(t)&=
a_0p\int_0^1[(1-\theta)V(t)+\theta V(2\lambda-t)]^{p-1}\,d\theta,\\
c_{2,\lambda}(t)&=
b_0q\int_0^1[(1-\theta)U(t)+\theta U(2\lambda-t)]^{q-1}\,d\theta.
\end{aligned}
\]
Suppose $m_i=W_{i,\lambda}(t_i)<0$ is a minimum on $t<\lambda$.
It is attained in $(0,\lambda)$ because the boundary values at
zero and $\lambda$ are nonnegative. Pair the fractional kernels
at $y$ and $2\lambda-y$. For $y<\lambda$,
\[
|t_i-y|^{-1-2s}-|t_i-(2\lambda-y)|^{-1-2s}>0,
\qquad m_i-W_{i,\lambda}(y)\le0.
\]
Dropping this nonpositive contribution gives
\begin{equation}\label{eq:antisymmetric-minimum}
\begin{aligned}
(-\Delta)^sW_{i,\lambda}(t_i)
&\le 2C(1,s)m_i\int_{-\infty}^\lambda
|t_i-(2\lambda-y)|^{-1-2s}\,dy\\
&=\frac{C(1,s)}s(\lambda-t_i)^{-2s}m_i
\le \frac{C(1,s)}s\delta^{-2s}m_i,
\end{aligned}
\end{equation}
if $\lambda-t_i\le\delta$. The first equation at $t_1$
forces $W_{2,\lambda}(t_1)<0$, so its negative minimum $m_2$
also exists. Applying the two inequalities at their minima
and using $m_2\le W_{2,\lambda}(t_1)$,
$m_1\le W_{1,\lambda}(t_2)$ gives
\begin{equation}\label{eq:narrow-product}
\frac{C(1,s)^2}{s^2}\delta^{-4s}
\le c_{1,\lambda}(t_1)c_{2,\lambda}(t_2).
\end{equation}

For small $\lambda$, choose $\delta=\lambda$.
Because $W_{2,\lambda}(t_1)<0$,
$V(t_1)>V(2\lambda-t_1)$, and
$\lambda<2\lambda-t_1<2\lambda$. Estimate
\eqref{eq:half-line-hopf} bounds both arguments of the first
coefficient below by $c\lambda^s$. If $p<1$, use this lower
bound in the negative power $p-1$; if $p\ge1$, use boundedness
of $V$. The same argument for the second coefficient yields
\[
c_{1,\lambda}(t_1)\le C\lambda^{-s(1-p)_+},\qquad
c_{2,\lambda}(t_2)\le C\lambda^{-s(1-q)_+}.
\]
Substitution into \eqref{eq:narrow-product} would imply
$1\le C\lambda^{\,s[4-(1-p)_+-(1-q)_+]}$.
Its exponent is positive because $p,q>0$; this is impossible
for small $\lambda$. Reflection therefore starts.

Suppose reflection stops at a finite $\lambda_0$.
At an interior zero of either nonnegative difference, the
paired-kernel expression is strictly negative unless that
difference vanishes identically. The equation has a
nonnegative right-hand side, and
$W_{i,\lambda_0}(0)>0$, so both differences are strictly
positive on $(0,\lambda_0)$.
They remain positive on every compact subset away from the
plane for nearby values of $\lambda$; their positive values at
zero also exclude negative points near zero. Any negative
minimum must therefore lie in $(\lambda-\delta,\lambda)$.
All its arguments lie in a fixed compact subinterval of
$(0,\infty)$, where $U,V$ have positive minima. Both
coefficients are consequently bounded, also for sublinear
powers. Taking $\delta$ sufficiently small contradicts
\eqref{eq:narrow-product}. Thus reflection never stops.

For $0<t<t'$, use $\lambda=(t+t')/2$ to obtain
$U(t)\le U(t')$ and $V(t)\le V(t')$. Boundedness gives limits
$U_\infty,V_\infty>0$ as $t\to\infty$.
For $0\le\phi\in C_c^\infty(\R)$ with $\int\phi>0$, translate
the first equation by $L$. For all large $L$,
\[
\int_\R U(t+L)(-\Delta)^s\phi(t)\,dt
=a_0\int_\R V(t+L)^p\phi(t)\,dt.
\]
Since $U$ is bounded and $(-\Delta)^s\phi\in L^1(\R)$,
dominated convergence sends the left side to
$U_\infty\int(-\Delta)^s\phi=0$. The right side tends to
$a_0V_\infty^p\int\phi>0$, a contradiction.
\end{proof}

\subsubsection{Bounds from the central value}\label{sec:radial-compactness}

For the remaining estimates assume \eqref{eq:parameters-main}
and \eqref{eq:radial-window}, which is automatic for
$s\ge1/2$. We first control concentration away from the origin,
then use the central value to control the origin itself.

\begin{lemma}\label{lem:radial-annular}
Assume \eqref{eq:parameters-main} and \eqref{eq:radial-window},
and fix $C_0>0$.
Let $(u,v)\in X_R\times X_R$ be a positive radial solution of
\eqref{eq:ball-green-system}, with $\mu\ge0$ and $\mu R^{2s}\le C_0$.
Put
\[
\begin{gathered}
\alpha_0=\frac{2s(p+1)}{pq-1},\qquad
\beta_0=\frac{2s(q+1)}{pq-1},\\
\mathcal M(r)=r^{(\alpha-\alpha_0)/\alpha_0}u(r)^{1/\alpha_0}
+r^{(\beta-\beta_0)/\beta_0}v(r)^{1/\beta_0}.
\end{gathered}
\]
Here $\alpha,\beta$ are defined in \eqref{eq:scales}.
There is $C=C(N,s,a,b,p,q,C_0)$, independent of $R,\mu,u,v$, such that
\begin{equation}\label{eq:annular-small}
\mathcal M(r)\min\{r,R-r\}\le C\qquad(0<r<R).
\end{equation}
Consequently,
\begin{equation}\label{eq:radial-decay-small}
u(r)\le Cr^{-\alpha},\qquad v(r)\le Cr^{-\beta}
\qquad(0<r\le R/2).
\end{equation}
\end{lemma}
\begin{proof}
Suppose \eqref{eq:annular-small} fails. For integers $k\to\infty$,
choose a point $x_0$ in a punctured ball such that
\[
\mathcal M(|x_0|)\min\{|x_0|,R-|x_0|\}>4k.
\]
If $\mathcal M\le2\mathcal M(|x_0|)$ on
$B_{k/\mathcal M(|x_0|)}(x_0)$, keep $x_0$. Otherwise replace it
by a point $x_1$ in that ball with
$\mathcal M(|x_1|)>2\mathcal M(|x_0|)$ and repeat. The total
displacement is bounded by
\[
\sum_{\nu\ge0}\frac{k}{2^\nu\mathcal M(|x_0|)}
=\frac{2k}{\mathcal M(|x_0|)}
<\tfrac12\min\{|x_0|,R-|x_0|\}.
\]
All chosen points remain in a compact subset of the punctured
ball, where $\mathcal M$ is continuous. Infinite iteration
would make $\mathcal M$ unbounded there, so the procedure stops.

After rotation and passage to a sequence, the selected points
have the form $x_j=r_je_1$. Put
$\ell_j=\mathcal M(r_j)^{-1}$. The selection gives
\[
r_j/\ell_j\to\infty,\qquad (R_j-r_j)/\ell_j\to\infty,\qquad
\mathcal M(|x|)\le2\mathcal M(r_j)
\quad(|x-x_j|\le k_j\ell_j),\quad k_j\to\infty.
\]
On $D_j=(B_{R_j}-x_j)/\ell_j$, define
\[
U_j(z)=r_j^{\alpha-\alpha_0}\ell_j^{\alpha_0}
u_j(x_j+\ell_jz),\qquad
V_j(z)=r_j^{\beta-\beta_0}\ell_j^{\beta_0}
v_j(x_j+\ell_jz).
\]
At $z=0$, their normalized powers add to one. For $|z|\le L$,
\[
\frac{|x_j+\ell_jz|}{r_j}=1+O(L\ell_j/r_j)\longrightarrow1,
\]
so substitution into the definition of $\mathcal M$ gives
\[
U_j(0)^{1/\alpha_0}+V_j(0)^{1/\beta_0}=1,\qquad
U_j(z)^{1/\alpha_0}+V_j(z)^{1/\beta_0}\le2+o_L(1).
\]
The scale identities are
\[
p\beta_0-\alpha_0=q\alpha_0-\beta_0=2s,\qquad
p(\beta-\beta_0)-(\alpha-\alpha_0)=a,\qquad
q(\alpha-\alpha_0)-(\beta-\beta_0)=b.
\]
Insert the definitions of $U_j,V_j$ into
\eqref{eq:ball-green-system} and use the scaling of its Green
kernel. Equivalently, the resulting equations are
\[
\begin{cases}
(-\Delta)^sU_j=\mu_j\ell_j^{2s}U_j+
 (|x_j+\ell_jz|/r_j)^aV_j^p,\\
(-\Delta)^sV_j=(|x_j+\ell_jz|/r_j)^bU_j^q.
\end{cases}
\]
They still have nonnegative Green sources on $D_j$.
The distance from zero to $\partial D_j$ tends to infinity,
and
\[
0\le\mu_j\ell_j^{2s}
\le C_0(\ell_j/R_j)^{2s}\longrightarrow0.
\]
The preceding bounds make both sources uniformly bounded on
each fixed ball and both central potential values bounded.
Lemma \ref{lem:green-limits} therefore gives a locally uniformly
convergent subsequence. The limiting sources are $V^p$ and
$U^q$, and the normalization and local bounds give
\[
U(0)^{1/\alpha_0}+V(0)^{1/\beta_0}=1,\qquad
U^{1/\alpha_0}+V^{1/\beta_0}\le2\quad\text{on }\R^N.
\]
Its second conclusion yields
\[
U=I_{2s}(V^p)+C_U,\qquad
V=I_{2s}(U^q)+C_V,\qquad C_U,C_V\ge0.
\]
Because $\ell_j/r_j\to0$, formula \eqref{eq:radial-flattening}
and local equicontinuity show that $U,V$ depend only on $z_1$.
This bounded nonzero pair contradicts Lemma
\ref{lem:profile-exclusion}, proving \eqref{eq:annular-small}.

For $0<r\le R/2$, one has $\min\{r,R-r\}=r$. Substitute this
in \eqref{eq:annular-small} and use the definition of $\mathcal M$:
\[
r^{\alpha/\alpha_0}u(r)^{1/\alpha_0}
+r^{\beta/\beta_0}v(r)^{1/\beta_0}
=r\mathcal M(r)\le C.
\]
Bounding each nonnegative summand separately proves
\eqref{eq:radial-decay-small}.
\end{proof}

\begin{theorem}
\label{thm:radial-compactness}
Assume \eqref{eq:parameters-main} and \eqref{eq:radial-window},
and fix $C_0>0$.
Consider positive radial solutions $(u,v)\in X_R\times X_R$
of \eqref{eq:ball-green-system}, with $\mu\ge0$ and $\mu R^{2s}\le C_0$.
For every $L,c>0$, there is $C_{L,c}$ such that
\[
R\ge4L,\quad u(0)+v(0)\le c
\quad\Longrightarrow\quad \sup_{B_L}(u+v)\le C_{L,c}.
\]
For every fixed $R,c>0$, there is $C_{R,c}$ such that
\[
u(0)+v(0)\le c
\quad\Longrightarrow\quad \sup_{B_R}(u+v)\le C_{R,c}.
\]
The constant $C_{L,c}$ is independent of $R,\mu,u,v$,
and $C_{R,c}$ is independent of $\mu,u,v$.
Both may depend on the equation parameters and $C_0$.
Every sequence with $R_j\to\infty$ and $u_j(0)+v_j(0)\le c$
has a subsequence converging locally uniformly on $\R^N$.
\end{theorem}
\begin{proof}
We first prove the bound on $B_L$, uniformly for $R\ge4L$.
If it fails, choose a sequence with central sum at most $c$ and
\[
M_j=\max_{\overline B_L}(u_j^{1/\alpha}+v_j^{1/\beta})
\longrightarrow\infty,\qquad \ell_j=M_j^{-1}.
\]
Here $\alpha,\beta>0$ by $a,b>-2s$ and $pq>1$.
The maximum cannot occur at zero for large $j$ because the
central sums are bounded. At a maximizing radius $r_j$,
\eqref{eq:radial-decay-small} gives $M_j\le C/r_j$, so
$r_j/\ell_j\le C$. Set
\[
\widetilde U_j(z)=\ell_j^\alpha u_j(\ell_jz),\qquad
\widetilde V_j(z)=\ell_j^\beta v_j(\ell_jz).
\]
For $|z|\le L/\ell_j$ these functions satisfy
\[
\widetilde U_j^{1/\alpha}+\widetilde V_j^{1/\beta}\le1,
\quad
\widetilde U_j(r_je_1/\ell_j)^{1/\alpha}
+\widetilde V_j(r_je_1/\ell_j)^{1/\beta}=1,
\quad
\widetilde U_j(0)+\widetilde V_j(0)\longrightarrow0.
\]
Their Green equations are those of
\eqref{eq:ball-green-system} on $B_{R_j/\ell_j}$ with
$\mu_j$ replaced by $\mu_j\ell_j^{2s}$. On a fixed $B_{2L'}$,
their sources are bounded by
$C_{L'}(1+|z|^a+|z|^b)$, and
\[
\mu_j\ell_j^{2s}\le C_0(\ell_j/R_j)^{2s}\longrightarrow0.
\]
The weighted extension in Lemma
\ref{lem:weighted-local-continuity} therefore supplies local
equicontinuity. After a diagonal subsequence and convergence of
$r_je_1/\ell_j$, the limit is continuous and nonzero.

Apply Fatou's lemma to the two nonnegative Green integrals at
zero, using \eqref{eq:green-exhaustion} on
$B_{R_j/\ell_j}$. Dropping the nonnegative linear term gives
\[
0=\widetilde U(0)\ge
I_{2s}(|z|^a\widetilde V^p)(0),\qquad
0=\widetilde V(0)\ge
I_{2s}(|z|^b\widetilde U^q)(0).
\]
The weights and kernels are positive off zero. Each integral
can vanish only if its continuous nonnegative source component
is identically zero. This contradicts the nonzero normalized
limit and proves the uniform local bound.

Now fix $R$. The first assertion bounds $u,v$ near zero, by
taking $L<R/4$. Lemma \ref{lem:radial-annular} bounds them on
every compact subset of $B_R\setminus\{0\}$. Thus an unbounded
sequence of full-ball maxima must occur at
$x_j=r_je_1$ with $r_j\to R$. Choose $x_j$ to maximize
$u_j^{1/\alpha_0}+v_j^{1/\beta_0}$, where $\alpha_0,\beta_0$
are as in Lemma \ref{lem:radial-annular}. Let its maximum be
$M_j\to\infty$, put $\ell_j=M_j^{-1}$, and define
\[
U_j(z)=\ell_j^{\alpha_0}u_j(x_j+\ell_jz),\qquad
V_j(z)=\ell_j^{\beta_0}v_j(x_j+\ell_jz)
\]
on $D_j=(B_R-x_j)/\ell_j$, with zero extensions. Then
\[
0\le U_j,V_j\le1,\qquad
U_j(0)^{1/\alpha_0}+V_j(0)^{1/\beta_0}=1,
\]
and their equations have coefficients
$\mu_j\ell_j^{2s}\to0$,
$|x_j+\ell_jz|^a\to R^a$,
$|x_j+\ell_jz|^b\to R^b$ on fixed balls away from the rescaled
origin. The zero extensions belong to $H^s(\R^N)$ by
Lemma \ref{lem:ball-maps}.

Put $L_j=(R-r_j)/\ell_j$. If $L_j$ is bounded, the point
$y_j=L_je_1$ is a nearest boundary point of $D_j$. For
$z\in D_j\cap B_8(y_j)$, the original radius lies between
$R-8\ell_j$ and $R$, so both right-hand sides are uniformly
bounded. Lemma \ref{lem:uniform-ball-boundary} gives
\[
U_j(0)+V_j(0)\le C L_j^s.
\]
Normalization excludes $L_j\to0$. Passing to a subsequence,
either $L_j\to\infty$ or $L_j\to L\in(0,\infty)$.

If $L_j\to\infty$, Lemma \ref{lem:green-limits} gives a
bounded nonzero entire limit satisfying
\[
U=R^aI_{2s}(V^p)+C_U,\qquad
V=R^bI_{2s}(U^q)+C_V,\qquad C_U,C_V\ge0.
\]
Since $\ell_j/r_j\to0$, radial flattening makes this limit
one-dimensional. Lemma \ref{lem:profile-exclusion} excludes it.

If $L_j\to L\in(0,\infty)$, the balls $D_j$ converge locally
to $\{z_1<L\}$. Interior estimates and
\eqref{eq:uniform-boundary-holder} give local uniform convergence
of the zero extensions, including at the limiting boundary.
The limit is continuous, bounded, nonzero, and zero on
$\{z_1\ge L\}$. For a compactly supported test function inside
$\{z_1<L\}$, pass to the equations using
$0\le U_j,V_j\le1$ and
$|(-\Delta)^s\phi(z)|\le C_\phi(1+|z|)^{-N-2s}$.
Thus
\[
(-\Delta)^sU=R^aV^p,\qquad
(-\Delta)^sV=R^bU^q
\quad\text{in }\{z_1<L\}.
\]
Again radial flattening makes both components depend only on
$z_1$. The normalization of the fractional Laplacian gives
\[
C(N,s)\int_{\R^{N-1}}
(t^2+|z'|^2)^{-(N+2s)/2}\,dz'
=C(1,s)|t|^{-1-2s}\qquad(t\ne0).
\]
Transverse integration in the distributional equations and
the coordinate change $t=L-z_1$ therefore yield the half-line
system in Lemma \ref{lem:half-line}, with constants $R^a,R^b$.
This is also impossible. The full-ball bound follows.

Finally, let $R_j\to\infty$ and $u_j(0)+v_j(0)\le c$.
For fixed $L$, the first bound controls $u_j,v_j$ on $B_{2L}$
for all sufficiently large $j$. Their nonnegative Green sources
are bounded there by $C_L(1+|x|^a+|x|^b)$, and their central
values are at most $c$. Lemma
\ref{lem:weighted-local-continuity} gives equicontinuity on
$B_L$, with every exponent below
$\min\{1,2s,2s+a,2s+b\}$. Arzel\`a-Ascoli on successively
larger balls and a diagonal choice give the claimed subsequence.
\end{proof}

\Needspace{12\baselineskip}
\subsection{Passage to the whole space}\label{sec:unified-existence}

\begin{theorem}
\label{thm:unified-existence}
Assume \eqref{eq:parameters-main}, \eqref{eq:supercritical},
and \eqref{eq:radial-window}.
System \eqref{eq:system} has a positive classical entire solution.
\end{theorem}
\begin{proof}
For each fixed $R>0$, Proposition \ref{prop:ball-continuum}
gives a connected component $\mathcal C_R$. On this component,
\[
0<\mu<\lambda_1(R)=\lambda_1(1)R^{-2s}.
\]
Hence Theorem \ref{thm:radial-compactness} applies with
$C_0=\lambda_1(1)$. If $u(0)+v(0)$ were bounded on
$\mathcal C_R$, its full-ball estimate would bound
$\|u\|_\infty+\|v\|_\infty$. Lemma \ref{lem:ball-maps} and
\eqref{eq:ball-green-system} would then give
\[
\|u\|_{X_R}
\le C_R\bigl(\lambda_1(R)\|u\|_\infty+\|v\|_\infty^p\bigr)
\le C_R,
\]
contradicting the unboundedness in Proposition
\ref{prop:ball-continuum}. The central sum is therefore
unbounded. It is continuous on $\mathcal C_R$ and tends to
zero at its bifurcation point. Its image is an interval with
infimum zero and no finite upper bound, so it contains one.
Choose a pair on each ball with
\begin{equation}\label{eq:ball-selection}
u_R(0)+v_R(0)=1,\qquad
0<\mu_R<\lambda_1(1)R^{-2s}.
\end{equation}

We next bound the integrals needed to pass to a nonlocal limit.
The local bound in Theorem \ref{thm:radial-compactness} controls
$u_R+v_R$ on $B_1$ for $R\ge4$. On $1\le r\le R/2$,
\eqref{eq:radial-decay-small} and $\alpha,\beta>0$ give
$u_R(r)+v_R(r)\le C$. Thus
\[
\sup_{B_{R/2}}(u_R+v_R)\le C\qquad(R\ge4).
\]
Apply Lemma \ref{lem:inner-central} to the two nonnegative
Green sources $\mu_Ru_R+|x|^av_R^p$ and $|x|^bu_R^q$.
Using \eqref{eq:ball-selection}, obtain
\[
\int_{B_R}(u_R+v_R)\,dx\le C_{N,s}R^N.
\]
For the zero extensions, split the tail at $R/2$. For
$K\ge1$ and $R>\max\{4,2K\}$ this gives
\begin{equation}\label{eq:unified-tail}
\begin{aligned}
\int_{|x|>K}\frac{u_R+v_R}{1+|x|^{N+2s}}\,dx
&\le C\int_K^{R/2}r^{-1-2s}\,dr
+CR^{-N-2s}\int_{B_R}(u_R+v_R)\,dx\\
&\le C(K^{-2s}+R^{-2s})\le CK^{-2s}.
\end{aligned}
\end{equation}
The constants are independent of $R$.

Take $R_j\to\infty$. Theorem \ref{thm:radial-compactness}
gives a locally uniform limit $(u,v)$ with
$u(0)+v(0)=1$. Fatou's lemma in the nonnegative Green
representations, together with \eqref{eq:green-exhaustion},
gives
\[
u\ge I_{2s}(|x|^av^p),\qquad
v\ge I_{2s}(|x|^bu^q).
\]
The central normalization excludes $u\equiv v\equiv0$.
Positivity of the Riesz kernel gives the two alternatives
\[
\begin{aligned}
u\not\equiv0&\ \Longrightarrow\ v\ge I_{2s}(|x|^bu^q)>0
\ \Longrightarrow\ u\ge I_{2s}(|x|^av^p)>0,\\
v\not\equiv0&\ \Longrightarrow\ u\ge I_{2s}(|x|^av^p)>0
\ \Longrightarrow\ v\ge I_{2s}(|x|^bu^q)>0.
\end{aligned}
\] Fatou's lemma in
\eqref{eq:unified-tail} also gives $u,v\in\mathcal L_{2s}$.

Let $\phi\in C_c^\infty(\R^N)$. For large $j$ it is a valid
test function in the ball equations. Its fractional Laplacian
satisfies
\[
|(-\Delta)^s\phi(x)|\le C_\phi(1+|x|)^{-N-2s}.
\]
On $B_K$, local uniform convergence permits passage in
$\int u_{R_j}(-\Delta)^s\phi$. Outside $B_K$,
\eqref{eq:unified-tail} bounds the integral by $C_\phi K^{-2s}$,
uniformly for large $j$, and gives the same bound for the limit.
First let $j\to\infty$ and then $K\to\infty$. This proves
\[
\int_{\R^N}u_{R_j}(-\Delta)^s\phi
\longrightarrow\int_{\R^N}u(-\Delta)^s\phi.
\]
On the support of $\phi$, the nonlinearities are locally
uniformly bounded before multiplication by their weights.
Since $a,b>-N$, dominated convergence gives
\[
\int |x|^av_{R_j}^p\phi\longrightarrow\int |x|^av^p\phi,
\qquad
\int |x|^bu_{R_j}^q\phi\longrightarrow\int |x|^bu^q\phi.
\]
Finally,
$|\mu_{R_j}\int u_{R_j}\phi|
\le\mu_{R_j}\|u_{R_j}\|_{L^\infty(\operatorname{supp}\phi)}
\|\phi\|_1\to0$. Passing to both weak equations proves
\eqref{eq:system} distributionally on all of $\R^N$.

On a compact subset of $\R^N\setminus\{0\}$, the weights are
smooth and $u,v$ have positive minima. Their positive powers
are consequently smooth functions of the unknowns there.
Apply the local estimates
\cite[Corollaries 2.4--2.5]{ros-boundary} first to smooth
convolutions and then pass to their limit. Localizing and
differentiating as in Lemma \ref{lem:regularity}, iteration in
the two equations gives
$u,v\in C^\infty_{\rm loc}(\R^N\setminus\{0\})$ and the
pointwise equations. Together with continuity, positivity,
and $u,v\in\mathcal L_{2s}$, these are precisely the conditions
in Definition \ref{def:class}.
\end{proof}

\begin{proof}[Proof of Theorem \ref{thm:small-exist}]
If $1/2\le s<1$, then $1-2s\le0$ whereas
$1/(p+1)+1/(q+1)>0$. Thus \eqref{eq:radial-window} holds.
Under \eqref{eq:supercritical}, Theorem
\ref{thm:unified-existence} gives a positive classical entire
solution. If \eqref{eq:supercritical} fails, its strict reverse
is \eqref{eq:hyperbola}, and Theorem \ref{thm:main} excludes
such a solution. This proves the equivalence.
\end{proof}
\section{Orders below one half}\label{sec:critical-criterion}

Throughout this section $N\ge2$ and $0<s<1/2$.
We first apply the construction of Section \ref{sec:construction}.
For the two nonexistence assertions, we then prove a common inequality
for integral equations on $\R\times\Sph^{N-1}$ and verify its hypotheses
for each original system. In both applications the original solution is
an arbitrary positive classical entire solution in the sense of
Definition \ref{def:class}.

\begin{proof}[Proof of Theorem \ref{thm:below-half-exist}]
The first inequality in \eqref{eq:sufficient-window-intro} is
\eqref{eq:supercritical}; the second is \eqref{eq:radial-window}.
Together with \eqref{eq:parameters-main}, these are the hypotheses of
Theorem \ref{thm:unified-existence}. That theorem gives a positive
classical entire solution.
\end{proof}

\subsection{A common nonexistence criterion}\label{sec:critical-finite-energy}

\subsubsection{The cylinder equation and reflection}

For $x\ne0$, write
\[
x=e^t\theta,\qquad t=\log|x|,\qquad
\theta\in\Sph^{N-1},\qquad h=\frac{N-2s}{2}.
\]
We use the cylinder $\mathcal C=\R\times\Sph^{N-1}$ with surface
measure $dt\,d\theta$ and the real $L^2$ pairing
$\langle f,g\rangle=\int_{\mathcal C}fg$.
The identity
\[
|e^t\theta-e^\tau\varphi|^2
=2e^{t+\tau}\bigl(\cosh(t-\tau)-\theta\cdot\varphi\bigr)
\]
leads, for $0<\sigma<1/2$, to the kernel and operator
\begin{equation}\label{eq:full-cylinder-kernel}
\begin{aligned}
\mathcal K_\sigma(t,z)
&=c_{N,\sigma}[2(\cosh t-z)]^{-(N-2\sigma)/2},\qquad -1\le z\le1,\\
(\mathcal T_\sigma f)(t,\theta)
&=\int_\R\int_{\Sph^{N-1}}
\mathcal K_\sigma(t-\tau,\theta\cdot\varphi)
 f(\tau,\varphi)\,d\varphi\,d\tau.
\end{aligned}
\end{equation}
The constant $c_{N,\sigma}$ is the Riesz normalization in
\eqref{eq:kernel}. For $\sigma=s$ write $\mathcal K=\mathcal K_s$
and $\mathcal T=\mathcal T_s$.

For example, suppose $u=I_{2s}(|x|^au^p)$ and $N+a=h(p+1)$.
Substitute $y=e^\tau\varphi$, $dy=e^{N\tau}d\tau\,d\varphi$,
and $u(e^\tau\varphi)=e^{-h\tau}U(\tau,\varphi)$ in this identity.
The remaining exponential factor is
$e^{(N+a-h-hp)\tau}=1$. Consequently
\[
U(t,\theta):=e^{ht}u(e^t\theta)
\quad\Longrightarrow\quad U=\mathcal T(U^p),
\qquad
\int_{\mathcal C}U^{p+1}
=\int_{\R^N}|x|^au^{p+1}\,dx.
\]
These identities retain the angular variable; no symmetry is used.
The integral on either side of the last identity is allowed to be
infinite at this stage.

The kernel $\mathcal K$ is positive and even in $t$, and decreases
strictly for $t>0$. Its angular integral
\begin{equation}\label{eq:angular-cylinder-kernel}
J_N(t)=\int_{\Sph^{N-1}}
\mathcal K(t,\theta\cdot\varphi)\,d\varphi
\end{equation}
is independent of $\theta$. For $|t|\le1$, polar coordinates about
$\theta$ compare this integral with
\[
\int_0^1\frac{\rho^{N-2}}{(t^2+\rho^2)^{(N-2s)/2}}\,d\rho
=|t|^{2s-1}\int_0^{1/|t|}
\frac{v^{N-2}}{(1+v^2)^{(N-2s)/2}}\,dv.
\]
The last integral is bounded above and below by positive constants,
because $2s-2<-1$. For $|t|\ge1$, use
$\cosh t-z\asymp e^{|t|}$ uniformly in $z\in[-1,1]$. Thus
\begin{equation}\label{eq:cylinder-kernel-bounds}
J_N(t)\asymp |t|^{2s-1}\quad(t\to0),\qquad
J_N(t)\asymp e^{-h|t|}\quad(|t|\to\infty),
\qquad J_N\in L^1(\R).
\end{equation}
All constants in these kernel bounds depend only on $N,s$.

\begin{lemma}\label{lem:cylinder-reflection}
Let $p>1$. Let $\mathcal L(t,\theta,\varphi)>0$ be even in $t$,
symmetric in $\theta,\varphi$, and strictly decreasing in $t>0$
for every angular pair. Assume
\[
J_{\mathcal L}(t):=\sup_\theta\int_{\Sph^{N-1}}
\mathcal L(t,\theta,\varphi)\,d\varphi\in L^1(\R).
\]
Suppose $W>0$ is smooth on $\mathcal C$ and satisfies
\[
W(t,\theta)=\int_{\mathcal C}
\mathcal L(t-\tau,\theta,\varphi)W(\tau,\varphi)^p
\,d\tau\,d\varphi<\infty.
\]
Assume, uniformly in $\theta$, that
$|W|+|W_t|\le Ce^{\mu t}$ for all sufficiently negative $t$,
where $\mu>0$. If
$W(t_1,\theta_0)>W(t_2,\theta_0)$ for some $t_1<t_2$ and $\theta_0$,
then there is $\lambda_0\in\R$ such that
\[
W(t,\theta)=W(2\lambda_0-t,\theta),\qquad
|W(t,\theta)|+|W_t(t,\theta)|\le C'e^{-\mu|t|}.
\]
In particular $0<\int_{\mathcal C}W^{p+1}<\infty$.
\end{lemma}
\begin{proof}
For $t<\lambda$, set $W_\lambda(t,\theta)=W(2\lambda-t,\theta)$.
Split each potential at $\tau=\lambda$ and replace
$\tau$ by $2\lambda-\tau$ in its right half. Subtraction gives
\begin{equation}\label{eq:full-reflection-identity}
\begin{split}
W_\lambda(t,\theta)-W(t,\theta)
={}&\int_{\tau<\lambda}\int_{\Sph^{N-1}}
\bigl[\mathcal L(t-\tau,\theta,\varphi)
-\mathcal L(2\lambda-t-\tau,\theta,\varphi)\bigr]\\
&\hspace{12mm}\cdot
\bigl(W_\lambda(\tau,\varphi)^p-W(\tau,\varphi)^p\bigr)
\,d\varphi\,d\tau.
\end{split}
\end{equation}
For $t,\tau<\lambda$, one has
$|t-\tau|<2\lambda-t-\tau$, so the bracket is positive.
Let $D_\lambda=(W-W_\lambda)_+$ on $t<\lambda$. On
$\{D_\lambda>0\}$, the mean-value theorem gives
$W^p-W_\lambda^p\le pW^{p-1}D_\lambda$.
Discarding the nonnegative contributions in
\eqref{eq:full-reflection-identity}, we obtain
\[
D_\lambda(t,\theta)
\le p\int_{\tau<\lambda}\int_{\Sph^{N-1}}
\mathcal L(t-\tau,\theta,\varphi)
W(\tau,\varphi)^{p-1}D_\lambda(\tau,\varphi)\,d\varphi\,d\tau.
\]
Here $D_\lambda\le W$ is bounded on the left half-cylinder. Hence
\[
\|D_\lambda\|_\infty
\le p\|J_{\mathcal L}\|_1
\sup_{\tau<\lambda,\varphi}W(\tau,\varphi)^{p-1}
\|D_\lambda\|_\infty.
\]
The assumed decay makes the coefficient less than one for all
sufficiently negative $\lambda$. Thus $W_\lambda\ge W$ there.

Let $\lambda_0$ be the supremum of the planes up to which this
inequality holds for every preceding plane. If $\lambda_0<\infty$,
continuity gives $W_{\lambda_0}\ge W$ on $t<\lambda_0$.
If equality does not hold identically, the positive bracket in
\eqref{eq:full-reflection-identity} gives
$W_{\lambda_0}>W$ at every point of this half-cylinder.
We show that this strict inequality permits an advance.
Choose $A$ so large that
\[
p\|J_{\mathcal L}\|_1\sup_{\tau<-A,\varphi}W^{p-1}<\tfrac14.
\]
For small $\varepsilon>0$, strict positivity has a positive minimum
on $[-A,\lambda_0-\varepsilon]\times\Sph^{N-1}$ and persists when
$\lambda$ is slightly larger than $\lambda_0$.
Negative differences can therefore occur only for $\tau<-A$ or
$\lambda_0-\varepsilon<\tau<\lambda$.
The latter strip has length at most $2\varepsilon$.
Since $J_{\mathcal L}$ is even and nonincreasing in $|t|$, its
contribution to the preceding supremum estimate is at most
\[
p\sup_{\substack{|\tau-\lambda_0|\le1\\\varphi\in\Sph^{N-1}}}
W(\tau,\varphi)^{p-1}
\int_{-2\varepsilon}^{2\varepsilon}J_{\mathcal L}(t)\,dt.
\]
This is less than $1/4$ after decreasing $\varepsilon$, by
$J_{\mathcal L}\in L^1$. The sum of the tail and strip coefficients
is less than one, so $D_\lambda=0$. This contradicts the definition
of $\lambda_0$. A finite stopping plane therefore gives symmetry.

If $\lambda_0=+\infty$, take $\lambda=(t_1+t_2)/2$ in the reflected
inequality to obtain $W(t_1,\theta)\le W(t_2,\theta)$ for every
$t_1<t_2$ and every $\theta$. This contradicts the stated decrease.
Thus $\lambda_0$ is finite. Reflecting the given left-end bounds
across $\lambda_0$ yields the asserted right-end bounds, and their
integration gives the finite positive energy.
\end{proof}

\subsubsection{The dilation form}\label{sec:critical-quadratic-form}

For the single kernel define
\begin{equation}\label{eq:full-Q-kernel}
\mathcal Q(t,z)=-t\partial_t\mathcal K(t,z)-(1-2s)\mathcal K(t,z).
\end{equation}
We also write $\mathcal Q$ for the associated integral operator.
Direct differentiation of \eqref{eq:full-cylinder-kernel} gives
\[
\frac{|t\partial_t\mathcal K(t,z)|}{\mathcal K(t,z)}
=h\frac{|t\sinh t|}{\cosh t-z}\le C(1+|t|).
\]
Together with \eqref{eq:cylinder-kernel-bounds}, this proves
\[
\int_\R\sup_\theta\int_{\Sph^{N-1}}
(\mathcal K+|t\partial_t\mathcal K|+|\mathcal Q|)
(t,\theta\cdot\varphi)\,d\varphi\,dt<\infty,
\qquad \varepsilon J_N(\varepsilon)\le C\varepsilon^{2s}\to0.
\]
Symmetry and the integral bound imply that these kernels define
bounded self-adjoint operators on $L^2(\mathcal C)$.

We prove the sign separately on the constant angular component
and on the other spherical harmonics.

\subsubsection{The constant angular component}

On functions independent of $\theta$, the operator $\mathcal Q$
acts by convolution with $-tJ_N'(t)-(1-2s)J_N(t)$.
The following calculation establishes its pointwise sign.

\begin{lemma}\label{lem:constant-mode}
For every $t\ne0$,
\begin{equation}\label{eq:constant-mode-strict}
-tJ_N'(t)>(1-2s)J_N(t).
\end{equation}
\end{lemma}
\begin{proof}
By evenness it suffices to take $t>0$. Set
$z=\sinh(t/2)$ and $h_+=(N-2+2s)/2>0$.
Spherical polar coordinates give
\[
J_N(t)=C\int_0^\pi
\frac{\sin^{N-2}\vartheta}
{[2(\cosh t-\cos\vartheta)]^h}\,d\vartheta,
\qquad h=\frac{N-2s}{2}.
\]
For $r=\tan(\vartheta/2)$, the two factors in this integral are
\[
2(\cosh t-\cos\vartheta)
=\frac{4[z^2+(1+z^2)r^2]}{1+r^2},\qquad
\sin^{N-2}\vartheta\,d\vartheta
=\frac{2^{N-1}r^{N-2}}{(1+r^2)^{N-1}}\,dr.
\]
Substituting them and using $N-1-h=h_+$ yields
\[
J_N(t)=C\int_0^\infty
\frac{r^{N-2}\,dr}
{(1+r^2)^{h_+}[z^2+(1+z^2)r^2]^h}.
\]
Now set $r=z(1+z^2)^{-1/2}y$. Since $N-1-2h=2s-1$,
\begin{equation}\label{eq:constant-mode-integral}
z^{1-2s}J_N(2\operatorname{arsinh}z)
=C(1+z^2)^{-(N-1)/2}
\int_0^\infty
\frac{y^{N-2}\,dy}
{(1+y^2)^h(1+\frac{z^2}{1+z^2}y^2)^{h_+}}.
\end{equation}
The integral remains finite at $z=0$: its integrand is
$O(y^{N-2})$ at zero and $O(y^{2s-2})$ at infinity.
For $z>0$ each factor on the right is nonincreasing, because
$h_+>0$ and $z^2/(1+z^2)$ is increasing.
Differentiation is justified on every compact subinterval of
$(0,\infty)$ by an integrable majorant. Therefore
\[
0\ge\frac{d}{dt}\bigl[(\sinh(t/2))^{1-2s}J_N(t)\bigr]
=(\sinh(t/2))^{1-2s}
\left[J_N'(t)+\frac{1-2s}{2}\coth(t/2)J_N(t)\right].
\]
Multiply by $-t$ and use $(t/2)\coth(t/2)>1$ to obtain
\eqref{eq:constant-mode-strict}. Here $1-2s>0$ supplies the strict
comparison with $(1-2s)J_N(t)$.
\end{proof}

\subsubsection{The other angular components}

Let $Y$ be a spherical harmonic of degree $\ell\ge0$:
$-\Delta_{\Sph^{N-1}}Y=\ell(\ell+N-2)Y$.
Rotation invariance of the kernel preserves this harmonic subspace.
Using the time Fourier transform
$\widehat\phi(\xi)=\int_\R e^{-it\xi}\phi(t)\,dt$, define
$m_{\ell,s}$ by
\[
\widehat{\mathcal T(\phi Y)}(\xi,\theta)
=m_{\ell,s}(\xi)\widehat\phi(\xi)Y(\theta).
\]
Put $N_\ell=N+2\ell$. In the cylinder formula of
Ao-Chan-DelaTorre-Fontelos-Gonz\'alez-Wei
\cite[Proposition 2.1 and (4.3)]{ao-ode}, take $Q_0=0$ and use
\[
\sqrt{\left(\frac{N-2}{2}\right)^2+\ell(\ell+N-2)}
=\ell+\frac{N-2}{2}.
\]
This gives the multiplier $\Theta_{\ell,s}$ of the conjugated
fractional Laplacian. Its inverse is $\mathcal T_s$: indeed, for a
smooth compactly supported cylinder function $F$, polar substitution
in the normalized Riesz potential gives
\[
u=I_{2s}\!\left(|x|^{-(N+2s)/2}
F(\log|x|,x/|x|)\right),\qquad
e^{ht}u(e^t\theta)=\mathcal T_sF(t,\theta).
\]
Applying $(-\Delta)^sI_{2s}=\mathrm{Id}$ shows that the two
multipliers have product one. Thus
\begin{equation}\label{eq:angular-multipliers}
\Theta_{\ell,s}(\xi)=2^{2s}
\frac{|\Gamma((N_\ell+2s)/4+i\xi/2)|^2}
{|\Gamma((N_\ell-2s)/4+i\xi/2)|^2},
\qquad m_{\ell,s}(\xi)=\Theta_{\ell,s}(\xi)^{-1}>0.
\end{equation}
The cited symbol is even in $\xi$, so changing its convention
$t=-\log r$ to $t=\log r$ leaves this formula unchanged.

Let $q_{\ell,s}$ be the multiplier of $\mathcal Q$ on the same
subspace. The Fourier rules
$\widehat{\partial_t k}=i\xi\widehat k$ and
$\widehat{tk}=i\partial_\xi\widehat k$ give
\begin{equation}\label{eq:Q-multiplier}
q_{\ell,s}(\xi)
=2s\,m_{\ell,s}(\xi)+\xi m_{\ell,s}'(\xi)
=m_{\ell,s}(\xi)
\left(2s-\xi\frac{\Theta_{\ell,s}'(\xi)}{\Theta_{\ell,s}(\xi)}\right).
\end{equation}
The differentiation rules apply distributionally at the kernel
singularity. The absolute kernel bounds proved above identify the
result with the bounded integral operator $\mathcal Q$.
For $\ell\ge1$, the condition $N\ge2$ gives $N_\ell\ge4$.
In view of $m_{\ell,s}>0$, the required sign is equivalent to
\begin{equation}\label{eq:multiplier-target}
2s-\xi\frac{\Theta_{\ell,s}'(\xi)}{\Theta_{\ell,s}(\xi)}>0
\qquad(\xi\in\R).
\end{equation}

To evaluate this expression, define, for $c\ge1/2$,
\[
R_{c,s}(t)=e^{-ct}\frac{\sinh(st/2)}{\sinh(t/2)},\qquad
g_{c,s}(t)=-R_{c,s}'(t)\quad(t>0).
\]
For the chosen degree, set
\[
y=\xi/2,\qquad c=(N_\ell-2)/4\ge1/2,\qquad
B_\ell=(N_\ell-2s)/4.
\]
Subtracting the integral formulas for the digamma function
$\Psi=\Gamma'/\Gamma$ gives
\[
\begin{aligned}
\Psi(B_\ell+s+iy)-\Psi(B_\ell+iy)
&=\int_0^\infty e^{-(B_\ell+iy)t}
\frac{1-e^{-st}}{1-e^{-t}}\,dt\\
&=\int_0^\infty R_{c,s}(t)e^{-iyt}\,dt.
\end{aligned}
\]
In the second equality we used
$B_\ell+(s-1)/2=c$. Differentiate the logarithm of
\eqref{eq:angular-multipliers} and take the imaginary part of this
formula. This gives
\[
\frac{\Theta_{\ell,s}'(\xi)}{\Theta_{\ell,s}(\xi)}
=\int_0^\infty R_{c,s}(t)\sin(\xi t/2)\,dt,
\qquad
\xi\frac{\Theta_{\ell,s}'}{\Theta_{\ell,s}}
=2y\int_0^\infty R_{c,s}(t)\sin(yt)\,dt.
\]
Since $R_{c,s}(0)=s$ and $R_{c,s}$ decays at infinity,
integration by parts yields
\begin{equation}\label{eq:cosine-multiplier-bridge}
2s-\xi\frac{\Theta_{\ell,s}'(\xi)}{\Theta_{\ell,s}(\xi)}
=2\int_0^\infty g_{c,s}(t)\cos(\xi t/2)\,dt.
\end{equation}
We now establish the sign of this cosine transform.

\begin{lemma}\label{lem:cosine-density}
For $0<s<1/2$ and $c\ge1/2$, the functions just defined satisfy
$g_{c,s}>0$, $g_{c,s}'<0$, and $g_{c,s}''>0$ on $(0,\infty)$.
The functions $g_{c,s},g_{c,s}',g_{c,s}''$ decay exponentially and
have finite limits at zero. Moreover,
\begin{equation}\label{eq:positive-cosine}
\int_0^\infty g_{c,s}(t)\cos(yt)\,dt>0
\qquad(y\in\R).
\end{equation}
\end{lemma}
\begin{proof}
We first record the implication of the derivative signs. If
$g>0$, $g''>0$, and $g,g',g''$ have the stated endpoint bounds,
then two integrations by parts give, for $y\ne0$,
\[
\begin{aligned}
\int_0^\infty g(t)\cos(yt)\,dt
&=-\frac1y\int_0^\infty g'(t)\sin(yt)\,dt\\
&=\frac1{y^2}\int_0^\infty g''(t)(1-\cos(yt))\,dt>0.
\end{aligned}
\]
The last integral is strictly positive because $g''>0$ and
$1-\cos(yt)>0$ on a set of positive measure. For $y=0$, use $g>0$.
It remains to verify these derivative signs for $g=-R_{c,s}'$.

Set
\[
b_s(t)=\frac12\coth(t/2)-\frac s2\coth(st/2),
\qquad D_s=\frac{1-s^2}{12},\qquad R=R_{c,s}.
\]
Logarithmic differentiation gives
\[
\frac{R'}R=-(c+b_s),\qquad
\frac{R''}R=(c+b_s)^2-b_s',\qquad
-\frac{R'''}R=(c+b_s)^3-3(c+b_s)b_s'+b_s''.
\]
We will substitute the following three bounds into these identities:
\begin{equation}\label{eq:hyperbolic-signs}
b_s\ge0,\qquad 0<b_s'\le D_s,\qquad b_s''+2b_sb_s'\ge0.
\end{equation}
For $x=t/2>0$, convexity of $\sinh$ and $0<s<1$ give
$\sinh(sx)<s\sinh x$. Hence
\[
b_s'(t)=\frac{s^2\operatorname{csch}^2(sx)
-\operatorname{csch}^2x}{4}>0,
\qquad b_s(0)=0,
\]
which proves $b_s\ge0$.
To obtain the upper bound on $b_s'$, put
$T(x)=x^3\cosh x/\sinh^3x$. The sign of $-T'(x)/T(x)$ is that of
\[
H(x)=x(2+\cosh2x)-\tfrac32\sinh2x.
\]
Indeed $H'=2J$, where
$J=1-\cosh2x+x\sinh2x$ satisfies
$J''=4x\sinh2x>0$ and $H(0)=J(0)=J'(0)=0$.
Successive integration gives $J'>0$, $J>0$, $H>0$, and therefore
$T'<0$. Substitution in the derivative of $b_s'$ gives
\[
b_s''(t)=\frac{T(x)-T(sx)}{4x^3}<0.
\]
Thus $b_s'$ decreases from
$b_s'(0)=(1-s^2)/12=D_s$, proving the middle inequality in
\eqref{eq:hyperbolic-signs}.

For the last inequality in \eqref{eq:hyperbolic-signs}, direct
substitution of the derivatives of $\coth$ gives
\begin{equation}\label{eq:hyperbolic-bracket}
b_s''+2b_sb_s'
=\frac s4\operatorname{csch}^2(sx)\coth(sx)
\left[\frac{\sinh^2(sx)}{\sinh^2x}
+s\frac{\tanh(sx)}{\tanh x}-2s^2\right].
\end{equation}
The prefactor is positive. Multiply the bracket by
$4\sinh^2x\cosh(sx)>0$ and use the product formulas for hyperbolic
functions. Its sign becomes the sign of
\begin{equation}\label{eq:hyperbolic-combination}
\begin{split}
F_s(x)={}&\cosh(3sx)-(1-4s^2)\cosh(sx)\\
&+s(1-2s)\cosh((2+s)x)
-s(1+2s)\cosh((2-s)x).
\end{split}
\end{equation}
The condition $s<1/2$ makes the four numbers
\[
\lambda_1=s^2<\lambda_2=9s^2<\lambda_3=(2-s)^2<\lambda_4=(2+s)^2
\]
distinct and increasing. Set
\[
c_1=-(1-4s^2),\quad c_2=1,\quad
c_3=-s(1+2s),\quad c_4=s(1-2s)>0.
\]
Expansion gives $\sum_{j=1}^4c_j\lambda_j^k=0$ for $k=0,1,2$.
The three moment equations have a one-dimensional nullspace, so
\[
c_j=\frac{\lambda_*}{\prod_{i\ne j}(\lambda_j-\lambda_i)},
\qquad \lambda_*=c_4\prod_{i<4}(\lambda_4-\lambda_i)>0.
\]
For $\Phi(\lambda)=\cosh(x\sqrt\lambda)$, its power series gives
\[
\Phi(\lambda)=\sum_{k=0}^\infty\frac{x^{2k}\lambda^k}{(2k)!},
\qquad \Phi'''(\lambda)>0\quad(x,\lambda>0).
\]
The third divided-difference formula therefore gives some
$\zeta\in(\lambda_1,\lambda_4)$ such that
\[
F_s(x)=\lambda_*\sum_{j=1}^4
\frac{\Phi(\lambda_j)}{\prod_{i\ne j}(\lambda_j-\lambda_i)}
=\frac{\lambda_*}{3!}\Phi'''(\zeta)>0.
\]
This proves the last inequality of \eqref{eq:hyperbolic-signs}.

We can now substitute all three bounds into the derivatives of $R$.
First, $R'/R=-(c+b_s)<0$ and
$R''/R\ge c^2-D_s>0$.
Next, $b_s''\ge-2b_sb_s'$ gives
\[
-\frac{R'''}R
\ge(c+b_s)^3-(3c+5b_s)b_s'
\ge(c+b_s)^3-(3c+5b_s)D_s.
\]
For fixed $c$, the last polynomial in $b_s\ge0$ is positive at zero,
since
$c(c^2-3D_s)\ge cs^2/4>0$.
Its derivative is also positive:
$3(c+b_s)^2-5D_s\ge1/3$.
Thus $R'''<0$. The identities $g=-R'$, $g'=-R''$, and $g''=-R'''$
prove the asserted signs. The formula for $R$ and its Taylor
expansion at zero give the required endpoint bounds. The two
integrations by parts at the start of the proof now yield
\eqref{eq:positive-cosine}.
\end{proof}

\begin{lemma}\label{lem:nonconstant-modes}
For every spherical harmonic degree $\ell\ge1$, the multiplier
$q_{\ell,s}(\xi)$ is strictly positive for every $\xi\in\R$.
\end{lemma}
\begin{proof}
For $\ell\ge1$ one has $c=(N+2\ell-2)/4\ge1/2$.
Apply Lemma \ref{lem:cosine-density} with this value of $c$ and with
$y=\xi/2$. Equation \eqref{eq:cosine-multiplier-bridge} then gives
\eqref{eq:multiplier-target}. Multiplication by the positive factor
$m_{\ell,s}(\xi)$ in \eqref{eq:Q-multiplier} proves the claim.
\end{proof}

\subsubsection{Scaling and the threshold}

\begin{lemma}\label{lem:cylinder-identity}
Let $\mathcal P$ be a bounded self-adjoint operator on
$L^2(\mathcal C)$, diagonal in the time Fourier transform and
spherical harmonics. Denote its real multipliers by $M_\ell(\xi)$,
and assume
\[
M_\ell\in C^1(\R),\qquad
\sup_{\ell,\xi}\bigl(|M_\ell(\xi)|+|\xi M_\ell'(\xi)|\bigr)<\infty.
\]
Let $p>1$ and let a positive smooth function $W$ satisfy
$W=\mathcal P(W^p)$ and
$|W|+|W_t|\le Ce^{-\mu|t|}$, uniformly in the angular variable,
for some $\mu>0$. Put $f=W^p$ and $E=\int_{\mathcal C}W^{p+1}$.
For an orthonormal spherical harmonic basis $Y_{\ell,k}$, set
$f_{\ell,k}(t)=\int_{\Sph^{N-1}}f(t,\theta)Y_{\ell,k}(\theta)\,d\theta$.
Then
\begin{equation}\label{eq:full-cylinder-virial}
\frac{2E}{p+1}
=\frac1{2\pi}\sum_{\ell,k}\int_\R
\bigl(M_\ell(\xi)+\xi M_\ell'(\xi)\bigr)
|\widehat f_{\ell,k}(\xi)|^2\,d\xi.
\end{equation}
\end{lemma}
\begin{proof}
The decay gives $f,t f_t\in L^2(\mathcal C)$.
For $\lambda>0$ set $f_\lambda(t,\theta)=f(\lambda t,\theta)$
and $B(\lambda)=\langle f_\lambda,\mathcal P f_\lambda\rangle$.
Differentiation in $L^2$, self-adjointness, and
$\mathcal P f=W$ give
\[
B'(1)=2\langle t f_t,W\rangle
=\frac{2p}{p+1}\int_{\mathcal C}t\,\partial_t(W^{p+1})
=-\frac{2p}{p+1}E.
\]
On the other hand, $\widehat{f_\lambda}(\xi)=
\lambda^{-1}\widehat f(\xi/\lambda)$ and Plancherel's formula give
\[
B(\lambda)=\frac1{2\pi\lambda}
\sum_{\ell,k}\int_\R M_\ell(\lambda\xi)
|\widehat f_{\ell,k}(\xi)|^2\,d\xi.
\]
The uniform multiplier bound permits differentiation under the
sum and integral near $\lambda=1$. Hence
\[
B'(1)=\frac1{2\pi}\sum_{\ell,k}\int_\R
\bigl(\xi M_\ell'(\xi)-M_\ell(\xi)\bigr)
|\widehat f_{\ell,k}(\xi)|^2\,d\xi.
\]
Equate the two expressions and use
$E=\langle f,\mathcal P f\rangle$ to obtain
\eqref{eq:full-cylinder-virial}.
\end{proof}

For each $0<\sigma<1/2$, write
\[
\mathcal Q_\sigma=-t\partial_t\mathcal K_\sigma
-(1-2\sigma)\mathcal K_\sigma.
\]
Lemmas \ref{lem:constant-mode} and \ref{lem:nonconstant-modes} apply
to this kernel with $s=\sigma$. We combine these signs for products
of the operators $\mathcal T_\sigma$.

\begin{theorem}\label{thm:cylinder-product-threshold}
Let $N\ge2$, $m\ge1$, $s_j>0$, and
$\sigma=\sum_{j=1}^m s_j<1/2$.
Let $\mathcal P=\mathcal T_{s_1}\cdots\mathcal T_{s_m}$ and denote
its integral kernel by $\mathcal L$.
Suppose $p>1$, $W>0$ is smooth on $\mathcal C$, and
$W=\mathcal P(W^p)$ with a finite integral at every point.
Assume there are $C,\mu>0$ such that, uniformly in $\theta$,
$|W|+|W_t|\le Ce^{\mu t}$ for all sufficiently negative $t$.
If $W(t_1,\theta_0)>W(t_2,\theta_0)$ for some
$t_1<t_2$ and $\theta_0\in\Sph^{N-1}$, then
\[
\frac2{p+1}>1-2\sigma.
\]
\end{theorem}
\begin{proof}
We first record the bounds for $\mathcal L$ itself.
Put $\rho^2=(t-\tau)^2+|\theta-\varphi|^2$.
The bounds for the factor kernels imply
\begin{equation}\label{eq:cylinder-product-kernel-bounds}
\begin{aligned}
\mathcal L(t-\tau,\theta,\varphi)&\le C\rho^{2\sigma-N}
&& (0<\rho<1),\\
\mathcal L(t-\tau,\theta,\varphi)&\le Ce^{-\kappa|t-\tau|}
&& (|t-\tau|\ge1).
\end{aligned}
\end{equation}
Indeed, for two factors and $r=\rho<1$, split the intermediate
point into distance less than $r/3$ from each endpoint and the
remaining annuli. In local cylinder coordinates their contributions
are at most
\[
\begin{aligned}
&Cr^{2s_2-N}\int_0^{r/3}v^{2s_1-1}\,dv
+Cr^{2s_1-N}\int_0^{r/3}v^{2s_2-1}\,dv\\
&\qquad+C\int_{r/3}^1v^{2(s_1+s_2)-N-1}\,dv+C
\le Cr^{2(s_1+s_2)-N}.
\end{aligned}
\]
Every partial sum of the orders is below $1/2$, so this calculation
iterates. For large time separation, at least one factor in a
convolution has time separation at least half as large.
Its exponential bound and the integrability of the other factor
give exponential decay, with a smaller positive $\kappa$ if needed.
No differentiated product-kernel estimate is required.

Each factor kernel is positive, even in time, and strictly
decreasing for positive time. For two even, integrable, strictly
decreasing time kernels with distinct angular arguments,
\[
(g*h)'(t)=\int_0^\infty g'(r)
\bigl[h(t-r)-h(t+r)\bigr]\,dr<0\qquad(t>0).
\]
Here $|t-r|<t+r$. Apply this comparison with the intermediate
angular variables fixed outside their measure-zero singular sets,
then integrate in those variables. This proves strict decrease of
$\mathcal L$. The factors commute because their multipliers are
scalar on every spherical harmonic space. Their product is
self-adjoint, so $\mathcal L$ is also symmetric in the angular
arguments. If $J_{N,s_j}$ is the angular integral of
$\mathcal K_{s_j}$, rotation invariance and Tonelli's theorem give
\[
J_{\mathcal L}=J_{N,s_1}*\cdots*J_{N,s_m},\qquad
\|J_{\mathcal L}\|_1=\prod_{j=1}^m\|J_{N,s_j}\|_1<\infty.
\]
The assumed decrease of $W$ therefore allows
Lemma \ref{lem:cylinder-reflection} to be applied:
\[
|W|+|W_t|\le C'e^{-\mu|t|},\qquad
E:=\int_{\mathcal C}W^{p+1}\in(0,\infty).
\]

The multiplier of $\mathcal P$ on degree $\ell$ is
$M_\ell=\prod_{j=1}^m m_{\ell,s_j}$. The integral bounds for
$\mathcal K_{s_j}$ and $\mathcal Q_{s_j}$, together with
\eqref{eq:Q-multiplier}, imply
\[
\sup_{\ell,\xi}\bigl(|m_{\ell,s_j}(\xi)|
+|\xi m_{\ell,s_j}'(\xi)|\bigr)<\infty.
\]
The product rule gives the same bound for $M_\ell$, so
Lemma \ref{lem:cylinder-identity} applies to $f=W^p$.
Subtract $(1-2\sigma)E$ from \eqref{eq:full-cylinder-virial}:
\begin{equation}\label{eq:cylinder-spectral-balance}
\left(\frac2{p+1}-(1-2\sigma)\right)E
=\frac1{2\pi}\sum_{\ell,k}\int_\R
(2\sigma M_\ell+\xi M_\ell')
|\widehat f_{\ell,k}|^2\,d\xi.
\end{equation}
Since $\sigma=\sum_j s_j$,
\[
2\sigma M_\ell+\xi M_\ell'
=\sum_{j=1}^m q_{\ell,s_j}\prod_{i\ne j}m_{\ell,s_i}.
\]
For $\ell\ge1$, every factor on the right is positive by
Lemma \ref{lem:nonconstant-modes}.

For the constant angular component, Lemma \ref{lem:constant-mode}
gives
\[
\overline Q_{s_j}(t)
:=-tJ_{N,s_j}'(t)-(1-2s_j)J_{N,s_j}(t)>0\qquad(t\ne0).
\]
The inverse Fourier transform of $2\sigma M_0+\xi M_0'$ is
\[
Q_0=\sum_{j=1}^m\overline Q_{s_j}*
\mathop{*}_{i\ne j}J_{N,s_i}>0\quad\hbox{a.e.}
\]
For $m=1$, this means $Q_0=\overline Q_{s_1}$.
Choose $Y_{0,1}=|\Sph^{N-1}|^{-1/2}$.
Then $f_0=f_{0,1}\ge0$ and $f_0\not\equiv0$, because $f=W^p>0$.
Consequently the right side of \eqref{eq:cylinder-spectral-balance}
equals
\begin{equation}\label{eq:positive-cylinder-form}
\begin{aligned}
&\int_\R\int_\R f_0(t)Q_0(t-\tau)f_0(\tau)\,d\tau\,dt\\
&\quad+\frac1{2\pi}\sum_{\ell\ge1,k}\int_\R
\left(\sum_{j=1}^m q_{\ell,s_j}\prod_{i\ne j}m_{\ell,s_i}\right)
|\widehat f_{\ell,k}|^2\,d\xi>0.
\end{aligned}
\end{equation}
The first term is strictly positive and all remaining terms are
nonnegative. The uniform multiplier bounds justify the sum.
Divide \eqref{eq:cylinder-spectral-balance} by $E>0$ to conclude.
\end{proof}

\subsection{The critical diagonal}\label{sec:proof-critical-diagonal}

Assume in this subsection that
\begin{equation}\label{eq:critical-parameters}
N\ge2,\qquad 0<s<\frac12,\qquad a=b>-2s,
\qquad p=q=p_S(a)=\frac{N+2s+2a}{N-2s}.
\end{equation}
Set $a_*=2s(N-1)/(1-2s)$. Substituting the displayed value of $p$
gives
\begin{equation}\label{eq:critical-target}
\frac2{p+1}-(1-2s)
=\frac{(1-2s)(a_*-a)}{N+a}.
\end{equation}
Thus Theorem \ref{thm:below-half-exist} applies when $-2s<a<a_*$.
For $a\ge a_*$ we reduce the system to a scalar equation, exclude
a monotone cylinder profile, and apply Theorem
\ref{thm:cylinder-product-threshold}.
The scalar reduction holds for the whole interval $a>-2s$.
The instability argument below is used only for $a\ge a_*>0$.

\begin{lemma}\label{lem:critical-equality}
Every positive classical entire solution $(u,v)$ under
\eqref{eq:critical-parameters} satisfies $u=v$.
\end{lemma}
\begin{proof}
Proposition \ref{thm:representation} gives
\[
u=I_{2s}(|x|^av^p),\qquad v=I_{2s}(|x|^au^p),
\qquad (-\Delta)^s(u-v)=|x|^a(v^p-u^p)
\]
on the whole space, with the last equality interpreted
distributionally. In particular Lemma \ref{lem:origin} supplies
that equality across the origin.
For $w=(u-v)_+$, smooth convex approximation of the positive part
gives the distributional inequality
\[
(-\Delta)^s w
\le\mathbf1_{\{u>v\}}|x|^a(v^p-u^p)\le0.
\]
Use the Green operator of Section \ref{sec:construction} to set
\[
H_Ru=u-\mathcal G_{s,R}(|x|^av^p).
\]
The equation for $u$ gives $(-\Delta)^sH_Ru=0$ in $B_R$;
the zero exterior Green condition gives $H_Ru=u$ outside $B_R$.
Moreover $H_Ru\ge0$, because $G_{s,R}\le K_s$ and
$u=I_{2s}(|x|^av^p)$.
Since $w\le u=H_Ru$ outside $B_R$, the maximum principle gives
$w\le H_Ru$ inside as well. One may apply this comparison to the
smooth convex approximations first: local boundedness and
$u,v\in\mathcal L_{2s}$ permit the limit in the distributional
inequality and in its exterior integrals.
Finally, Green exhaustion \eqref{eq:green-exhaustion} and the exact
potential identity give, for each fixed $x$,
\[
0\le w(x)\le H_Ru(x)
=u(x)-\mathcal G_{s,R}(|x|^av^p)(x)\downarrow0.
\]
Thus $u\le v$. Interchanging the two components proves $v\le u$.
\end{proof}

After this reduction the equation and potential identity are
\begin{equation}\label{eq:critical-scalar}
(-\Delta)^su=|x|^au^p,\qquad
u=I_{2s}(|x|^au^p).
\end{equation}
As in Definition \ref{def:class}, the equation holds pointwise away
from zero and distributionally on $\R^N$; the potential identity
holds at every point. With $h=(N-2s)/2$, the parameter equalities are
\begin{equation}\label{eq:critical-relations}
N+a=h(p+1),\qquad a+2s=h(p-1).
\end{equation}
The next two statements show that this scalar solution cannot
satisfy the stability inequality. This fact will be used to exclude
monotonicity in the logarithmic radius.

\begin{corollary}\label{lem:critical-masses}
Assume \eqref{eq:critical-parameters} and $a>0$. For every positive
classical entire solution $u$ of \eqref{eq:critical-scalar} and every
$R>0$,
\begin{equation}\label{eq:critical-mass-bounds}
\int_{B_R}|x|^au^p\le CR^h,\qquad
\int_{B_R}u\le CR^{N-h},\qquad
\int_{|x|>R}\frac{u(x)}{|x|^{N+2s}}\,dx\le CR^{-h-2s}.
\end{equation}
The constant $C$ depends only on $N,s,a$.
\end{corollary}
\begin{proof}
Apply the earlier estimates to the pair $(u,u)$.
Substituting $a=b$ and $p=q$ in \eqref{eq:scales}, and using
$a+2s=h(p-1)$ from \eqref{eq:critical-relations}, gives
$\alpha=\beta=h$.
Proposition \ref{prop:source} then gives the first estimate with
exponent $N-2s-h=h$. Proposition \ref{prop:inputs} gives the second
with exponent $N-h$. In \eqref{eq:general-tail}, choose $\tau=s$;
its exponent is $-h-2s$, which proves the third estimate.
\end{proof}

\begin{lemma}\label{lem:critical-unstable}
Under the assumptions of Corollary \ref{lem:critical-masses}, the
inequality
\begin{equation}\label{eq:critical-stability}
p\int_{\R^N}|x|^au^{p-1}\phi^2\,dx
\le\|\phi\|_{\dot H^s(\R^N)}^2
\qquad\text{for every }\phi\in C_c^\infty(\R^N)
\end{equation}
cannot hold. The norm has the normalization in
\eqref{eq:common-energy-form}--\eqref{eq:common-homogeneous-space}.
\end{lemma}
\begin{proof}
Suppose \eqref{eq:critical-stability} holds. Let $\eta$ be a smooth
compactly supported cutoff. Test \eqref{eq:critical-scalar} with
$u\eta^2$ and \eqref{eq:critical-stability} with $u\eta$.
These tests follow by approximation: $a>0$ and local boundedness of
$u$ make the source locally bounded, which gives local $H^s$
regularity; the exterior integrals are finite because
$u\in\mathcal L_{2s}$.
In the difference of the two energy forms use the identity
\[
(u_x\eta_x-u_y\eta_y)^2
-(u_x-u_y)(u_x\eta_x^2-u_y\eta_y^2)
=u_xu_y(\eta_x-\eta_y)^2,
\]
where $u_x=u(x)$ and $\eta_x=\eta(x)$. This gives
\begin{equation}\label{eq:critical-cutoff-product}
(p-1)\int_{\R^N}|x|^au^{p+1}\eta^2\,dx
\le\frac{C(N,s)}2\iint_{\R^N\times\R^N}
\frac{u(x)u(y)(\eta(x)-\eta(y))^2}{|x-y|^{N+2s}}\,dx\,dy.
\end{equation}
The product $u(x)u(y)$ is retained when one point lies outside a
fixed ball; only a weighted $L^1$ bound is available there.

Define $\mathcal E(r)=\int_{B_r}|x|^au^{p+1}$.
For $1\le r<R\le2$, put $d=R-r$ and choose $\eta=1$ on $B_r$,
$\eta=0$ outside $B_{(r+R)/2}$, and $|\nabla\eta|\le C/d$.
For the part of \eqref{eq:critical-cutoff-product} on $B_R^2$, use
$2u(x)u(y)\le u(x)^2+u(y)^2$ and
\[
\sup_x\int_{\R^N}
\frac{(\eta(x)-\eta(y))^2}{|x-y|^{N+2s}}\,dy
\le C\int_0^\infty\min\{1,t^2/d^2\}t^{-1-2s}\,dt
\le Cd^{-2s}.
\]
This part is bounded by $Cd^{-2s}\int_{B_R}u^2$.
In each remaining nonzero pair, one point is in
$B_{(r+R)/2}$ and the other is outside $B_R$.
For the latter point in $B_4$, use $|x-y|\ge d/2$.
For $|y|>4$, use $|x-y|\ge|y|/2$.
The last two estimates of \eqref{eq:critical-mass-bounds}, at fixed
radii $2$ and $4$, give
\[
\begin{aligned}
&\int_{B_{(r+R)/2}}\int_{\R^N\setminus B_R}
\frac{u(x)u(y)}{|x-y|^{N+2s}}\,dy\,dx\\
&\qquad\le C\left(\int_{B_2}u\right)
\left[d^{-N-2s}\int_{B_4}u+
\int_{|y|>4}\frac{u(y)}{|y|^{N+2s}}\,dy\right]
\le Cd^{-N-2s}.
\end{aligned}
\]
Consequently \eqref{eq:critical-cutoff-product} implies
\begin{equation}\label{eq:critical-hole-estimate}
\mathcal E(r)\le C(R-r)^{-N-2s}
\left(1+\int_{B_R}u^2\right).
\end{equation}

Apply H\"older's inequality with exponents $(p+1)/2$ and
$(p+1)/(p-1)$:
\[
\int_{B_R}u^2
\le\mathcal E(R)^{2/(p+1)}
\left(\int_{B_R}|x|^{-2a/(p-1)}\,dx\right)^{(p-1)/(p+1)}
\le C\mathcal E(R)^{2/(p+1)}.
\]
The last weight is integrable, since
$2a/(p-1)=2h-4s/(p-1)<N$ by \eqref{eq:critical-relations}.
Set $\sigma_0=2/(p+1)<1$ and
$\nu=(N+2s)/(1-\sigma_0)$.
Substitute this estimate into \eqref{eq:critical-hole-estimate},
and use Young's inequality, to obtain
\begin{equation}\label{eq:critical-iteration}
\mathcal E(r)\le\varepsilon\mathcal E(R)
+C_\varepsilon(R-r)^{-\nu}.
\end{equation}
Take $r_j=2-2^{-j}$ and choose $\varepsilon>0$ so that
$\varepsilon2^\nu<1$. Iterating from $r_0=1$ gives
\[
\mathcal E(1)
\le\varepsilon^m\mathcal E(r_m)
+C_\varepsilon2^\nu\sum_{j=0}^{m-1}(\varepsilon2^\nu)^j
\le\varepsilon^m\mathcal E(2)+C.
\]
For each fixed solution $\mathcal E(2)<\infty$.
Letting $m\to\infty$ proves $\mathcal E(1)\le C$, with a constant
independent of the solution.

For every $L>0$, the function $u_L(x)=L^hu(Lx)$ satisfies both
\eqref{eq:critical-scalar} and \eqref{eq:critical-stability}, because
$a+2s=h(p-1)$.
Apply the uniform bound on $B_1$ to $u_L$. The other critical
identity, $N+a=h(p+1)$, cancels the scaling factor and yields
\[
\int_{B_L}|x|^au^{p+1}\,dx
=\int_{B_1}|x|^a u_L^{p+1}\,dx\le C.
\]
Letting $L\to\infty$ gives
$E:=\int_{\R^N}|x|^au^{p+1}\in(0,\infty)$.
This is a consequence of the assumed stability inequality.

To use $u$ itself in stability, set
$f=|x|^au^p$, $f_j=f\mathbf1_{B_j}$, and $u_j=I_{2s}f_j$.
Each $f_j$ is bounded and compactly supported. The Riesz energy
identity and positivity give
\[
\begin{aligned}
\|u_j\|_{\dot H^s}^2&=\int f_jI_{2s}f_j\uparrow E,\\
\|u_j-u_k\|_{\dot H^s}^2
&=\int(f_j-f_k)I_{2s}(f_j-f_k)\\
&\le\int_{\R^N\setminus B_k}fu\longrightarrow0\qquad(j>k).
\end{aligned}
\]
Thus $(u_j)$ converges in $\dot H^s$.
The Sobolev embedding and the pointwise convergence $u_j\uparrow u$
identify the limit with $u$, so
$u\in\dot H^s$ and $\|u\|_{\dot H^s}^2=E$.
Choose compact smooth approximations to $u$ in this norm and a
subsequence converging almost everywhere. Apply
\eqref{eq:critical-stability} to these approximations and use Fatou's
lemma on its left side. The result is $pE\le E$, which contradicts
$p>1$ and $E>0$.
\end{proof}

\begin{proof}[Proof of Theorem \ref{thm:critical-criterion}]
Suppose first that $a\ge a_*$ and a positive classical entire pair
exists. Since $a_*>0$, Lemmas \ref{lem:critical-equality} and
\ref{lem:critical-unstable} both apply: the pair reduces to a scalar
solution $u$, and this solution does not satisfy
\eqref{eq:critical-stability}.
Define $U(t,\theta)=e^{ht}u(e^t\theta)$.
Substitution in the potential formula, using
$N+a=h(p+1)$, gives
\begin{equation}\label{eq:cylinder-scalar}
U=\mathcal T_s(U^p)\qquad\hbox{on }\mathcal C.
\end{equation}
Continuity at zero and \eqref{eq:origin-gradient} imply, uniformly
in $\theta$,
\begin{equation}\label{eq:left-cylinder-asymptotic}
\begin{aligned}
U(t,\theta)&=e^{ht}(u(0)+o(1)),\\
U_t(t,\theta)&=e^{ht}\bigl(hu(e^t\theta)+e^t\theta\cdot\nabla u(e^t\theta)\bigr),\\
|U(t,\theta)|+|U_t(t,\theta)|&\le Ce^{ht}\qquad(t\to-\infty).
\end{aligned}
\end{equation}
Smoothness on the cylinder follows from Lemma \ref{lem:regularity}.

We must verify the decrease hypothesis of Theorem
\ref{thm:cylinder-product-threshold}.
Suppose instead that $U$ is nondecreasing in $t$ for every $\theta$.
For $\rho>1$, set $u_\rho(x)=\rho^hu(\rho x)$ and
$z_\rho=u_\rho-u$. Monotonicity gives $z_\rho\ge0$ for $x\ne0$,
while continuity gives
$z_\rho(0)=(\rho^h-1)u(0)>0$.
The scaling relation $a+2s=h(p-1)$ and the exact potentials give
\[
z_\rho=I_{2s}\bigl(|x|^a(u_\rho^p-u^p)\bigr)>0.
\]
Strict positivity follows because the nonnegative source is
positive on a set of positive measure near zero.
Subtracting the two equations gives
\[
(-\Delta)^sz_\rho=c_\rho z_\rho,
\qquad
c_\rho(x)=p|x|^a\int_0^1
\bigl((1-\eta)u(x)+\eta u_\rho(x)\bigr)^{p-1}\,d\eta
\ge p|x|^au(x)^{p-1}.
\]
For $\phi\in C_c^\infty(\R^N)$, test this equation with
$\phi^2/z_\rho$. The algebraic identity
\[
(\phi_x-\phi_y)^2
-(z_x-z_y)\left(\frac{\phi_x^2}{z_x}-\frac{\phi_y^2}{z_y}\right)
=\left(\sqrt{\frac{z_y}{z_x}}\phi_x
-\sqrt{\frac{z_x}{z_y}}\phi_y\right)^2\ge0,
\qquad z_x=z_\rho(x),
\]
combined with \eqref{eq:common-energy-form}, yields
\[
p\int|x|^au^{p-1}\phi^2
\le\int c_\rho\phi^2
\le\|\phi\|_{\dot H^s}^2.
\]
To justify the test, use the positive minimum of $z_\rho$ on the
support of $\phi$, its local $H^s$ regularity, and
$z_\rho\in\mathcal L_{2s}$ for the exterior integrals.
This is precisely \eqref{eq:critical-stability}, contradicting
Lemma \ref{lem:critical-unstable}.

Therefore $U(t_1,\theta_0)>U(t_2,\theta_0)$ for some $t_1<t_2$.
Equation \eqref{eq:cylinder-scalar}, the bounds
\eqref{eq:left-cylinder-asymptotic}, and this decrease verify every
hypothesis of Theorem \ref{thm:cylinder-product-threshold} with
$m=1$, $s_1=s$, and $W=U$. It gives
\[
\frac2{p+1}>1-2s.
\]
By \eqref{eq:critical-target}, this is equivalent to $a<a_*$,
contrary to $a\ge a_*$. The strict inequality includes nonexistence
at $a=a_*$. The finite cylinder energy used in the threshold theorem
was obtained from reflection after monotonicity was excluded.

Conversely, let $-2s<a<a_*$ and $p=q=p_S(a)$.
Then $p>1$, and direct substitution gives
\[
\frac{N+a}{p+1}+\frac{N+a}{p+1}=N-2s,
\qquad \frac1{p+1}+\frac1{p+1}>1-2s
\]
by \eqref{eq:critical-target}. These are
\eqref{eq:sufficient-window-intro}. Theorem
\ref{thm:below-half-exist} therefore gives a positive classical
entire solution, completing both directions.
\end{proof}

\subsection{A fixed critical curve with unequal weights}\label{sec:unequal-critical-curve}

Fix the parameters in \eqref{eq:fixed-weights-main}:
\[
N=2,\qquad 0<s<\frac14,\qquad -2s<b<0,
\qquad a=\frac{4s-b}{1-4s},\qquad
p_0=\frac{1+4s}{1-4s}>1.
\]
These weights remain fixed throughout the subsection.
We first exclude $(p,q)=(p_0,1)$ and then prove existence for
$p>p_0$ on the same curve \eqref{eq:fixed-critical-curve}.

Suppose a positive classical entire pair exists at $(p_0,1)$.
Lemma \ref{lem:system-potentials} gives the finite identities
\begin{equation}\label{eq:unequal-critical-potentials}
\begin{cases}
u=I_{2s}(|x|^av^{p_0}),\\
v=I_{2s}(|x|^bu),
\end{cases}
\qquad x\in\R^2.
\end{equation}
In particular $u(0),v(0)>0$.
Substitution in \eqref{eq:scales} gives
\begin{equation}\label{eq:unequal-critical-scales}
\begin{aligned}
\alpha&=1+\frac b2>0,&
\beta&=1-2s-\frac b2>0,&
\delta&=s+\frac b2\in(0,s),\\
\alpha+\beta&=2-2s,&
2+b&=2\alpha,&
2+a&=(p_0+1)\beta.
\end{aligned}
\end{equation}
Identify $\Sph^1$ with $\mathbb T=\R/(2\pi\mathbb Z)$, using
$e_\theta=(\cos\theta,\sin\theta)$ and the usual angle measure.
On $\mathcal C=\R\times\mathbb T$ set
\[
U(t,\theta)=e^{\alpha t}u(e^te_\theta),\qquad
V(t,\theta)=e^{\beta t}v(e^te_\theta).
\]
For $0<\sigma<1/2$, formula \eqref{eq:full-cylinder-kernel} becomes
\[
\mathcal K_\sigma(t,\vartheta)
=c_{2,\sigma}[2(\cosh t-\cos\vartheta)]^{-(1-\sigma)},
\qquad \mathcal T_\sigma f=\mathcal K_\sigma*f,
\]
where $*$ is convolution on $\R\times\mathbb T$ with measure
$dt\,d\theta$, and $\vartheta=\theta-\varphi$.
Define
$\mathcal A_{\pm\delta}f=(e^{\pm\delta t}\mathcal K_s)*f$.
The two exponents in the first polar substitution are
\[
\alpha-(1-s)=\delta,
\qquad 2+a-(1-s)-p_0\beta=-\delta;
\]
in the second substitution they are
\[
\beta-(1-s)=-\delta,
\qquad 2+b-(1-s)-\alpha=\delta.
\]
Thus \eqref{eq:unequal-critical-potentials} becomes
\begin{equation}\label{eq:unequal-cylinder-system}
\begin{cases}
U=\mathcal A_\delta(V^{p_0}),\\
V=\mathcal A_{-\delta}U,
\end{cases}
\qquad (t,\theta)\in\mathcal C.
\end{equation}
The next lemma identifies the product of these two weighted kernels.

\begin{lemma}\label{lem:weighted-cylinder-factorization}
With the preceding parameters, define
\[
s_1=s+\delta=2s+\frac b2,
\qquad s_2=s-\delta=-\frac b2,
\qquad L=\mathcal K_{s_1}*\mathcal K_{s_2}.
\]
Then $s_1,s_2>0$, $s_1+s_2=2s<1/2$, and
\begin{equation}\label{eq:unequal-kernel-factorization}
\mathcal A_{-\delta}\mathcal A_\delta
=\mathcal T_{s_1}\mathcal T_{s_2}.
\end{equation}
Consequently the function $V$ above satisfies
\begin{equation}\label{eq:unequal-cylinder-scalar}
V=L*V^{p_0}<\infty\qquad\hbox{at every point of }\mathcal C.
\end{equation}
\end{lemma}
\begin{proof}
The inequalities $0<\delta<s$ give $s_1,s_2>0$;
the restriction $s<1/4$ gives $s_1+s_2=2s<1/2$.
For the angular Fourier mode $k\in\mathbb Z$, use
$N=2$, $\ell=|k|$, and $s=\sigma$ in
\eqref{eq:angular-multipliers}. Its analytic continuation is
\[
m_{k,\sigma}(\zeta)=2^{-2\sigma}
\frac{\Gamma((1+|k|-\sigma+i\zeta)/2)
\Gamma((1+|k|-\sigma-i\zeta)/2)}
{\Gamma((1+|k|+\sigma+i\zeta)/2)
\Gamma((1+|k|+\sigma-i\zeta)/2)}.
\]
The kernel decays as $e^{-(1-\sigma)|t|}$, so its transform is
analytic for $|\operatorname{Im}\zeta|<1-\sigma$.
Multiplication by $e^{\pm\delta t}$ replaces the real Fourier
variable $\xi$ by $\xi\pm i\delta$.
These shifts lie in that strip because $\delta<s<1/4$.
In the product $m_{k,s}(\xi-i\delta)m_{k,s}(\xi+i\delta)$,
the four numerator Gamma factors have parameters
$1+|k|-(s+\delta)\pm i\xi$ and
$1+|k|-(s-\delta)\pm i\xi$, divided by two.
The four denominator factors group in the same way with positive
orders. Therefore
\[
m_{k,s}(\xi-i\delta)m_{k,s}(\xi+i\delta)
=m_{k,s+\delta}(\xi)m_{k,s-\delta}(\xi).
\]
Both normalizing factors equal $2^{-4s}$.
All four kernels are in $L^1(\mathcal C)$, so uniqueness of the
Fourier transform in time and the Fourier series in angle proves
\eqref{eq:unequal-kernel-factorization} as a kernel identity.
Apply it first to nonnegative bounded sources with compact support.
For $V^{p_0}$, pass to increasing truncations and use Tonelli's
theorem. Substitution of the first equation of
\eqref{eq:unequal-cylinder-system} into the second then gives
\eqref{eq:unequal-cylinder-scalar}, with finiteness inherited from $V$.
\end{proof}

Continuity of $v$ at zero and \eqref{eq:origin-gradient} give, uniformly in $\theta$,
\begin{equation}\label{eq:unequal-left-decay}
\begin{aligned}
V(t,\theta)&=e^{\beta t}(v(0)+o(1)),\\
V_t(t,\theta)&=e^{\beta t}\bigl(\beta v(e^te_\theta)
+e^te_\theta\cdot\nabla v(e^te_\theta)\bigr),\\
|V(t,\theta)|+|V_t(t,\theta)|&\le Ce^{\beta t}\qquad(t\to-\infty).
\end{aligned}
\end{equation}
Lemma \ref{lem:regularity} also gives smoothness of $V$ on the
cylinder. To exclude monotonicity, we need an estimate for the source
$f:=V^{p_0}$ on unit time strips.
Proposition \ref{prop:source} applies with $q=1$; using
$2-2s-\alpha=\beta$ in its source exponent gives
\[
\int_{B_R}|x|^av^{p_0}\,dx\le CR^\beta\qquad(R>0).
\]
Since $p_0\beta-2=a-\beta$, polar substitution yields, for every
$j\in\mathbb Z$,
\begin{equation}\label{eq:unequal-strip-mass}
\begin{aligned}
\int_{[j,j+1]\times\mathbb T}f
&=\int_{e^j<|x|<e^{j+1}}|x|^{-\beta}|x|^av^{p_0}\,dx\\
&\le e^{-j\beta}\int_{B_{e^{j+1}}}|x|^av^{p_0}\,dx
\le Ce^\beta,\qquad
M:=\sup_{j\in\mathbb Z}\int_{[j,j+1]\times\mathbb T}f<\infty.
\end{aligned}
\end{equation}
In the following proof write $z=(t,\theta)$, $w=(\tau,\varphi)$,
$dz=dt\,d\theta$, and $dw=d\tau\,d\varphi$.

\begin{lemma}\label{lem:unequal-critical-nonmonotone}
For the positive classical entire pair assumed above, its rescaled
component $V$ is not nondecreasing in $t$ for every angle.
\end{lemma}
\begin{proof}
Assume the contrary. For $\ell>0$ set
$V_\ell(t,\theta)=V(t+\ell,\theta)$ and $Z=V_\ell-V\ge0$.
Equation \eqref{eq:unequal-left-decay} gives
$Z(t,\theta)=e^{\beta t}((e^{\beta\ell}-1)v(0)+o(1))$, so
$Z\not\equiv0$.
Subtract the original equation from the translated version of
\eqref{eq:unequal-cylinder-scalar} and use
$p_0>1$. Positivity of $L$ and convexity of the power give
\[
Z=L*(V_\ell^{p_0}-V^{p_0})>0,
\qquad Z\ge p_0L*(V^{p_0-1}Z).
\]
Define the symmetric kernel and positive function
\[
B(z,w)=p_0V(z)^{(p_0-1)/2}L(z-w)V(w)^{(p_0-1)/2},
\qquad H(z)=V(z)^{(p_0-1)/2}Z(z).
\]
Multiplying the preceding inequality by $V(z)^{(p_0-1)/2}$ gives
$\int_{\mathcal C}B(z,w)H(w)\,dw\le H(z)$.
For real $g\in C_c^\infty(\mathcal C)$, multiply
\[
2g(z)g(w)\le g(z)^2\frac{H(w)}{H(z)}
+g(w)^2\frac{H(z)}{H(w)}
\]
by $B(z,w)/2$ and integrate. Symmetry and $BH\le H$ imply
\[
\iint_{\mathcal C^2}B(z,w)g(z)g(w)\,dz\,dw
\le\int_{\mathcal C}\frac{g(z)^2}{H(z)}(BH)(z)\,dz
\le\int_{\mathcal C}g^2.
\]
The positive minimum of $H$ on the compact support and $BH\le H$
justify all integrals in this comparison.

Choose $0\le\eta_R\le1$ smooth on $\R$, equal to one on $[-R,R]$,
zero outside $[-2R,2R]$, and satisfying $|\eta_R'|\le C/R$,
where $R\ge1$. Define
\[
E_R=\int_{\mathcal C}V^{p_0+1}\eta_R^2,
\qquad
D_R=\iint_{\mathcal C^2}L(z-w)f(z)f(w)
(\eta_R(t)-\eta_R(\tau))^2\,dz\,dw.
\]
Here $E_R<\infty$ follows from smoothness on a compact strip.
Also $(\eta_R(t)-\eta_R(\tau))^2\le2\eta_R(t)^2+2\eta_R(\tau)^2$
and $L*f=V$ give $D_R\le4E_R<\infty$.
Use $g=V^{(p_0+1)/2}\eta_R$ in the preceding quadratic inequality.
Since $V^{(p_0-1)/2}g=f\eta_R$, it becomes
\[
p_0\langle f\eta_R,L*(f\eta_R)\rangle\le E_R.
\]
Expanding the square defining $D_R$ and substituting $L*f=V$ gives
\[
D_R=2E_R-2\langle f\eta_R,L*(f\eta_R)\rangle.
\]
Eliminating the pairing from these two formulas yields
\begin{equation}\label{eq:unequal-cutoff-energy}
(p_0-1)E_R\le\frac{p_0}{2}D_R.
\end{equation}

We estimate $D_R$ using only the strip bound
\eqref{eq:unequal-strip-mass}.
Partition $\mathcal C$ into $C_j=[j,j+1]\times\mathbb T$.
For the product kernel, \eqref{eq:cylinder-product-kernel-bounds}
with $N=2$ and $\sigma=s_1+s_2=2s$ gives near the diagonal
\[
|t-\tau|^2L(z-w)
\le C\rho^{4s}\le C,
\qquad
\rho^2=(t-\tau)^2+|e_\theta-e_\varphi|^2.
\]
The same bound holds on the remaining compact set with
$|t-\tau|\le3$, away from the diagonal.
Thus, if $|j-k|\le2$, the Lipschitz bound for $\eta_R$ gives
\[
L(z-w)(\eta_R(t)-\eta_R(\tau))^2\le C/R^2
\qquad(z\in C_j,\ w\in C_k).
\]
If $|j-k|\ge3$, then
$|t-\tau|\ge|j-k|-1$ and $|t-\tau|\le|j-k|+1$.
The exponential kernel bound instead gives
\[
L(z-w)(\eta_R(t)-\eta_R(\tau))^2
\le\frac C{R^2}e^{-\kappa|j-k|}(|j-k|+1)^2.
\]
After increasing $C$, this last estimate includes $|j-k|\le2$.
If the difference of cutoffs is nonzero, at least one strip meets
$[-2R,2R]\times\mathbb T$.
Let $\mathcal J_R=\{j:[j,j+1]\cap[-2R,2R]\ne\varnothing\}$;
then $|\mathcal J_R|\le4R+4$.
For each pair of strips, \eqref{eq:unequal-strip-mass} gives
$\int_{C_j}f\int_{C_k}f\le M^2$. Summing the preceding estimates
therefore proves
\begin{equation}\label{eq:unequal-cutoff-error}
D_R\le\frac{CM^2}{R^2}
\sum_{j\in\mathcal J_R}\sum_{k\in\mathbb Z}
 e^{-\kappa|j-k|}(|j-k|+1)^2
\le\frac{CM^2}{R}.
\end{equation}
Finally, $\eta_R=1$ on $[-1,1]$ and
\eqref{eq:unequal-cutoff-energy}--\eqref{eq:unequal-cutoff-error} give
\[
0<(p_0-1)\int_{[-1,1]\times\mathbb T}V^{p_0+1}
\le(p_0-1)E_R\le CM^2/R\longrightarrow0.
\]
This contradiction excludes the assumed monotonicity.
\end{proof}

\begin{proof}[Proof of Theorem \ref{thm:fixed-critical-curve}]
At $(p,q)=(p_0,1)$, the assumed positive entire pair gives the finite
equation \eqref{eq:unequal-cylinder-scalar} by Lemma
\ref{lem:weighted-cylinder-factorization}.
The left-end estimate \eqref{eq:unequal-left-decay} has exponent
$\beta>0$, and Lemma \ref{lem:unequal-critical-nonmonotone} supplies
the required decrease.
Apply Theorem \ref{thm:cylinder-product-threshold} with
\[
m=2,\qquad s_1=2s+b/2,\qquad s_2=-b/2,
\qquad \sigma=s_1+s_2=2s<1/2,\qquad W=V.
\]
The resulting necessary inequality is
\[
\frac2{p_0+1}>1-2\sigma=1-4s.
\]
But the definition $p_0=(1+4s)/(1-4s)$ gives
$2/(p_0+1)=1-4s$. This contradiction proves nonexistence at that
point, without an energy or symmetry assumption on the original pair.

For existence, keep $N,s,a,b$ fixed and let $p>p_0$ on
\eqref{eq:fixed-critical-curve}. Solving that curve for $q$ gives
\[
q+1=\frac{2+b}{2-2s-(2+a)/(p+1)}.
\]
The denominator increases from $(2+b)/2$ at $p=p_0$ to $2-2s$.
Because $-2s<b<0$, it follows that
\[
0<\frac{b+2s}{2-2s}<q<1.
\]
The same curve also gives
\[
\frac1{p+1}+\frac1{q+1}
=\frac{2-2s}{2+b}-\frac{a-b}{(2+b)(p+1)}.
\]
Here $a-b>0$, so the right-hand side increases strictly with $p$.
At $p=p_0$, $q=1$, its value is $1-2s$.
Its limit as $p\to\infty$ is $(2-2s)/(2+b)<1$, since $b>-2s$.
Thus
\[
1-2s<\frac1{p+1}+\frac1{q+1}<1.
\]
The upper inequality is equivalent to $pq>1$; the lower is
\eqref{eq:radial-window}. The fixed critical curve supplies equality
in \eqref{eq:supercritical}, while the chosen weights satisfy
$a,b>-2s$. All hypotheses of Theorem \ref{thm:below-half-exist}
therefore hold. It gives a positive classical entire solution at
every such point of this same fixed curve.
\end{proof}

\section*{Acknowledgements}
\addcontentsline{toc}{section}{Acknowledgements}
We used AI assistance in parts of the proof of the subcritical
fractional H\'enon-Lane-Emden conjecture. We
discovered new nonexistence phenomena on the critical hyperbola
for $0<s<1/2$ and proved further existence and nonexistence results.
\begingroup
\raggedright

\endgroup
\end{document}